\documentclass{article}

\usepackage[letterpaper,margin=1.2in]{geometry}

\usepackage{amsmath,amssymb,amsthm}
\usepackage{enumerate}
\usepackage[retainorgcmds]{IEEEtrantools}
\usepackage{hyperref}
\usepackage{mathtools,pifont}
\usepackage{tikz-cd}
\usepackage{tikz}
\usepackage{mathrsfs}
\usepackage{xcolor}

\newtheorem{theorem}{Theorem}[section]
\newtheorem{observation}[theorem]{Observation}
\newtheorem{lemma}[theorem]{Lemma}
\newtheorem{proposition}[theorem]{Proposition}
\newtheorem{corollary}[theorem]{Corollary}
\newtheorem{question}[theorem]{Question}

\newtheorem*{theorem*}{Theorem}

\numberwithin{equation}{section}

\theoremstyle{definition}
\newtheorem{definition}[theorem]{Definition}

\theoremstyle{remark}
\newtheorem{remark}[theorem]{Remark}
\newtheorem{example}[theorem]{Example}

\newcommand\R{\mathbb{R}}
\newcommand\C{\mathbb{C}}

\newcommand\Z{\mathbb{Z}}
\newcommand\N{\mathbb{N}}

\newcommand\E{\mathbb{E}}

\newcommand{\cA}{\mathcal{A}}
\newcommand{\cB}{\mathcal{B}}

\newcommand{\cD}{\mathcal{D}}

\newcommand{\cF}{\mathcal{F}}

\newcommand{\cK}{\mathcal{K}}
\newcommand{\cL}{\mathcal{L}}
\newcommand{\cM}{\mathcal{M}}
\newcommand{\cN}{\mathcal{N}}
\newcommand{\cP}{\mathcal{P}}
\newcommand{\cQ}{\mathcal{Q}}
\newcommand{\cR}{\mathcal{R}}
\newcommand{\cS}{\mathcal{S}}
\newcommand{\cU}{\mathcal{U}}
\newcommand{\cV}{\mathcal{V}}
\newcommand{\cO}{\mathcal{O}}

\newcommand{\bA}{\mathbf{A}}
\newcommand{\bD}{\mathbf{D}}
\newcommand{\bS}{\mathbf{S}}
\newcommand{\bT}{\rT}
\newcommand{\bW}{\mathbf{W}}
\newcommand{\bX}{\mathbf{X}}
\newcommand{\bY}{\mathbf{Y}}
\newcommand{\bZ}{\mathbf{Z}}

\newcommand{\rT}{\mathrm{T}}

\DeclareMathOperator{\re}{Re}
\DeclareMathOperator{\im}{Im}

\DeclareMathOperator{\tr}{tr}

\DeclareMathOperator{\sa}{sa}

\DeclareMathOperator{\qf}{qf}
\DeclareMathOperator{\tp}{tp}
\DeclareMathOperator{\full}{full}

\DeclareMathOperator{\Th}{Th}
\DeclareMathOperator{\diam}{diam}

\DeclareMathOperator{\vol}{vol}

\DeclarePairedDelimiter{\norm}{\lVert}{\rVert}
\DeclarePairedDelimiter{\ip}{\langle}{\rangle}

\newcommand{\forkindep}[1][]{%
	\mathrel{
		\mathop{
			\vcenter{
				\hbox{\oalign{\noalign{\kern-.3ex}\hfil$\vert$\hfil\cr
						\noalign{\kern-.7ex}
						$\smile$\cr\noalign{\kern-.3ex}}}
			}
		}\displaylimits_{#1}
	}
}

\title{Free probability and model theory of tracial $\mathrm{W}^*$-algebras}
\author{David Jekel}

\begin{document}
	
	\maketitle
	
	\begin{abstract}
	The notion of a $*$-law or $*$-distribution in free probability is also known as the \emph{quantifier-free type} in Farah, Hart, and Sherman's model theoretic framework for tracial von Neumann algebras.  However, the \emph{full type} can also be considered an analog of a classical probability distribution (indeed, Ben Yaacov showed that in the classical setting, atomless probability spaces admit quantifier elimination and hence there is no difference between the full type and the quantifier-free type). We therefore develop a notion of Voiculescu's free microstates entropy for a full type, and we show that if $\bX$ is a $d$-tuple in $\cM$ with $\chi^{\cU}(\bX:\cM) > -\infty$ for a given ultrafilter $\cU$, then there exists an embedding $\iota$ of $\cM$ into $\cQ = \prod_{n \to \cU} M_n(\C)$ with $\chi(\iota(X): \cQ) = \chi(X:\cM)$; in particular, such an embedding will satisfy $\iota(X)' \cap \cQ = \C$ by the results of Voiculescu.  Furthermore, we sketch some open problems and challenges for developing model-theoretic versions of free independence and free Gibbs laws.
	\end{abstract}
	
	\section{Introduction}
	
	This article aims to show the actual and possible applications of a model-theoretic framework to problems in free probability.  In particular, we discuss an analog of Voiculescu's free entropy for types rather than simply non-commutative $*$-laws.  We use this to deduce that, for instance, a free group $\mathrm{W}^*$-algebra admits an embedding to any given matrix ultraproduct $\cQ = \prod_{n \to \cU} M_n(\C)$ which has trivial relative commutant (also known as an \emph{irreducible} embedding).  This corollary is closely related to the work of Voiculescu in \cite{VoiculescuFE3}, but we in fact prove the stronger result that there exists an embedding such that the entropy of the canonical generators in the presence of $\cQ$ is finite.  The author has also studied the model-theoretic version of $1$-bounded entropy in \cite{JekelCoveringEntropy}.
	
	The final sections of this article discuss some open questions about defining a model-theoretic notion of (free) independence in non-commutative probability theory, which is intimately connected to questions about free entropy and random matrix theory.  Furthermore, we explain open problems pertaining to optimization problems in free probability for which model theory may provide a way forward.
	
	\subsection{Non-commutative probability spaces and laws; free independence}
	
	For this article, a \emph{tracial $\mathrm{W}^*$-algebra} is a von Neumann algebra $M$ together with a faithful normal tracial state $\tau: M \to \C$.  We refer to Adrian Ioana's article in this book for background.   We will often denote the pair $(M,\tau)$ by a single calligraphic letter $\cM$.  This fits conveniently with the model-theoretic convention of using $\cM$ for a structure or for a model of some theory.
	
	Free probability theory, developed in large part by Voiculescu (see e.g.\ \cite{Voiculescu1986}, \cite{VDN1992}), studies free products of operator algebras from a probabilistic viewpoint. Tracial $\mathrm{W}^*$-algebras have long been seen as an analog of probability spaces since every commutative tracial $\mathrm{W}^*$-algebra is isomorphic to an $L^\infty$ space with the trace given by the expectation (see e.g.\ \cite{Sakai1971}).  The elements of a tracial $\mathrm{W}^*$-algebra are thus non-commutative analogs of bounded complex random variables.  Now Cartesian products of groups lead to tensor products of tracial $\mathrm{W}^*$-algebras, and tensor products of commutative tracial $\mathrm{W}^*$-algebras produce independent random variables.  In the same way, free products of groups lead to free products of tracial $\mathrm{W}^*$-algebras, which produces \emph{freely independent} random variables.  For background on free products, see \cite[\S 5]{AGZ2009} or \cite{VDN1992} or \cite[\S 5.3]{AP2017}.
	
	The parallel between classical and free independence (and in general between probability spaces and tracial $\mathrm{W}^*$-algebras) motivated numerous constructions in free probability.  First, \emph{non-commutative $*$-laws} are an analog of a (compactly supported) multivariable probability distribution in the non-commutative setting.  The law or distribution of a $d$-tuple $(X_1,\dots,X_d)$ of complex random variables is the unique measure satisfying
	\[
	\int_{\C^d} f\,d\mu = E[f(X_1,\dots,X_d)]
	\]
	for all appropriate test functions $f$; in the compactly supported case, we can take $f$ to be polynomials in $x_1$, \dots, $x_d$ and $\overline{x}_1$, \dots, $\overline{x}_d$ since a compactly supported probability measure $\mu$ on $\C^d$ is uniquely determined by its \emph{$*$-moments}
	\[
	\int_{\C^d} x_1^{k_1} \dots x_d^{k_d} \overline{x}_1^{\ell_1} \dots \overline{x}_d^{\ell_d} \,d\mu(x_1,\dots,x_d)
	\]
	for $k_1$, \dots, $k_d$ and $\ell_1$, \dots, $\ell_d \in \N_0$.  In the non-commutative case, we \emph{define} a non-commutative $*$-law through its non-commutative $*$-moments.  Instead of the space of classical polynomials $\C[t_1,\dots,t_d,\overline{t}_1,\dots,\overline{t}_d]$, we use the space of non-commutative $*$-polynomials $\C\ip{t_1,\dots,t_d,t_1^*,\dots,t_d^*}$.  The non-commutative $*$-law of a $d$-tuple $\bX$ is the map $\mu_{\bX}: \C\ip{t_1,\dots,t_d,t_1^*,\dots,t_d^*} \to \C$ given by
	\[
	\mu_{\bX}: p \mapsto \tau(p(\bX)).
	\]
	For $r > 0$, we denote by $\Lambda_{d,r}$ the space of non-commutative of $*$-laws of tuples $\bX$ such that each $\norm{X_j} \leq r$.  From a $\mathrm{C}^*$-algebraic viewpoint, $\Lambda_{d,r}$ is the space of tracial states on the $\mathrm{C}^*$-universal free product of $d$ copies of the universal $\mathrm{C}^*$-algebra generated by an operator of norm $\leq R$.%for $\mathrm{C}^*$-algebraic free product, see Example 5.15 in Szabo's article in this volume).
	We equip $\Lambda_{d,r}$ with the weak-$*$ topology obtained by viewing it as a subset of the dual of $\C\ip{t_1,\dots,t_d,t_1^*,\dots,t_d^*}$ (that is, the topology of pointwise convergence on $\C\ip{t_1,\dots,t_d,t_1^*,\dots,t_d^*}$); in this topology it is compact and metrizable.
	
	One of the central notions of free probability is free independence.  Given a tracial $\mathrm{W}^*$-algebra $\cM$, the free independence of two $\mathrm{W}^*$-subalgebras $\cA$ and $\cB$ is a condition on the joint moments $\tau(a_1 b_1 \dots a_n b_n)$ for $a_j \in \cA$ and $b_j \in \cB$, which allows them to be expressed in terms the traces of elements from $\cA$ and from $\cB$ individually (see e.g.\ \cite[Theorem 19]{MS2017}).  %to do: find Voiculescu-Dykema-Nica reference later
	For our purposes, it is sufficient to use the following characterization of free independence.
	
	\begin{proposition}[{See e.g.\ \cite[\S 5]{AGZ2009} and \cite{VDN1992}}]
		Let $\cA$ and $\cB$ be tracial $\mathrm{W}^*$-algebras and let $\cA * \cB$ be their (tracial) free product.  If $\iota: \cA * \cB \hookrightarrow \cM$ is a trace-preserving $*$-homomorphism, then $\iota(\cA)$ and $\iota(\cB)$ are freely independent in $\cM$.
	\end{proposition}
	
	One of the most important laws in classical probability theory is the Gaussian or normal distribution.  The analog in free probability is the semicircular distribution (see for instance \cite{VDN1992}).  A \emph{semicircular random variable} is a self-adjoint $X$ from some tracial $\mathrm{W}^*$-algebra $\cM$ satisfying
	\[
	\tau_{\cM}[f(X)] = \frac{1}{2\pi} \int_{-2}^2 f(x) \sqrt{4 - x^2}\,dx \text{ for } f \in C(\R).
	\]
	The multivariate analog of a semicircular random variable is a \emph{free semicircular family}, that is, a $d$-tuple $(X_1,\dots,X_d)$ such that each $X_j$ is semicircular, and $\mathrm{W}^*(X_1)$, \dots, $\mathrm{W}^*(X_d)$ are freely independent.  Similarly, the analog of a complex Gaussian is a \emph{circular random variable}, that is, a variable of the form $Z = (X + iY) / \sqrt{2}$ where $X$ and $Y$ are freely independent semicircular variables; and a \emph{free circular family} is a $d$-tuple $(Z_1,\dots,Z_d)$ of circular random variables generating freely independent $\mathrm{W}^*$-algebras.
	
	\subsection{Free entropy} \label{subsec:entropyintro}
	
	If there are non-commutative laws or probability distributions, there should also be a free analog of the entropy of a probability distribution.  Recall that the differential entropy of a probability measure $\mu$ on $\C^d$ is $h(\mu) = -\int \rho \log \rho \,dx$ whenever $\mu$ has a density $\rho$ with respect to Lebesgue measure (and if there is no density, $h(\mu)$ is defined to be $-\infty$).
	
	Voiculescu defined several analogs of $h$ in free probability.  We will primarily be concerned with the ``free microstate entropy'' $\chi(\mu)$ for a non-commutative $*$-law of a $d$-tuple, defined in \cite{VoiculescuFE2}.  Unlike $h$, the free entropy $\chi$ cannot be defined directly by an integral formula (except in the case of a single self-adjoint random variable).  Instead Voiculescu took an approach based on statistical mechanics and Boltzmann's formulation of entropy in terms of microstates.  The microstate entropy, roughly speaking, quantifies the amount of matrix tuples that have non-commutative laws close to $\mu$.
	
	More precisely, the free microstate entropy is defined as follows.  Let $\cO$ be an open set in $\Lambda_{d,r}$.  Then we define the microstate space
	\[
	\Gamma_r^{(n)}(\cO) := \{ \bY \in M_n(\C)^d: \mu_{\bY} \in \cO \}.
	\]
	The free entropy is defined using the exponential growth rate of the Lebesgue measure of these microstate spaces.  Note that $M_n(\C)^d$ is an $dn^2$-dimensional complex inner-product space with the inner product
	\[
	\ip{\bX,\bY} = \sum_{j=1}^d \tr_n(X_j^* Y_j).
	\]
	Thus, there is a canonical Lebesgue measure on $M_n(\C)_{\sa}^d$ defined by identifying it with $\C^{dn^2}$ by an isometry.  Let $\vol \Gamma_r^{(n)}(\cO)$ be its Lebesgue measure.  Then
	\[
	\chi_r(\mu) = \inf_{\cO \ni \mu} \left( \limsup_{n \to \infty} \left(\frac{1}{n^2} \log \vol \Gamma_r^{(n)}(\cO) + 2d \log n\right) \right),
	\]
	where $\cO$ ranges over all open neighborhoods in $\Lambda_{d,r}$ of $\mu$.  The normalization of dividing by $n^2$ and adding $2d \log n$ can be motivated by the fact that this normalization will yield a finite limit if we applied it to the log-volume of an $r$-ball in $\C^{dn^2}$; similarly, one can check that this normalization works for variables whose real and imaginary parts are freely independent semicirculars \cite{VoiculescuFE2} (for further discussion of the normalization, see \cite[appendix]{ST2022}).  We also write $\chi_r(\bX) = \chi_r(\mu_{\bX})$ for a $d$-tuple $\bX$.
	
	Here we give the definition of $\chi$ for arbitrary (non-self-adjoint) $d$-tuples rather than the self-adjoint version from \cite{VoiculescuFE2,VoiculescuFE3}.  This is because we have chosen to work with non-self-adjoint tuples in this paper because it seems to fit more naturally with the definition of types in the model theory of tracial $\mathrm{W}^*$-algebras.  The self-adjoint version is defined in the same way as described above except that we restrict to self-adjoint tuples $\bX$ and correspondingly consider microstate spaces of self-adjoint matrix $d$-tuples rather than arbitrary matrix $d$-tuples.  Because every matrix $Z$ can be uniquely written as $X + iY$ where the ``real and imaginary parts'' $X$ and $Y$ are self-adjoint, one can show that the free entropy for $d$ non-self-adjoint variables $(Z_1,\dots,Z_d)$ agrees with the free entropy for the $2d$ self-adjoint variables obtained by taking the real and imaginary parts of the $Z_j$.
	
	In the definition of $\chi$, we do not know whether using a $\limsup$ or $\liminf$ will give the same answer (see \cite{Voiculescu2002}).  Thus, it is also convenient to use an ultrafilter variant as in \cite{VoiculescuFE4}:  For a free ultrafilter $\cU$ on $\N$, set
	\[
	\chi_R^{\cU}(\mu) = \inf_{\cO \ni \mu} \lim_{n \to \cU} \left( \frac{1}{n^2} \log \vol(\Gamma_R^{(n)}(\cO)) + 2d \log n \right).
	\]
	Intuitively, if $\cM$ is a tracial $\mathrm{W}^*$-algebra generated by $\bX = (X_1,\dots,X_d)$, then $\chi_R^{\cU}(\mu_{\bX})$ measures the amount of ways to embed $\cM$ into the ultraproduct $\cQ = \prod_{n \to \cU} M_n(\C)$.  Indeed, an embedding of $\cM$ into $\cQ$ is uniquely determined by where it sends the generators $\bX$, and thus corresponds to an equivalence class of sequences $[\bY^{(n)}]_{n \in \N}$ such that $\bY^{(n)} \in M_n(\C)^d$ and $\lim_{n \to \cU} \lambda_{\bY^{(n)}} = \lambda_{\bX}$.  Moreover, $\lim_{n \to \cU} \mu_{\bY^{(n)}} = \mu_{\bX}$ means precisely that for every neighborhood $\cO$ of $\mu_{\bX}$, we have that $\{n: \mu_{\bY^{(n)}} \in \cO\}$ or $\{n: \bY^{(n)} \in \Gamma_R^{(n)}(\cO)\}$ is in the ultrafilter $\cU$.
	
	Voiculescu also defined in \cite{VoiculescuFE3} the free entropy $\chi(\mathbf{X}: \mathbf{Y})$, called the \emph{free entropy of $\mathbf{X}$ in the presence of $\mathbf{Y}$}, where $\mathbf{X} = (X_1,\dots,X_d)$ and $\mathbf{Y} = (Y_1,\dots,Y_m)$ are tuples in a tracial $\mathrm{W}^*$-algebra $\cM$.  This version is defined by replacing the space of matrix microstates for $\mathbf{X}$ alone with the space of microstates for $\mathbf{X}$ which extend to microstates for $(\mathbf{X},\mathbf{Y})$.  See \S \ref{subsec:inthepresence} for details.
	
	Entropy in the presence has several operator-algebraic applications.  The first one which is of interest for this paper is its relationship with relative commutants.  Voiculescu showed in \cite[Corollary 4.3]{VoiculescuFE3} that if $\chi(\mathbf{X}: \mathbf{Y}) > -\infty$, then $\mathrm{W}^*(\mathbf{X})' \cap \mathrm{W}^*(\mathbf{X},\mathbf{Y}) = \C$ (and in fact there are no nontrivial sequences in $\mathrm{W}^*(\bX,\bY)$ that asymptotically commute with $\mathrm{W}^*(\bX)$).  Later, Voiculescu defined
	\[
	\chi(\mathbf{X}:\cM) = \inf_{m \in \N} \inf_{\mathbf{Y} \in M^m} \chi(\mathbf{X}: \mathbf{Y}).
	\]
	Note that if $\chi(\mathbf{X}:\cM) > -\infty$, then $\mathrm{W}^*(\mathbf{X})' \cap \cM = \C$ since we can apply \cite[Corollary 4.3]{VoiculescuFE3} to any tuple $\mathbf{Y}$ from $\mathrm{W}^*(\mathbf{X})' \cap \cM$.
	
	In this paper, we will prove the following result about entropy in the presence.
	
	\begin{theorem} \label{thm:entropyembedding}
		Let $\cM$ be a tracial non-commutative probability space and $\bX \in \cM^d$.  If $\chi^{\cU}(\bX: \cM) > -\infty$, then there exists an embedding of $\iota: \cM \to \cQ = \prod_{n \to \cU} M_n(\C)$ such that
		\[
		\chi^{\cU}(\iota(\bX): \cQ) = \chi^{\cU}(\bX,\cM).
		\]
		In particular, by \cite[Corollary 4.3]{VoiculescuFE3}, this implies that $\iota(\bX)' \cap \cQ = \C$, hence also $\iota(\cM)' \cap \cQ = \C$.
	\end{theorem}
	
	The hypotheses of the theorem hold when $\bX$ is a free circular family and $\cM = \mathrm{W}^*(\bX)$ (or more generally the free product of $\mathrm{W}^*(\bX)$ with some other Connes-embeddable tracial $\mathrm{W}^*$-algebra); see \cite{VoiculescuFE3}.  Thus, in particular, for $d \geq 1$ there exists an embedding of a free group $\mathrm{W}^*$-algebra $L(\mathbb{F}_{2d})$ into $\cQ$ with trivial relative commutant since $L(\mathbb{F}_{2d})$ is isomorphic to the tracial $\mathrm{W}^*$-algebra generated by a free circular $d$-tuple.  (One can show the case for $L(\mathbb{F}_m)$ with $m \geq 3$ odd either by considering the self-adjoint version of entropy or observing that $L(\mathbb{F}_{2d+1}) = L(\mathbb{F}_{2d}) * L(\Z)$.)
	
	We remark that the result that if $\chi^{\cU}(\bX:\cM) > -\infty$, then there exists an embedding of $\mathrm{W}^*(\bX)$ into $\cQ$ with trivial relative commutant could be proved directly from the same arguments as \cite[\S 3]{VoiculescuFE3}.  Roughly speaking, the argument is that for $\delta \in (0,1/2)$, the set of matrix tuples for which there exists some projection $P$ of trace between $\delta$ and $1 - \delta$ almost commuting with $\bX$ has very small volume compared to the microstate space for $\bX$ in the presence of $\cM$.  Consequently, a randomly chosen embedding for $\mathrm{W}^*(\bX)$ extending to an embedding of $\cM$ could not have any nontrivial projection that commutes with $\iota(\mathrm{W}^*(\bX))$.  This probabilistic argument is similar in spirit to von Neumann's earlier work \cite{vN1942} which showed that for large $n$ most matrices $A$ do not approximately commute with any matrices $B$ other than those which are close to scalar multiples of identity.
	
	However, the stronger statement that there exists an embedding with $\chi^{\cU}(\iota(\bX):\cQ) = \chi^{\cU}(\bX:\cM)$ is new to our work.  We will prove this theorem by means of adapting Voiculescu's free entropy to model-theoretic \emph{types} rather than non-commutative laws (which are equivalent to \emph{quantifier-free types}).  First, we remind the reader of the basic model-theoretic setup.
	
	\subsection{Model theory and types} \label{subsec:modelintro}
	
	The work of \cite{FHS2013,FHS2014,FHS2014b} put tracial $\mathrm{W}^*$-algebras into the framework of continuous model theory of \cite{BYBHU2008,BYU2010}.  The authors constructed a language $\cL_{\tr}$ for tracial $\mathrm{W}^*$-algebras as well as a set of axioms $\rT_{\tr}$ that characterize when an $\cL_{\tr}$-structure comes from a tracial $\mathrm{W}^*$-algebra.  %For further background, see Goldbring and Hart's article in this volume.
	
	We will follow the treatment in \cite{FHS2014} which introduces ``domains of quantification'' to cut down on the number of ``sorts'' neeeded.  A language can have many sorts, and each sort can have many domains of quantification.  The function and relation symbols come with a given modulus of uniform continuity for each product of domains of quantification.  For tracial $\mathrm{W}^*$-algebras, we have a single sort and the domains of quantification are operator-norm balls $D_r$ for $r \in (0,\infty)$. (Although the original paper used only integer values of $r$, we choose here to work with arbitrary positive values, which changes the setup of the languages but does not affect the proofs.)
	
	Model-theoretic notions such as terms and formulas correspond to familiar objects in von Neumann algebras and free probability:
	\begin{itemize}
		\item Terms in $\cL_{\tr}$ are expressions obtained from iterating scalar multiplication, addition, multiplication, and the $*$-operation on variables and the unit symbol $1$.  If $(M,\tau)$ is a tracial $\mathrm{W}^*$-algebra, then the interpretation of the term in $\mathcal{M}$ is a function represented by a $*$-polynomial.
		\item In $\mathcal{L}_{\tr}$, a basic formula can take the form $\re \tr(f)$ or $\im \tr(f)$ where $f$ is an expression obtained by iterating the $*$-algebraic operations.  Thus, when evaluated in a tracial $\mathrm{W}^*$-algebra, it corresponds to the real or imaginary part of the trace of a non-commutative $*$-polynomial.
		\item Quantifier-free formulas can be obtained by applying continuous functions (``connectives'') to basic formulas.  When the connectives are polynomial, this relates to an object in the free probability and random matrix literature called a \emph{trace polynomial}.
	\end{itemize}
	General formulas are obtained from basic formulas recursively by applying such connectives and also taking the supremum or infimum in some variable over some domain of quantification.  Although general formulas have not been studied much in free probability, our goal is to show that such a study is natural and worthwhile.
	
	For convenience, we will assume that our formulas do not have two copies of the same variable (i.e. if a variable is bound to a quantifier, there is no other variable of the same name that is free or bound to a different quantifier).  For instance, in the formula
	\[
	\im \tr(x_1) \sup_{x_1 \in D_1} \re \tr(x_1x_2 + x_3 x_1^*),
	\]
	the first occurrence of $x_1$ is free while the latter two occurrences are bound to the quantifier $\sup_{x_1 \in D_1}$, but we can rewrite this formula equivalently as
	\[
	\im \tr(x_1) \sup_{y_1 \in D_1} \re \tr(y_1x_2 + x_3 y_1^*).
	\]
	We will typically denote the free variables by $(x_i)_{i \in \N}$ and the bound variables by $(y_i)_{i \in \N}$.  Lowercase letters will be used for formal variables while uppercase letters will be used for individual operators in operator algebras.
	
	For the most part, we will use formulas in finite or countably many free variables.  Thus, we will consider an index set $I$ (often finite or countable), an $I$-tuple $\bS$ of sorts, and an $I$-tuple $\mathbf{x} = (x_i)_{i \in I}$ of variables in those sorts.  We denote by $\cF_{\bS}$ the set of formulas with free variables $(x_i)_{i \in I}$ where $x_i$ is from $S_i$.  Since tracial $\mathrm{W}^*$-algebras have only one sort, we will use the notation $\cF_I$ in this case.
	
	Given an $\cL$-structure $\cM$ and an $I$-tuple $\bX \in \prod_{i \in I} S_i^{\cM}$, each formula $\phi$ has an interpretation $\phi^{\cM}$ which can then be evaluated at $\bX$ to yield $\phi^{\cM}(\bX)$.  The \emph{type} of $\bX$ in $\cM$ is the map
	\[
	\tp^{\cM}(\bX): \cF_{\bS} \to \R, \ \phi \mapsto \phi^{\cM}(\bX).
	\]
	
	In the case of tracial $\mathrm{W}^*$-algebras, this is very similar to how a non-commutative law of $\bX$ is defined through the evaluation of the trace of a non-commutative polynomial on $\bX$.  Indeed, the non-commutative law of $\bX$ is obtained by restricting the linear functional $\tp^{\cM}(\bX)$ to the set of basic formulas inside of $\cF_I$.  Since the value of basic formulas at $\bX$ uniquely determines the value of all quantifier-free formulas, the non-commutative law contains the same information as the evaluation of quantifier-free formulas on $\bX$, namely the \emph{quantifier-free type} $\tp_{\qf}^{\cM}(\bX)$.
	
	Therefore, from a probabilistic viewpoint, the type of a tuple is an enrichment of the non-commutative $*$-law, which in turn is an analog of a classical probability distribution.  One could also think of the type of a $d$-tuple in tracial $\mathrm{W}^*$-algebras as an analog of the type of classical random variables described in \cite{BY2012}.  However, it turns out that the theory of atomless classical probability spaces admits quantifier elimination (see \cite[Example 4.3]{BYU2010} and \cite[Fact 2.10]{BY2012}%as well as Berenstein and Henson's article in this volume
	), and so there is no practical distinction between the type and the quantifier-free type.  We will state and prove this result about quantifier elimination in the framework of commutative tracial $\mathrm{W}^*$-algebras (see Theorem \ref{thm:commutativeQE}), showing that the $\cL_{\tr}$-structure $L^\infty[0,1]$, with the trace given by Lebesgue measure, admits quantifier elimination.  Hence the type of a $d$-tuple in $L^\infty[0,1]$ (or any model of the same theory) is uniquely determined by its quantifier-free type.   Therefore, it is worth asking whether in the non-commutative setting the full type might in some sense be a better analog for a probability distribution than the $*$-law itself.  We discuss this idea further in \S \ref{sec:optimization}.
	
	But if the type is to be viewed as an analog of the classical probability distribution, then we would want corresponding notions of independence and of entropy.  This idea provides additional motivation for our development of the analog of Voiculescu's entropy $\chi^{\cU}$ for types.  However, we will not resolve the problem of finding a corresponding notion of independence; we give some suggestions and challenges in \S \ref{sec:independence}.

	\subsection{Outline}
	
	\S \ref{sec:entropyfortypes} defines free entropy $\chi_{\full}^{\cU}(\mu)$ for a full (or complete) type $\mu$.  Then in \S \ref{sec:qfexists}, we explain how the analogous construction for quantifier-free types agrees with Voiculescu's original $\chi^{\cU}$, and furthermore how the analogous quantity $\chi_{\exists}^{\cU}$ for existential types agrees with Voiculescu's entropy in the presence.  The section conclude with the proof of Theorem \ref{thm:entropyembedding}.  \S\ref{sec:entropyfortypes} and \S \ref{sec:qfexists} are closely parallel to \cite{JekelCoveringEntropy}, which handles the full-type analog of the $1$-bounded entropy of Jung \cite{Jung2007S1B} and Hayes \cite{Hayes2018}.
	
	In \S \ref{sec:independence}, we consider the problem of defining an analog of free independence for types rather than for non-commutative laws.  We sketch several approaches to independence (free products, random matrix theory, and model theory), and state many open questions based on these approaches.
	
	In \S \ref{sec:optimization}, we discuss several open problems in free probability involving optimization (i.e.\ $\sup$'s and $\inf$'s), motivated by attempts to adapt the classical theory about optimal transport, entropy, and Gibbs $*$-laws to the non-commutative setting.  We suggest ways in which the model-theoretic framework could bring further insight into these issues.
	
	\subsection*{Acknowledgements}
	
	I thank Isaac Goldbring for suggesting for me to write about this topic and for offering detailed comments.  Adrian Ioana brought to my attention the problem of using random matrix theory to construct embeddings into matrix ultraproducts that have trivial relative commutant, as well as von Neumann's work on almost commuting matrices.  This question was further motivated by various discussions by Sorin Popa about the factorial relative commutant embedding problem.
	
	\section{Types, definable functions, and quantifier elimination}
	
	\subsection{Types and definable predicates}
	
	We use the following notation for concepts in continuous model theory:
	\begin{itemize}
		\item $\cL$ denotes a language.
		\item $\cS$ denotes the collection of sorts in $\cL$.
		\item For each sort $S$, $\cD_S$ denotes the set of domains of quantification in $S$.
		\item $\cM$ will be an $\cL$-structure.
		\item We denote the interpretation of sorts, domains, relation symbols, function symbols, etc. in $\cM$ using a superscript $\cM$.
		\item We usually use $\rT$ for theories.
		\item $\cL_{\tr}$ will be the language of tracial $\mathrm{W}^*$-algebras and $\rT_{\tr}$ will be the theory of tracial $\mathrm{W}^*$-algebras given by \cite{FHS2014}.
		\item A tracial $\mathrm{W}^*$-algebra is a pair $\cM = (M,\tau)$ where $M$ is a $\mathrm{W}^*$-algebra and $\tau$ is a faithful normal tracial state on $M$.
		\item $\cL_{\tr}$ has only one sort; we do not explicitly name this sort, but its interpretation in $\cM$ will simply be denoted $M$.
	\end{itemize}
	
	\begin{definition}
		Let $I$ be an index set, and let $\mathbf{S} = (S_i)_{i \in I}$ be an $I$-tuple of sorts in a language $\mathcal{L}$.  Let $\cF_{\mathbf{S}}$ be the space of $\cL$-formulas with free variables $(x_i)_{i \in I}$ with $x_i$ from the sort $S_i$.  If $\cM$ is an $\cL$-structure and $\mathbf{X} \in \prod_{i \in I} S_i^{\cM}$, then the \emph{type} of $\mathbf{X}$ is the map
		\[
		\tp^{\cM}(\mathbf{X}): \cF_{\mathbf{S}} \to \R
		\]
		given by
		\[
		\phi \mapsto \phi^{\cM}(\mathbf{X}).
		\]
	\end{definition}
	
	\begin{remark}
		In model theory, one often studies the type of a tuple \emph{over $\cA \subseteq \cM$}, which records the values of formulas with coefficients in $\cA$ (see \cite[\S 8]{BYBHU2008}).  For most of this paper, we take $\cA = \varnothing$.
	\end{remark}
	
	\begin{definition}
		Let $\mathbf{S} = (S_i)_{i \in I}$ be an $I$-tuple of sorts in $\mathcal{L}$, and let $\rT$ be an $\cL$-theory.  We denote by $\mathbb{S}_{\bS}(\rT)$ the set of types $\tp^{\cM}(\bX)$ for all $\bX \in \prod_{i \in I} S_i^{\cM}$ for all $\cM \models \rT$.  Similarly, if $\mathbf{D} \in \prod_{i \in I} \mathcal{D}_{S_i}$, then we denote by $\mathbb{S}_{\bD}(\rT)$ the set of types $\tp^{\cM}(\bX)$ of all $\bX \in \prod_{i \in I} D_i^{\cM}$ for all $\cM \models \rT$.
	\end{definition}
	
	The space of types $\mathbb{S}_{\bD}(\rT)$ is equipped with a weak-$*$ topology as follows.  If $\bS$ is an $I$-tuple of $\cL$-sorts, the set $\cF_{\bS}$ of formulas defines a real vector space.  For each $\cL$-structure $\cM$ and $\bX \in \prod_{i \in I} S_j^{\cM}$, the type $\tp^{\cM}(\bX)$ is a linear map $\cF_{\bS} \to \R$.  Thus, for each $\cL$-theory $\rT$ and $\bD \in \prod_{i \in I} \cD_{S_i}$, the space $\mathbb{S}_{\bD}(\rT)$ is a subset of the real dual $\cF_{\mathbf{S}}^\dagger$.  We equip $\mathbb{S}_{\bD}(\rT)$ with the weak-$*$ topology (also known as the \emph{logic topology}).
	
	$\mathbb{S}_{\bD}(\rT)$ is compact in this topology \cite[Proposition 8.6]{BYBHU2008}.  Moreover, if $I$ is at most countable and the language is separable, then $\mathbb{S}_{\bD}(\rT)$ is metrizable.  In particular, this applies to $\cL_{\tr}$ and $\rT_{\tr}$ (see \cite[Observations 3.13 and 3.14]{JekelCoveringEntropy} for further explanation).
	
	In the setting with domains of quantification, it is also convenient to have a topology on $\mathbb{S}_{\bS}(\rT)$.  We say that $\cO \subseteq \mathbb{S}_{\bS}(\rT)$ is \emph{open} if $\cO \cap \mathbb{S}_{\bD}(\rT)$ is open for every $\bD \in \prod_{j \in \N} \cD_{S_j}$; this defines a topology on $\mathbb{S}_{\bS}(\rT)$, which we will also call the \emph{logic topology}.  One can then show that the inclusion map $\mathbb{S}_{\bD}(\rT) \to \mathbb{S}_{\bS}(\rT)$ is a topological embedding \cite[Observation 3.6]{JekelCoveringEntropy}.
	
	Every formula defines a continuous function on $\mathbb{S}_{\bS}(\rT)$ for each theory $\rT$, but the converse is not true.  The objects that correspond to continuous functions on $\mathbb{S}_{\bS}(\rT)$ are a certain completion of the set of formulas, called \emph{definable predicates}.  Our approach to the definition will be semantic rather than syntactic, defining these objects immediately in terms of their interpretations.
	
	\begin{definition} \label{def:definablepredicate}
		Let $\mathcal{L}$ be a language and $\rT$ an $\mathcal{L}$-theory.  A \emph{definable predicate relative to $\rT$} is a collection of functions $\phi^{\cM}: \prod_{i \in I} S_i^{\cM} \to \R$ (for each $\cM \models \rT$) such that for every collection of domains $\mathbf{D} = (D_j)_{j \in \N}$ and every $\epsilon > 0$, there exists a finite $F \subseteq I$ and an $\mathcal{L}$-formula $\psi(x_i: i \in F)$ such that whenever $\cM \models \rT$ and $\mathbf{X} \in \prod_{j \in \N} D_j^{\cM}$, we have
		\[
		|\phi^{\cM}(\mathbf{X}) - \psi^{\cM}(X_i: i \in F)| < \epsilon.
		\]
	\end{definition}
	
	Here we describe definable predicates relative to an arbitrary theory $\rT$ because we want to study definable predicates that make sense uniformly for all models of some theory $\rT$ (for instance $\rT_{\tr}$ in $\cL_{\tr}$) rather than merely for a single $\cL$-structure or for all $\cL$-structures.  In our definition, two distinct formulas may result in the same definable predicate.
	
	In \cite[Theorem 9.9]{BYBHU2008}, the authors show that definable predicates correspond to continuous functions on the space of types.  This result adapts to the setting with multiple sorts and domains of quantification as well (in fact, our definition of the logic topology on $\mathbb{S}_{\bS}(\rT)$ was chosen so that this would work!).  For proof, see \cite[Proposition 3.9]{JekelCoveringEntropy} (although this proposition assume countable index sets, removing this restriction does not affect the argument).
	
	\begin{proposition}[{\cite[Observation 3.11]{JekelCoveringEntropy}}] \label{prop:defpredcontinuous}
		Let $\cL$ be a language and $\rT$ an $\cL$-theory.  Let $\phi$ be a collection of functions $\phi^{\cM}: \prod_{i \in I} S_i^{\cM} \to \R$ for each $\cM \models \rT$.  The following are equivalent:
		\begin{enumerate}[(1)]
			\item $\phi$ is a definable predicate relative to $\rT$.
			\item There exists a continuous $\gamma: \mathbb{S}_{\bS}(\rT) \to \R$ such that $\phi^{\cM}(\bX) = \gamma(\tp^{\cM}(\bX))$ for all $\cM \models \rT$ and $\bX \in \prod_{i \in I} S_i^{\cM}$.
		\end{enumerate}
	\end{proposition}
	
	Just like formulas, definable predicates satisfy a certain uniform continuity property with respect to the metric $d^{\cM}$.
	
	\begin{observation} \label{obs:defpredunifcont}
		If $\phi = (\phi^{\cM})$ is a definable predicate in sorts $\bS$ over $\mathcal{L}$ relative to $\rT$, then $\phi$ satisfies the following uniform continuity property:
		
		For every $\mathbf{D} \in \prod_{i \in I} \mathcal{D}_{S_i}$ and $\epsilon > 0$, there exists a finite $F \subseteq I$ and $\delta > 0$ such that, whenever $\mathcal{M} \models \rT$ and $\mathbf{X}, \mathbf{Y} \in \prod_{j \in \N} D_j^{\cM}$,
		\[
		d^{\cM}(X_i,Y_i) < \delta \text{ for all } i \in F \implies |\phi^{\cM}(\mathbf{X}) - \phi^{\cM}(\mathbf{Y})| < \epsilon.
		\]
		Moreover, for every $\mathbf{D} \in \prod_{i \in O} \mathcal{D}_{S_i}$, there exists a constant $C$ such that $|\phi^{\cM}| \leq C$ for all $\cM \models \rT$.
	\end{observation}
	
	\begin{corollary}
		Let $\cL$ be a language, $\cM$ a model of a theory $\rT$, let $\bS$ be an $I$-tuple of sorts, and let $\bD \in \prod_{i \in I} \cD_{S_i}$.  Let $J$ be a directed set, let $(\bX^{(j)})_{j \in J}$ be a net of $I$-tuples from $\prod_{i \in I} D_i^{\cM}$, and let $\bX \in \prod_{i \in I} D_i^{\cM}$.  If $\lim_{j \in J} X_i^{(j)} = X_i$ for each $i\in I$, then $\lim_{j \in J} \tp^{\cM}(\bX^{(j)}) = \tp^{\cM}(\bX)$.
	\end{corollary}
	
	Finally, we recall that there are quantifier-free versions of all the notions in this section, which are defined analogously except with quantifier-free formulas instead of all formulas; in particular:
	\begin{itemize}
		\item Let $\cF_{\qf,\bS}$ denote the set of quantifier-free formulas (i.e.\ those defined without using any $\sup$ or $\inf$) in variables from an $I$-tuple $\bS$ of sorts.  Given $\bX \in \prod_{i \in I} S_i^{\cM}$, the \emph{quantifier-free type} $\tp_{\qf}^{\cM}(\bX)$ is the map $\cF_{\qf,\bS} \to \R$ given by $\phi \mapsto \phi^{\cM}(\bX)$.
		\item Given $\bD \in \prod_{i \in I} \cD_{S_i}$, we denote by $\mathbb{S}_{\qf,\bD}(\rT)$ the space of quantifier-free types of elements from $\prod_{i \in I} D_i^{\cM}$ for any $\cM \models \rT$.  We equip it with the weak-$*$ topology as a subset of the dual of $\cF_{\qf,\bS}$.
		\item We denote the set of all quantifier-free types of tuples from $\bS$ by $\mathbb{S}_{\bS}(\rT)$.  As in the case of $\mathbb{S}_{\bS}(\rT)$, we equip $\mathbb{S}_{\qf,\bS}(\rT)$ with the ``union topology.''
		\item A \emph{quantifier-free definable predicate} relative to an $\cL$-theory $\rT$ is defined analogously to Definition \ref{def:definablepredicate} except that now the formula $\psi$ is required to be a quantifier-free formula.
		\item Analogous to Proposition \ref{prop:defpredcontinuous}, quantifier-free definable predicates correspond to continuous functions on $\mathbb{S}_{\qf,\bS}(\rT)$.
	\end{itemize}
	
	The relationship between types and quantifier-free types is as follows (we will need this to relate our model-theoretic version of $\chi$ with Voiculescu's).  The quantifier-free type $\tp_{\qf}^{\cM}(\bX)$ is the restriction of $\tp^{\cM}(\bX)$ from $\cF_{\bS}$ to $\cF_{\qf,\bS}$.  The restriction operation defines a map $\pi_{\qf}: \mathbb{S}_{\bS}(\rT) \to \mathbb{S}_{\qf,\bS}(\rT)$.  It is immediate that for each $I$-tuple $\bD$ of domains, the map $\pi_{\qf}$ is weak-$*$ continuous and surjective $\mathbb{S}_{\bD}(\rT) \to \mathbb{S}_{\bS}(\rT)$, and since the domain and codomain are Hausdorff it is therefore a topological quotient map.  By definition of the topology on $\mathbb{S}_{\bS}(\rT)$, it follows that $\pi_{\qf}$ is continuous $\mathbb{S}_{\bS}(\rT) \to \mathbb{S}_{\qf,\bS}(\rT)$, and furthermore it is a topological quotient map.
	
	\subsection{Definable functions}
	
	Here we summarize a few definitions and facts about definable functions needed to establish the properties of entropy in \S \ref{sec:qfexists}.  In particular, we recall the result from \cite{JekelCoveringEntropy} that for an $I$-tuple $\bX$ in a tracial $\mathrm{W}^*$-algebra, every element $Y \in \mathrm{W}^*(\bX)$ can be expressed as $Y = f(\bX)$ for some quantifier-free definable function $f$.
	
	\begin{definition} \label{def:definablefunctions}
		Let $\cL$ be a language, $\rT$ an $\cL$-theory, $I$ an index set, $\bS$ an $I$-tuple of sorts and $S'$ another sort.  A \emph{definable function} $f: \prod_{i \in I} S_i \to S'$ relative to $\rT$ is a collection of functions $f^{\cM}: \prod_{i \in I} S_i^{\cM} \to (S')^{\cM}$ for each $\cM \models \rT$ satisfying the following conditions:
		\begin{enumerate}
			\item For every $\bD \in \prod_{i \in I} \cD_{S_i}$, there exists $D' \in \cD_{S'}$ such that each $f^{\cM}$ maps $\prod_{i \in I} \cD_{S_i}^{\cM}$ into $D_{S'}^{\cM}$.
			\item The function $\phi^{\cM}(\bX,Y) = d^{\cM}(f^{\cM}(\bX),\bY)$ for $\cM \models \rT$ is a definable predicate in variables from the sorts $(\bS,S')$.
		\end{enumerate}
	\end{definition}
	
	It is convenient to consider also definable functions where the output is a tuple.  A \emph{definable function} $\mathbf{f}: \prod_{i \in I} S_i \to \prod_{j \in I'} S_j'$ relative to $\rT$ is simply a tuple where each $f_j$ is definable.  We also have the following characterization (or alternative definition if you will) of definable functions in terms of how they interact with definable predicates.  For proof of this equivalence, see \cite[Proposition 3.17]{JekelCoveringEntropy} (although the statement there was only given for countable index sets, it generalizes to arbitrary index sets without any essential changes).
	
	\begin{proposition} \label{prop:definablefunction}
		Let $\bS$ and $\bS'$ be $I$ and $I'$-tuples of sorts in the language $\cL$.  Consider a collection of maps $\mathbf{f}^{\cM}: \prod_{j \in \N} S_j^{\cM} \to \prod_{j \in \N} (S_j')^{\cM}$ for $\cM \models \rT$.  Then $\mathbf{f}$ is a definable function $\prod_{i \in I} S_i \to \prod_{i \in I'} S_i'$ relative to the $\cL$-theory $\rT$ if and only if the following conditions hold:
		\begin{enumerate}[(1)]
			\item For each $\bD \in \prod_{j \in \N} \cD_{S_j}$, there exists $\bD' \in \prod_{j \in \N} \cD_{S_j'}$ such that for every $\cM \models \rT$, $\mathbf{f}^{\cM}$ maps $\prod_{j \in \N} D_j^{\cM}$ into $\prod_{j \in \N} (D_j')^{\cM}$.
			\item Whenever $\tilde{\bS}$ is an $\tilde{I}$-tuple of sorts and $\phi$ is a definable predicate relative to $\rT$ in the free variables $x_j' \in S_j'$ for $j \in I'$ and $\tilde{x}_j \in \tilde{S}_j$ for $j \in \tilde{I}$, then $\phi(\mathbf{f}(\mathbf{x}),\tilde{\mathbf{x}})$ is a definable predicate in the variables $\mathbf{x} = (x_j)_{j\in \N}$ and $\tilde{\mathbf{x}} = (\tilde{x}_j)_{j\in\N}$.
		\end{enumerate}
	\end{proposition}
	
	As a consequence, if $\mathbf{f}$ is a definable function $\prod_{i \in I} S_i \to \prod_{i \in I'} S_i'$ and $\phi$ is a definable predicate in the variables $x_j' \in S_j'$, then $\phi \circ \mathbf{f}$ is a definable predicate (see \cite[Lemma 3.20(1)]{JekelCoveringEntropy}).
	
	We want to define quantifier-free definable functions analogously to Definition \ref{def:definablefunctions} except using quantifier-free definable predicates rather than definable predicates.  Unfortunately, we do not know in general whether the analog of Proposition \ref{prop:definablefunction} holds for the quantifier-free setting for an arbitrary choice of theory $\rT$.  Thus, we prefer to define quantifier-free definable functions by the analog of the properties listed in Proposition \ref{prop:definablefunction} instead.
	
	\begin{definition} \label{def:definablefunction}
		Let $\bS$ and $\bS'$ be $I$ and $I'$-tuples of sorts in the language $\cL$.  Consider a collection of maps $\mathbf{f}^{\cM}: \prod_{j \in \N} S_j^{\cM} \to \prod_{j \in \N} (S_j')^{\cM}$ for $\cM \models \rT$.  Then $\mathbf{f}$ is a \emph{quantifier-free definable function} $\prod_{i \in I} S_i \to \prod_{i \in I'} S_i'$ relative to the $\cL$-theory $\rT$ if and only if the following conditions hold:
		\begin{enumerate}[(1)]
			\item For each $\bD \in \prod_{j \in \N} \cD_{S_j}$, there exists $\bD' \in \prod_{j \in \N} \cD_{S_j'}$ such that for every $\cM \models \rT$, $\mathbf{f}^{\cM}$ maps $\prod_{j \in \N} D_j^{\cM}$ into $\prod_{j \in \N} (D_j')^{\cM}$.
			\item Whenever $\tilde{\bS}$ is an $\tilde{I}$-tuple of sorts and $\phi$ is a quantifier-free definable predicate relative to $\rT$ in the free variables $x_j' \in S_j'$ for $j \in I'$ and $\tilde{x}_j \in \tilde{S}_j$ for $j \in \tilde{I}$, then $\phi(\mathbf{f}(\mathbf{x}),\tilde{\mathbf{x}})$ is a quantifier-free definable predicate in the variables $\mathbf{x} = (x_j)_{j\in \N}$ and $\tilde{\mathbf{x}} = (\tilde{x}_j)_{j\in\N}$.
		\end{enumerate}
	\end{definition}
	
	An important consequence of this definition is that if $\phi$ is a quantifier-free definable predicate and $\mathbf{f}$ is a quantifier-free definable function, and if the domain of $\phi$ is the same as the codomain of $\mathbf{f}$, then $\phi \circ \mathbf{f}$ is a quantifier-free definable predicate.  But in any case, \cite[Theorem 3.30]{JekelCoveringEntropy} shows that the two possible definitions of quantifier-free definable functions proposed above are equivalent in the case of $\cL_{\tr}$ and $\rT_{\tr}$.
	
	Furthermore, it turns out that every element in a $\mathrm{W}^*$-algebra can be expressed as a quantifier-free definable function of the generators.
	
	\begin{proposition}[{\cite[Proposition 3.32]{JekelCoveringEntropy}}] \label{prop:deffuncrealization}
		If $\mathcal{M} = (M,\tau)$ is a tracial $\mathrm{W}^*$-algebra and $\mathbf{X} \in \prod_{i \in I} D_{r_i}$ generates $M$ and $\bY \in \prod_{i \in I'} D_{r_i'}^{\cM}$, then there exists a quantifier-free definable function $\mathbf{f}$ in $\cL_{\tr}$ relative to $\rT_{\tr}$ such that $\bY = \mathbf{f}(\mathbf{X})$.  In fact, $\mathbf{f}$ can be chosen so that $f_k^{\cM}$ maps $M^I$ into $\prod_{i \in I'} D_{r_i'}^{\cM}$ for all $\cM \models \rT$.
	\end{proposition}
	
	A special case (and one of the key steps in the proof) of this result is that there exists a one-variable quantifier-free definable predicate $f$ such that $f^{\cM}$ maps $M$ into the unit ball and equals the identity on the unit ball (this is more subtle than in the $\mathrm{C}^*$-algebraic setting because the operator norm is not a definable predicate in $\cL_{\tr}$).  By composing with this $f$, one can extend a definable predicate defined only on a product of operator-norm balls to a globally defined definable predicate.  See \cite[Remark 3.33]{JekelCoveringEntropy} for details.
	
	\begin{proposition} \label{prop:extension}
		Let $I$ be an index set and $\mathbf{r} \in (0,\infty)^I$.  Let $\mathbb{S}_{\mathbf{r}}(\rT_{\tr})$ be the set of types for $I$-tuples in $\prod_{i \in I} D_{r_i}^{\cM}$ for $\cM \models \rT_{\tr}$.  Then for every continuous $\phi: \mathbb{S}_{\mathbf{r}} \to \R$, there exists a continuous extension to $\mathbb{S}_I(\rT_{\tr})$, that is, there us a definable predicate $\psi$ such that $\psi|_{\mathbb{S}_{\mathbf{r}}(\rT_{\tr})} = \phi$.
	\end{proposition}

	\subsection{Quantifier elimination in classical probability} \label{subsec:classicalQE}.
	
	A theory $\rT$ is said to admit \emph{quantifier elimination} if for every $d$, every $d$-variable definable predicate with respect to $\rT$ is a quantifier-free definable predicate.  An $\cL$-structure $\cM$ is said to \emph{admit quantifier elimination} if its theory admits quantifier elimination.  In this section, we show that a diffuse classical probability space admits quantifier elimination (in contrast to general tracial $\mathrm{W}^*$-algebras), which gives a model-theoretic heuristic for why non-commutative probability theory is more difficult than classical probability theory.
	
	\begin{theorem} \label{thm:commutativeQE}
		The $\cL_{\tr}$-structure $L^\infty[0,1]$, with the trace $E: L^\infty[0,1] \to \C$ given by integration, admits quantifier elimination.
	\end{theorem}
	
	This result is closely related to \cite[Example 4.3]{BYU2010} and \cite[Fact 2.10]{BY2012} which showed quantifier elimination for the theory of diffuse classical probability spaces.  %Moreover, Berenstein and Henson's article in this volume gives a more thorough treatment of the model theory of classical probability spaces.
	We will give an argument for quantifier elimination of $L^\infty[0,1]$ directly in the $\mathrm{W}^*$-algebraic framework.  We need several lemmas, the first of which is the characterization of quantifier elimination in terms of types.
	
	\begin{lemma} \label{lem:typeQE}
		Let $\rT$ be a theory in a language $\cL$.  Then $\rT$ admits quantifier elimination if and only if for every $d \in \N$ and $\bS = (S_1,\dots,S_d)$ of sorts, the restriction map $\pi_{\qf}: \mathbb{S}_{\bS}(\rT) \to \mathbb{S}_{\qf,\bS}(\rT)$ is injective (or in other words, an $d$-tuple's type is uniquely determined by its quantifier-free type).
	\end{lemma}
	
	\begin{proof}
		It suffices to show that for every $d \in \N$ and $d$-tuple $\bS$ of sorts, the map $\mathbb{S}_{\bD}(\rT) \to \mathbb{S}_{\bD,\qf}(\rT)$ is injective if and only if every definable predicate in variables $x_1$, \dots, $x_d$ from $S_1$, \dots, $S_d$ is a quantifier-free definable predicate.
		
		Fix $d$ and let $\bS$ be a $d$-tuple of sorts.  If the map $\mathbb{S}_{\bS}(\rT) \to \mathbb{S}_{\qf,\bS}(\rT)$ is injective, then it is bijective.  This implies that the restricted map $\mathbb{S}_{\bD}(\rT) \to \mathbb{S}_{\qf,\bD}(\rT)$ is bijective as well, and since $\mathbb{S}_{\bD}(\rT)$ is compact and $\mathbb{S}_{\qf,\bD}(\rT)$ is Hausdorff, the map $\mathbb{S}_{\bD}(\rT) \to \mathbb{S}_{\qf,\bD}(\rT)$ is a homeomorphism.  But this implies that $\mathbb{S}_{\bS}(\rT) \to \mathbb{S}_{\qf,\bS}(\rT)$ is a homeomorphism.  Definable predicates are equivalent to continuous functions on $\mathbb{S}_{\bS}(\rT)$ while quantifier-free definable predicates are equivalent to continuous functions on $\mathbb{S}_{\qf,\bS}(\rT)$.  The restriction map $\mathbb{S}_{\bS}(\rT) \to \mathbb{S}_{\qf,\bS}(\rT)$ induces the inclusion map $C(\mathbb{S}_{\qf,\bS}(\rT)) \to C(\mathbb{S}_{\bS}(\rT))$ from quantifier-free definable predicates to all definable predicates.  Since this map is bijective, then every definable predicate is quantifier-free definable.
		
		Conversely, suppose that every definable predicate is quantifier-free definable.  If $\mu$ and $\nu$ are two distinct types, then there exists a formula $\phi$ such that $\mu(\phi) \neq \nu(\phi)$.  By assumption $\phi$ is a quantifier-free definable predicate, and hence $\mu(\phi) \neq \nu(\phi)$ implies that $\pi_{\qf}(\mu) \neq \pi_{\qf}(\nu)$.  Thus, $\pi_{\qf}$ is injective as desired.
	\end{proof}
	
	The following fact was observed from the viewpoint of probability algebras in \cite[Fact 2.10]{BY2012}.  %Again, see Berenstein and Henson's article for further discussion.
	
	\begin{lemma} \label{lem:commutativeisomorphism}
		Every separable model of $\Th(L^\infty[0,1],E)$ is isomorphic to $(L^\infty[0,1],E)$.  In other words, $\Th(L^\infty[0,1],E)$ is $\aleph_0$-categorical.
	\end{lemma}
	
	\begin{proof}
		Let $\cM$ be a separable model of $\Th(L^\infty[0,1],E)$.  Then $\cM$ is a tracial $\mathrm{W}^*$-algebra.  Moreover, it is commutative since
		\[
		\sup_{X, Y \in D_1^{\cM}} d^{\cM}(XY, YX) = \sup_{X,Y \in D_1^{L^\infty[0,1]}} \norm{XY - YX}_{L^2[0,1]} = 0.
		\]
		To show that $\cM$ is diffuse, consider for each $n \in \N$ the sentence
		\[
		\inf_{x_1,\dots,x_n \in D_1^{\cM}} \left( d\left(x_1 + \dots + x_n, 1 \right) + \sum_{j=1}^n d(x_j^2,x_j) + \sum_{j=1}^n d(x_j,x_j^*) + \sum_{j=1}^n \left|\tau(x_j) - \frac{1}{n}\right| \right).
		\]
		The value of these sentences in $L^\infty[0,1]$ is $0$ since we can take $x_j = \mathbf{1}_{[(j-1)/n,j/n)}$.  Therefore, the value of these sentences in $\cM$ is also zero.  However, this could not happen if $\cM$ had atoms.
		
		Now we apply the fact that every diffuse separable commutative tracial $\mathrm{W}^*$-algebra is isomorphic to $(L^\infty[0,1],E)$ (see e.g. \cite{Sakai1971} or \cite[\S 3]{AP2017}).  %(See \S 2.4 in Ioana's article in this volume.)
	\end{proof}
	
	\begin{lemma} \label{lem:transformation}
		Let $\bX, \bY \in L^\infty[0,1]^d$.  If $\bX$ and $\bY$ have the same quantifier-free type, then for every $\epsilon > 0$, there exists an automorphism $\alpha$ of $(L^\infty[0,1],E)$ such that $\norm{\alpha(\bX) - \bY}_{L^2[0,1]^d} \leq \epsilon$.
	\end{lemma}
	
	\begin{proof}
		Fix $\bX$ and $\bY$.  Let $P$ denote the Lebesgue (probability) measure on $[0,1]$.  Fix $r > 0$ such that $\norm{X_j}_{L^\infty[0,1]} \leq r$ and $\norm{Y_j}_{L^\infty[0,1]} \leq r$.  Fix $\epsilon > 0$.  Let $S_1$, \dots, $S_k$ be a partition of $\{z \in \C: |z| \leq r\}^d$ into measurable sets each of which has diameter at most $\epsilon$.
		
		Since $\bX$ and $\bY$ have the same quantifier-free type (or probability distribution), we have $P(\bX \in S_j) = P(\bY \in S_j)$.  Thus, $\bX^{-1}(S_j)$ and $\bY^{-1}(S_j)$ are two Lebesgue-measurable subsets of $[0,1]$ with the same measure.  This implies (by a standard exercise in measure theory%; see Theorem 2.18 in Ioana's article in this volume
		) that there is a measurable isomorphism $f_j: \bX^{-1}(S_j) \to \bY^{-1}(S_j)$.  Let $f: [0,1] \to [0,1]$ be the measurable map given by $f|_{\bX^{-1}(S_j)} = f_j$.  Let $\alpha$ be the automorphism of $L^\infty[0,1]$ given by $Z \mapsto Z \circ f^{-1}$.
		
		Then $\alpha(\bX) = \bX \circ f^{-1}$ maps $\bY^{-1}(S_j)$ into $S_j$, and therefore, for $\omega \in \bY^{-1}(S_j)$, both $\alpha(\bX)(\omega)$ and $\bY(\omega)$ are in $S_j$, which implies that $\norm{\alpha(\bX)(\omega) - \bY(\omega)}_{\C^d} \leq \diam(S_j) \leq \epsilon$.  Therefore, $\norm{\alpha(\bX)(\omega) - \bY(\omega)}_{\C^d} \leq \epsilon$ for all $\omega \in [0,1]$ and in particular $\norm{\alpha(\bX) - \bY}_{L^2[0,1]^d} \leq \epsilon$.
	\end{proof}
	
	\begin{proof}[Proof of Theorem \ref{thm:commutativeQE}]
		Fix $d$ and let $\mathbb{S}_d(\Th(L^\infty[0,1],E))$ denote the space of types of $d$-tuples in models of the theory of $L^\infty([0,1],E)$.  Consider two types $\mu$ and $\nu$ in $\mathbb{S}_d(\Th(L^\infty[0,1],E))$ with $\pi_{\qf}(\mu) = \pi_{\qf}(\nu)$.  Let $\bX$ be a tuple from some model $\cM$ with type $\mu$ and let $\bY$ be a tuple from $\cN$ with type $\nu$.  By passing to separable elementary submodels through the L{\"o}wenheim-Skolem theorem, we may assume without loss of generality that $\cM$ and $\cN$ are separable.  Then by Lemma \ref{lem:commutativeisomorphism}, we may assume that $\cM = \cN = (L^\infty[0,1],E)$.
		
		Because $\bX$ and $\bY$ have the same quantifier-free type, Lemma \ref{lem:transformation} implies that for every $n \in \N$, there exists an automorphism $\alpha_n$ of $(L^\infty[0,1],E)$ such that $\norm{\alpha_n(\bX) - \bY}_{L^2[0,1]^d} < 1/n$.  Therefore,
		\[
		\tp^{L^\infty[0,1],E}(\bY) = \lim_{n \to \infty} \tp^{L^\infty[0,1],E}(\alpha_n(\bX)) = \tp^{L^\infty[0,1],E}(\bX).
		\]
		This implies that the map $\pi_{\qf}: \mathbb{S}_d(\Th(L^\infty[0,1],E)) \to \mathbb{S}_{\qf,d}(\Th(L^\infty[0,1],E))$ is injective for every $m$, and therefore, $(L^\infty[0,1],E)$ admits quantifier elimination by Lemma \ref{lem:typeQE}.
	\end{proof}
	
	We also remark that the matrix algebra $M_n(\C)$ admits quantifier elimination as a tracial $\mathrm{W}^*$-algebra, as was observed in \cite[end of \S 2]{EFKV2017}.  This fact can be proved by combining Lemma \ref{lem:typeQE} with the following result.
	
	\begin{lemma}
		Let $\bX$, $\bY \in M_n(\C)^d$.  The following are equivalent:
		\begin{enumerate}[(1)]
			\item $\bX$ and $\bY$ have the same type in $(M_n(\C),\tr_n)$.
			\item $\bX$ and $\bY$ have the same quantifier-free type $(M_n(\C),\tr_n)$.
			\item There exists a unitary $U$ such that $U \bX U^* = \bY$.
		\end{enumerate}
	\end{lemma}
	
	(1) $\implies$ (2) and (3) $\implies$ (1) are immediate. The claim (2) $\implies$ (3) follows from the multivariate version of Specht's theorem \cite{Specht1940} observed by Wiegmann \cite{Wiegmann1961} and verified in \cite[Theorem 2.2]{Jing2015}.  See \cite{Shapiro1991} for a survey of related results.
	
	However, there are many tracial $\mathrm{W}^*$-algebras that do not admit quantifier elimination; Goldbring, Hart, and Sinclair showed in \cite[Theorem 2.2]{GHS2013} that a tracial $\mathrm{W}^*$-algebra that is locally universal and McDuff cannot have quantifier elimination %(and in fact this applies to all McDuff tracial $\mathrm{W}^*$-algebras as %explained in Goldbring and Hart's article in this volume).
	The complete classification of $\mathrm{C}^*$-algebras with quantifier elimination was handled in \cite{EFKV2017}, so the general investigation of quantifier elimination for tracial $\mathrm{W}^*$-algebras seems like an achievable and worthwhile research project.
	
	\begin{question}
		Determine which separable tracial $\mathrm{W}^*$-algebras admit quantifier elimination.
	\end{question}
	
	\section{Free entropy for types} \label{sec:entropyfortypes}
	
	We define a version of Voiculescu's entropy for full types rather than only quantifier-free types.  We will see in \S \ref{sec:qfexists} that Voiculescu's free entropy of a $d$-tuple $\bX$ in the presence of $\cM$ can be realized as a special case of free entropy for a closed subset of the type space.
	
	\subsection{Definition}
	
	Because we only have one sort in $\mathcal{L}_{\tr}$, it will be convenient to use slightly different notation for the spaces of types.  Specifically, $\mathbb{S}_d(\rT_{\tr})$ will denote the space of types for $d$-tuples, or in other words $\mathbb{S}_{\bS}(\rT_{\tr})$ for $S = (M,\dots,M)$.  Moreover, $\mathbb{S}_{d,r}(\rT_{\tr})$ will denote the space of $d$-tuples with operator norm bounded by $r$, or in other words, $\mathbb{S}_{\bD}(\rT_{\tr})$ where $\bD = (D_r,\dots,D_r)$.  To simplify notation, we use the same $r$ for all $d$-variables rather than choosing radii $r_1$, \dots, $r_d$.
	
	\begin{definition}[Microstate spaces] \label{def:typemicrostates}
		If $\cK \subseteq \mathbb{S}_d(\rT_{\tr})$ and $r \in (0,\infty)^{\N}$, we define
		\[
		\Gamma_r^{(n)}(\cK) = \left\{\bY \in \prod_{j \in \N} D_r^{M_n(\C)}: \tp^{M_n(\C)}(\bY) \in \cK \right\}.
		\]
	\end{definition}
	
	We view this as a microstate space as in Voiculescu's free entropy theory.  The entropy will be defined in terms of Lebesgue measure of the microstate spaces, with the normalization of Lebesgue measure corresponding to the inner product from the normalized trace, as described in \S \ref{subsec:entropyintro}.  By transporting the canonical Lebesgue measure on $\C^{n^2d}$ through such an isometry, we obtain a canonical Lebesgue measure on $M_n(\C)^d$.  Note that this convention differs from that of Voiculescu since we use the normalized trace rather than the unnormalized trace to define the inner product.
	
	Moreover, here we consider all matrices rather than only self-adjoint matrices like Voiculescu, because the notion of type is more natural without restricting to the self-adjoint setting.  However, most of the results here would also work with the same proof for self-adjoint tuples, provided that we restrict to microstate spaces of self-adjoint matrices.
	
	\begin{definition}[Free entropy for types] \label{def:typeentropy}
		Let $\cK \subseteq \mathbb{S}_r(\rT_{\tr})$.  Then define
		\[
		\chi_{\full,r}^{\cU}(\cK) := \inf_{\text{open } \cO \supseteq \cK} \lim_{n \to \cU} \left( \frac{1}{n^2} \log \vol \Gamma_r^{(n)}(\cO) + 2d \log n \right),
		\]
		where $\cO$ ranges over all neighborhoods of $\cK$ in $\mathbb{S}_{d,r}(\rT_{\tr})$.  Moreover, we define
		\[
		\chi_{\full}^{\cU}(\cK) := \sup_{r > 0} \chi_{\full,r}^{\cU}(\cK).
		\]
	\end{definition}
	
	An important special case is when $\cK$ consists of a single point $\mu \in \mathbb{S}_d(\rT_{\tr})$.  We will write
	\[
	\chi_{\full,r}^{\cU}(\mu) = \chi_{\full,r}^{\cU}(\{\mu\}), \qquad \chi_{\full}^{\cU}(\mu) = \chi_{\full}^{\cU}(\{\mu\}).
	\]
	Looking at the free entropy of a single type $\mu$ is the analog of what Voiculescu originally did for non-commutative $*$-laws (quantifier-free types), since he defined the entropy for the non-commutative $*$-law of a tuple rather than more generally for a set of $*$-laws.  However, the variational principle (Proposition \ref{prop:variational}) in the next section will allow us to express the entropy of a closed set in terms of the entropy of individual types.
	
	We begin with a few immediate observations that will be useful in the sequel.
	
	\begin{observation}
		Since the map $\mathbb{S}_{d,r}(\rT_{\tr}) \to \mathbb{S}_d(\rT_{\tr})$ is a topological embedding, every open subset $\cO$ of the former is the restriction of an open set $\cO'$ in $\mathbb{S}_d(\rT_{\tr})$.  In fact, because $\mathbb{S}_{d,r}(\rT_{\tr})$ is closed in $\mathbb{S}_d(\rT_{\tr})$, we can take
		\[
		\cO' = \cO \cup [\mathbb{S}_d(\rT_{\tr}) \setminus \mathbb{S}_{d,r}(\rT_{\tr})].
		\]
		Thus, in the definition of $\chi_{\full,r}^{\cU}(\cK)$ in Definition \ref{def:typeentropy}, we may take $\cO$ to range over all open neighborhoods of $\cK$ in $\mathbb{S}_d(\rT_{\tr})$ or all open neighborhoods of $\cK \cap \mathbb{S}_{d,r}(\rT_{\tr})$ in $\mathbb{S}_{d,r}(\rT_{\tr})$, and the resulting quantity will be the same.
	\end{observation}
	
	\begin{observation}[Monotonicity] \label{obs:monotonicity}
		Note that if $\cK \subseteq \cK' \subseteq \mathbb{S}_d(\rT_{\tr})$, then $\Gamma_r^{(n)}(\cK) \subseteq \Gamma_r^{(n)}(\cK')$.  It follows from this that if $\cO \subseteq \mathbb{S}_{d,r}(\rT_{\tr})$ is open, then
		\[
		\chi_{\full,r}^{\cU}(\cO) = \lim_{n \to \cU} \left( \frac{1}{n^2} \log \vol \Gamma_r^{(n)}(\cO) + 2d \log n \right),
		\]
		that is, the infimum over neighborhoods of $\cO$ is achieved by $\cO$ itself.  Similarly, it follows that if $\cK \subseteq \cK'$, then $\chi_{\full,r}^{\cU}(\cK) \leq \chi_{\full,r}^{\cU}(\cK')$ and hence $\chi_{\full}^{\cU}(\cK) \leq \chi_{\full}^{\cU}(\cK')$.
	\end{observation}
	
	\begin{remark} \label{rem:radiusindependent}
		A similar argument as in \cite[Proposition 2.4]{VoiculescuFE2} shows that if $\cK \subseteq \mathbb{S}_{d,r}(\rT_{\tr})$, and if $r < r_1 < r_2$, then
		\[
		\chi_{\full,r_1}^{\cU}(\cK) = \chi_{\full,r_2}^{\cU}(\cK) = \chi_{\full,r}^{\cU}(\cK).
		\]
	\end{remark}

	\subsection{Variational principle}
	
	In this section, we show that the free entropy defines an upper semi-continuous function on the type space, and then deduce a variational principle for the entropy of a closed set, in the spirit of various results in the theory of entropy and large deviations.
	
	\begin{lemma}[Upper semi-continuity] \label{lem:USC}
		For each $r \in (0,\infty)$ and $d \in \N$, the function $\mu \mapsto \chi_{\full,r}^{\cU}(\mu)$ is upper semi-continuous on $\mathbb{S}_d(\rT_{\tr})$.
	\end{lemma}
	
	\begin{proof}
		Fix $r$.  For each open $\cO \subseteq \mathbb{S}_d(\rT_{\tr})$, let
		\[
		f_{\cO}(\mu) = \begin{cases} \chi_{\full,r}^{\cU}(\cO), & \mu \in \cO \\ \infty, & \text{otherwise.} \end{cases}
		\]
		Since $\cO$ is open, $f_{\cO}$ is upper semi-continuous.  Moreover, $\chi_{\full,r}^{\cU}(\mu)$ is the infimum of $f_{\cO}(\mu)$ over all open $\cO \subseteq \mathbb{S}(\rT_{\tr})$, and the infimum of a family of upper semi-continuous functions is upper semi-continuous.
	\end{proof}

	\begin{proposition}[Variational principle] \label{prop:variational}
		Let $\mathcal{K} \subseteq \mathbb{S}_d(\rT_{\tr})$ and let $r > 0$.  Then
		\begin{equation} \label{eq:variational1}
			\sup_{\mu \in \cK} \chi_{\full,r}^{\cU}(\mu) \leq \chi_{\full,r}^{\cU}(\cK) \leq \sup_{\mu \in \operatorname{cl}(\cK) } \chi_{\full,r}^{\cU}(\mu).
		\end{equation}
		Hence,
		\begin{equation} \label{eq:variational2}
			\sup_{\mu \in \cK} \chi_{\full}^{\cU}(\mu) \leq \chi_{\full}^{\cU}(\cK) \leq \sup_{\mu \in \operatorname{cl}(\cK) } \chi_{\full}^{\cU}(\mu).
		\end{equation}
		In particular, if $\cK$ is closed, then these inequalities are equalities.
	\end{proposition}
	
	\begin{proof}
		If $\mu \in \cK$, then by monotonicity $\chi_{\full,r}^{\cU}(\{\mu\}) \leq \chi_{\full,r}^{\cU}(\cK)$.  Taking the supremum over $\mu \in \cK$, we obtain the first inequality of \eqref{eq:variational1}.
		
		For the second inequality of \eqref{eq:variational1}, let $C = \sup_{\mu \in \operatorname{cl}(\cK)} \chi_{\full,r}^{\cU}(\mu)$.  If $C = \infty$, there is nothing to prove.  Otherwise, let $C' > C$.  For each $\mu \in \operatorname{cl}(\cK) \cap \mathbb{S}_{d,r}(\rT_{\tr})$, there exists some open neighborhood $\cO_\mu$ of $\mu$ in $\mathbb{S}_{d,r}(\rT_{\tr})$ such that $\chi_{\full,r}^{\cU}(\cO_\mu) < C'$.   Since $\{\cO_\mu\}_{\mu \in \operatorname{cl}(\cK) \cap \mathbb{S}_{d,r}(\rT_{\tr})}$ is an open cover of the compact set $\operatorname{cl}(\cK) \cap \mathbb{S}_{d,r}(\rT_{\tr})$, there exist $\mu_1$, \dots, $\mu_k \in \operatorname{cl}(\cK) \cap \mathbb{S}_{d,r}(\rT_{\tr})$ such that
		\[
		\operatorname{cl}(\cK) \cap \mathbb{S}_{r}(\rT_{\tr}) \subseteq \bigcup_{j=1}^k \cO_{\mu_j}.
		\]
		Let $\cO = \bigcup_{j=1}^k \cO_{\mu_j}$.  Then
		\[
		\vol \Gamma_r^{(n)}(\cO)) \leq \sum_{j=1}^k \vol \Gamma_r^{(n)}(\cO_{\mu_k}) \leq k \max_j \vol \Gamma_r^{(n)}(\cO_{\mu_j}).
		\]
		Thus,
		\[
		\frac{1}{n^2} \log \vol \Gamma_r^{(n)}(\cO)) + 2d \log n \leq \frac{1}{n^2} \log k + \max_j \frac{1}{n^2} \log \vol \Gamma_r^{(n)}(\cO_{\mu_j}) + 2d \log n.
		\]
		Taking the limit as $n \to \cU$,
		\[
		\chi_{\full,r}^{\cU}(\operatorname{cl}(\cK)) \leq  \chi_r^{\cU}(\cO) \leq \max_j \chi_{\full,r}^{\cU}(\cO_{\mu_j}) \leq C'.
		\]
		Since $C' > C$ was arbitrary,
		\[
		\chi_{\full,r}^{\cU}(\cK) \leq \chi_{\full,r}^{\cU}(\operatorname{cl}(\cK)) \leq C = \sup_{\mu \in \operatorname{cl}(\cK)} \chi_{\full,r}^{\cU}(\mu),
		\]
		completing the proof of \eqref{eq:variational1}.  Taking the supremum over $r$ in \eqref{eq:variational1}, we obtain \eqref{eq:variational2}.
	\end{proof}
	
	\begin{remark}
		If we assume that $\cK$ is a closed subset of $\mathbb{S}_{d,r}(\rT_{\tr})$, then $\chi^{\cU}(\mu) = \chi_r^{\cU}(\mu)$ for all $\mu \in \cK$.  Hence, by upper semi-continuity, the supremum in \eqref{eq:variational2} is a maximum. 
	\end{remark}
	
	\subsection{Entropy and ultraproduct embeddings}
	
	\begin{lemma}[Ultraproduct realization of types] \label{lem:ultraproductrealization}
		Let $\cQ = \prod_{n \to \cU} M_n(\C)$.  Let $\mu \in \mathbb{S}_{d}(\rT_{\tr})$. If $\chi_{\full}^{\cU}(\mu) > -\infty$, then there exists $\mathbf{X} \in Q^d$ such that $\tp^{\cQ}(\bX) = \mu$.
	\end{lemma}
	
	\begin{proof}
		If $\chi_{\full}^{\cU}(\mu) > -\infty$, then $\chi_{\full,r}^{\cU}(\mu) > -\infty$ for some $r \in (0,\infty)$.  Since $\mathbb{S}_{d,r}(\rT_{\tr})$ is metrizable, there is a sequence $(\cO_k)_{k \in \N}$ of neighborhoods of $\mu$ in $\mathbb{S}(\rT)$ such that $\overline{\cO}_{k+1} \subseteq \cO_k$ and $\bigcap_{k \in \N} \cO_k = \{\mu\}$.  For $k \in \N$, let
		\[
		E_k = \{n \in \N: n \geq k,  \Gamma_r^{(n)}(\cO_k) \neq \varnothing\}.
		\]
		Now choose $\mathbf{X}^{(n)} \in M_n(\C)^d$ as follows.  For each $n \not \in E_1$, set $\mathbf{X}^{(n)} = 0$.  For each $n \in E_k \setminus E_{k+1}$, let $\mathbf{X}^{(n)}$ be an element of $\Gamma_r^{(n)}(\cO_k)$.
		
		Since $\cU$ is an ultrafilter, either $E_k \in \cU$ or $E_k^c \in \cU$.  If we had $E_k^c \in \cU$, then $\lim_{n \to \cU} (1/n^2) \log \vol(\Gamma^{(n)}(\cO_k)) + 2d \log n$ would be $-\infty$ since $\Gamma^{(n)}(\cO_k)$ would be empty for $n \in E_k^c$.  Hence, $E_k \in \cU$.  For $n \in E_k$, we have $\tp^{M_n(\C)}(\mathbf{X}^{(n)}) \in \cO_k$.  Therefore, $\lim_{n \to \cU} \tp^{M_n(\C)}(\mathbf{X}^{(n)}) \in \overline{\cO}_k$.  Since this holds for all $k$, $\lim_{n \to \cU} \tp^{M_n(\C)}(\mathbf{X}^{(n)}) = \mu$.  Let $\mathbf{X} = [\mathbf{X}^{(n)}]_{n \in \N} \in Q^{\N}$.  Then
		\[
		\tp^{\cQ}(\mathbf{X}) = \lim_{n \to \cU} \tp^{M_n(\C)}(\mathbf{X}^{(n)}) = \mu.  \qedhere
		\]
	\end{proof}
	
	In particular, since $\tp^{\cM}(\bX)$ can be realized in $\cQ$, this implies that $\cM$ is elementarily equivalent to $\cQ$.  Then because $\cQ$ is countably saturated, the same reasoning as \cite[Lemma 4.12]{FHS2014} implies the following result, which will later play a role in the proof of Theorem \ref{thm:entropyembedding}.
	
	\begin{corollary} \label{cor:elementaryembeddings}
		Suppose that $\cM$ is a separable tracial $\mathrm{W}^*$-algebra and $\bX \in M^d$.  If $\chi_{\full}^{\cU}(\tp^{\cM}(\bX)) > -\infty$, then there exists an elementary embedding $\iota: \cM \to \cQ$.  Since the embedding $\iota: \cM \to \cQ$ is elementary, in particular $\tp^{\cQ}(\iota(\bX)) = \tp^{\cM}(\bX)$, and hence $\chi_{\full}^{\cU}(\tp^{\cQ}(\iota(\bX))) = \chi_{\full}^{\cU}(\tp^{\cM}(\bX))$.
	\end{corollary}
	
	\section{Entropy for quantifier-free and existential types} \label{sec:qfexists}
	
	In this section, we relate the free entropy for types $\chi_{\full}^{\cU}$ back to Voiculescu's original free entropy from \cite{VoiculescuFE2}, as well as free entropy in the presence from \cite{VoiculescuFE3}.  In particular, we show that Voiculescu's free entropy $\chi^{\cU}(\bX)$ is the version for quantifier-free types and the entropy in the presence $\chi^{\cU}(\bX:\cM)$ is the version for existential types.
	
	\subsection{Entropy for quantifier-free types}
	
	Voiculescu's definition of entropy for quantifier-free types is exactly analogous to Definitions \ref{def:typemicrostates} and \ref{def:typeentropy}.  Note that below we use the same notation $\Gamma_r^{(n)}(\cO)$ for microstate spaces associated to a neighborhood $\cO$ whether $\cO$ is a neighborhood in the space of full types or quantifier-free types; therefore, the reader should pay attention to which space the neighborhood $\cO$ comes from.
	
	\begin{definition}[Entropy for quantifier-free types]
		For $\cK \subseteq \mathbb{S}_{\qf,d}(\rT_{\tr})$ and $r \in (0,\infty)$, we define
		\[
		\Gamma_r^{(n)}(\cK) := \left\{\bX \in (D_r^{M_n(\C)})^d: \tp_{\qf}^{M_n(\C)}(\bX) \in \cK \right\}.
		\]
		Then we define for $r \in (0,\infty)$,
		\[
		\chi_{\qf,r}^{\cU}(\cK) := \inf_{\cO \supseteq \cK \text{ open in } \mathbb{S}_{\qf,d,r}(\rT_{\tr})} \lim_{n \to \cU} \left( \frac{1}{n^2} \log \vol \Gamma_r^{(n)}(\cO) + 2d \log n \right),
		\]
		and we set
		\[
		\chi_{\qf}^{\cU}(\cK) := \sup_{r \in (0,\infty)} \chi_{\qf,r}^{\cU}(\cK).
		\]
		For $\mu \in \mathbb{S}_{\qf}(\rT_{\tr})$, let $\chi_{\qf}^{\cU}(\mu) := \chi_{\qf}^{\cU}(\{\mu\})$.
	\end{definition}
	
	Voiculescu \cite[Definition 2.1]{VoiculescuFE2} originally defined $\chi_{\qf}$ in terms of particular open sets defined by looking at moments of order up to $m$ being within some distance $\epsilon$ of the moments of $\mu$; in other words, he used a neighborhood basis $\mu$ in $\mathbb{S}_{\qf,d,r}(\rT_{\tr})$ rather than all open sets.  The following lemma shows that our definition is equivalent.
	
	\begin{lemma} \label{lem:sequenceofneighborhoods}
		Let $r \in (0,\infty)$.  Let $\cK \subseteq \mathbb{S}_{\qf,d,r}(\rT_{\tr})$.  Let $(\cO_k)_{k \in \N}$ be a sequence of open subsets of $\mathbb{S}(\rT_{\tr})$ such that $\overline{\cO_{k+1}} \subseteq \cO_k$ and $\bigcap_{k=1}^\infty \cO_k = \cK$.  Then
		\[
		\chi_{\qf,r}^{\cU}(\cK) = \lim_{k \to \infty} \chi_{\qf,r}^{\cU}(\cO_k) = \inf_{k \in \N} \chi_{\qf,r}^{\cU}(\cO_k).
		\]
		The same holds with $\mathbb{S}_{\qf,d,r}(\rT_{\tr})$ replaced by $\mathbb{S}_{d,r}(\rT_{\tr})$ and $\chi_{\qf,r}^{\cU}$ replaced by $\chi_{\full,r}^{\cU}$.
	\end{lemma}
	
	\begin{proof}
		The argument for the case of quantifier-free types and the case of full types is the same, and we will write the argument for full types.
		By Observation \ref{obs:monotonicity},
		\[
		\chi_{\full,r}^{\cU}(\cK) \leq \chi_{\full,r}^{\cU}(\cO_{k+1}) \leq \chi_{\full,r}^{\cU}(\cO_k),
		\]
		so that
		\[
		\chi_{\full,r}^{\cU}(\cK) \leq \inf_{k \in \N} \chi_{\full,r}^{\cU}(\cO_k) = \lim_{k \to \infty} \chi_{\full,r}^{\cU}(\cO_k).
		\]
		For the inequality in the other direction, fix $\cO \supseteq \cK$ open.  Then $\mathbb{S}_{d,r}(\rT_{\tr}) \setminus \cO$ is closed and contained in $\cK^c = \bigcup_{k \in \N} \cO_k^c = \bigcup_{k \in \N} \overline{\cO}_k^c$.  By compactness, there is a finite subcollection of $\overline{\cO}_k^c$'s that covers $\mathbb{S}_{d,r}(\rT_{\tr}) \setminus \cO$.  The $\cO_k$'s are nested, so there exists some $k$ such that $\mathbb{S}_{d,r}(\rT_{\tr}) \setminus \cO \subseteq \overline{\cO_k}^c$, hence $\cO_k \cap \mathbb{S}_{d,r}(\rT_{\tr}) \subseteq \cO$.  Therefore,
		\[
		\inf_{k \in \N} \chi_{\full,r}^{\cU}(\cO_k) \leq \chi_{\full,r}^{\cU}(\cO).
		\]
		Since $\cO$ was an arbitrary neighborhood of $\cK$, we conclude that $\inf_{k \in \N} \chi_{\full,r}^{\cU}(\cO_k) \leq \chi_{\full,r}^{\cU}(\cK)$ as desired.
	\end{proof}
	
	This lemma also allows us to relate the entropy for quantifier-free types directly to the entropy for types.
	
	\begin{lemma} \label{lem:qfentropyrelationship}
		Let $\pi: \mathbb{S}_d(\rT_{\tr}) \to \mathbb{S}_{\qf,d}(\rT_{\tr})$ be the canonical restriction map.  Let $\cK \subseteq \mathbb{S}_{\qf,d}(\rT_{\tr})$ be closed.  Then
		\[
		\chi_{\qf}^{\cU}(\cK) = \chi_{\full}^{\cU}(\pi^{-1}(\cK)).
		\]
	\end{lemma}
	
	\begin{proof}
		Fix $r \in (0,\infty)^{\cN}$, and let $\cK_r = \cK \cap \mathbb{S}_{\qf,d,r}(\rT_{\tr})$.  Since $\mathbb{S}_{\qf,d,r}(\rT_{\tr})$ is metrizable and $\cK \subseteq \mathbb{S}_{\qf,d,r}$ is closed, there exists a sequence of open sets $\cO_k$ in $\mathbb{S}_{\qf,d,r}(\rT_{\tr})$ such that $\overline{\cO_{k+1}} \subseteq \cO_k$ and $\bigcap_{k \in \N} \cO_k = \cK_r$ (and these can be extended to open sets in $\mathbb{S}_{\qf,d}(\rT_{\tr})$ since the inclusion of $\mathbb{S}_{\qf,d,r}(\rT_{\tr})$ is a topological embedding).  Now $\pi^{-1}(\cO_k)$ is open in $\mathbb{S}_{d,r}(\rT_{\tr})$ and $\overline{\pi^{-1}(\cO_{k+1})} \subseteq \pi^{-1}(\overline{\cO_{k+1}}) \subseteq \pi^{-1}(\cO_k)$ and $\bigcap_{k \in \N} \pi^{-1}(\cO_k) = \pi^{-1}(\cK_r)$.  Note that $\Gamma_r^{(n)}(\cO_k) = \Gamma_r^{(n)}(\pi^{-1}(\cO_k))$.  Thus, using the previous lemma,
		\begin{align*}
			\chi_{\qf,r}^{\cU}(\cK) &= \chi_{\qf,r}^{\cU}(\cK_r) \\
			&= \inf_{k \in \N} \chi_{\qf,r}^{\cU}(\cO_k) \\
			&= \inf_{k \in \N} \chi_{\full,r}^{\cU}(\pi^{-1}(\cO_k)) \\
			&= \chi_{\full,r}^{\cU}(\pi^{-1}(\cK_r)) \\
			&= \chi_{\full,r}^{\cU}(\pi^{-1}(\cK)).
		\end{align*}
		Taking the supremum over $r$ completes the argument.
	\end{proof}
	
	In particular, by combining this with the variational principle (Proposition \ref{prop:variational}), we obtain the following corollary.
	
	\begin{corollary} \label{cor:qfvariational}
		Let $\pi: \mathbb{S}_d(\rT_{\tr}) \to \mathbb{S}_{\qf,d}(\rT_{\tr})$ be the restriction map.  If $\mu \in \mathbb{S}_{\qf,d}(\rT_{\tr})$, then
		\[
		\chi_{\qf}^{\cU}(\mu) = \sup_{\nu \in \pi^{-1}(\mu)} \chi_{\full}^{\cU}(\nu).
		\]
	\end{corollary}
	
	\begin{remark} \label{rem:qfmaximizer}
		Furthermore, if $r$ is large enough that $\mu \in \mathbb{S}_{\qf,d,r}$, then $\pi^{-1}(\mu) \subseteq \mathbb{S}_{d,r}$.  Hence, for $\nu \in \pi^{-1}(\mu)$, we have $\chi_{\full,r+1}^{\cU}(\nu) = \chi_{\full}^{\cU}(\mu)$ by Remark \ref{rem:radiusindependent}.  Finally, by Lemma \ref{lem:USC}, $\chi^{\cU}$ achieves a maximum on $\pi^{-1}(\cK)$.  It follows that
		\[
		\chi_{\qf}^{\cU}(\mu) = \max_{\nu \in \pi^{-1}(\mu)} \chi_{\full}^{\cU}(\nu).
		\]
	\end{remark}

	\subsection{Existential types}
	
	\begin{definition}
		Let $I$ be an index set, and consider variables $\mathbf{x} = (x_i)_{i \in I}$ from sorts $(S_i)_{i \in I}$.  An \emph{existential formula} in the language $\cL$ is a formula of the form
		\[
		\phi(\mathbf{x}) = \inf_{y_1 \in D_1, \dots, y_k \in D_k} \psi(\mathbf{x},y_1,\dots,y_k),
		\]
		where $\psi$ is a quantifier-free formula and $D_1$, \dots, $D_k$ are domains of quantification in the appropriate sorts.  Similarly, we say that $\phi$ is an \emph{existential definable predicate relative to $\rT$} if
		\[
		\phi^{\cM}(\bX) = \inf_{\bY \in \prod_{j \in \N} D_j^{\cM}} \psi^{\cM}(\bX,\bY)
		\]
		for $\cM \models \rT$, where $\psi$ is a quantifier-free definable predicate.
	\end{definition}
	
	Since a quantifier-free definable predicate can be approximated on each product of domains by a quantifier-free formula in finitely many variables, one can argue that an existential definable predicate can be approximated uniformly on each product of domains of quantification by an existential formula (and this is why it is a definable predicate to begin with).
	
	\begin{definition}
		Let $\cM$ be an $\cL$-structure, $\bS = (S_1,\dots,S_d)$ a $d$-tuple of sorts, and $\bX \in \prod_{j=1}^d S_j^{\cM}$.  Let $\cF_{\exists,\bS}$ denote the space of existential formulas.  The \emph{existential type} $\tp_{\exists}^{\cM}(\bX)$ is the map
		\[
		\tp_{\exists}^{\cM}(\bX): \cF_{\exists,\bS} \to \R, \phi \mapsto \phi^{\cM}(\bX).
		\]
		If $\rT$ is an $\cL$-theory, we denote the set of existential types that arise in models of $\rT$ by $\mathbb{S}_{\exists,\bS}(\rT)$.
	\end{definition}
	
	The topology for existential types, however, is not simply the weak-$*$ topology on $\mathbb{S}_{\exists,\bD}(\rT)$ for each tuple of domains.  Rather, we define neighborhoods of a type $\mu = \tp^{\cM}(\bX)$ using sets of the form $\{\nu: \nu(\phi) < \mu(\phi) + \epsilon\}$.  The idea is that if $\phi^{\cM}(\bX) = \inf_{\bY \in \prod_{j \in \N} D_j^{\cM}} \psi^{\cM}(\bX,\bY)$ for some quantifier-free definable predicate $\phi$, then $\mu(\phi) \leq c$ means that there exists $\bY$ such that $\psi^{\cM}(\bX,\bY) < c + \delta$ for any $\delta > 0$.  Thus, a neighborhood corresponds to types $\nu$ where there exists $\bY$ that gets within $\epsilon$ of the infimum achieved for $\mu$.
	
	\begin{definition} \label{def:existentialtopology}
		Let $\rT$ be an $\cL$-theory, $\bS$ an $d$-tuple of sorts, and $\bD \in \prod_{j=1}^d \cD_{S_j}$.  We say that $\cO \subseteq \mathbb{S}_{\exists,\bD}(\rT)$ is \emph{open} if for every $\mu \in \cO$, there exist existential formulas $\phi_1$, \dots, $\phi_k$ and $\epsilon_1$, \dots, $\epsilon_k > 0$ such that
		\[
		\{\nu \in \mathbb{S}_{\exists,\bD}(\rT): \nu(\phi_j) < \mu(\phi_j) + \epsilon_j \text{ for } j = 1, \dots, k\} \subseteq \cO.
		\]
		Moreover, we say that $\cO \subseteq \mathbb{S}_{\exists,\bS}(\rT)$ is open if $\cO \cap \mathbb{S}_{\exists,\bD}(\rT)$ is open in $\mathbb{S}_{\exists,\bD}(\rT)$ for all $\bD \in \prod_{j \in \N} \cD_{S_j}$.
	\end{definition}
	
	\begin{observation}~
		\begin{itemize}
			\item Any set of the form $\{\nu: \nu(\phi_1) < c_1, \dots, \nu(\phi_k) < c_k\}$, where $\phi_1$, \dots, $\phi_k$ are existential formulas, is open in $\mathbb{S}_{\exists,\bS}(\rT)$.
			\item The same holds if $\phi_j$ is an existential definable predicate rather than existential formula, since it can be uniformly approximated by existential formulas on each product of domains of quantification.  Hence existential definable predicates may be used in Definition \ref{def:existentialtopology} without changing the definition.
			\item The inclusion $\mathbb{S}_{\exists,\bD}(\rT) \to \mathbb{S}_{\exists,\bS}(\rT)$ is a topological embedding since each of the basic open sets in $\mathbb{S}_{\exists,\bD}(\rT)$ given by $\nu(\phi_j) < \mu(\phi_j) + \epsilon_j$ for $j = 1$, \dots, $k$ extends to an open set in $\mathbb{S}_{\exists,\bS}(\rT)$.
			\item The restriction map $\pi_{\exists}: \mathbb{S}_{\bS}(\rT) \to \mathbb{S}_{\exists,\bS}(\rT)$ is continuous.
		\end{itemize}
	\end{observation}
	
	\begin{remark}
		Like the Zariski topology on the space of ideals in a commutative ring, the topology on $\mathbb{S}_{\exists,\bS}(\rT)$ is often non-Hausdorff.  Indeed, one can check that the intersection of all neighborhoods of an existential type $\mu$ is
		\begin{equation} \label{eq:nbhd}
			\cK_\mu = \{\nu: \nu(\phi) \leq \mu(\phi) \text{ for all } \phi \in \cF_{\exists,\bS}\}.
		\end{equation}
		We say that $\nu$ \emph{extends} $\mu$ if $\nu(\phi) \leq \mu(\phi)$ for all existential formulas $\phi$, which is equivalent to saying that $\mu(\phi) = 0$ implies that $\nu(\phi) = 0$ for all $\phi \in \cF_{\exists,\bS}$ (since $\max(\phi - c,0)$ is an existential formula if $\phi$ is).  Then $\{\mu\}$ is closed if and only if it does have any proper extension, or it is \emph{maximal}.  These closed points correspond to existential types from existentially closed models (see \cite[\S 6.2]{Goldbring2021enforceable}), and such maximal existential types in $\mathbb{S}_{\exists,\bD}(\rT)$ form a compact Hausdorff space.  However, our present goal is to work with general tracial $\mathrm{W}^*$-algebras, not only those that are existentially closed.
	\end{remark}
	
	\subsection{Entropy for existential types}
	
	Here we define the entropy for existential types.  In the next subsection, we will show that this corresponds to Voiculescu's $\chi(\mathbf{X}:\cM)$.  Here $\mathbb{S}_{\exists,d}(\rT_{\tr})$ will denote the spaces of existential types for $d$-tuples from tracial $\mathrm{W}^*$-algebras and $\mathbb{S}_{\exists,d,r}(\rT_{\tr})$ will denote those arising from operators bounded in operator norm by $r$.
	
	\begin{definition}
		For $\cK \subseteq \mathbb{S}_{\exists,d}(\rT_{\tr})$ and $r \in (0,\infty)$, let $\Gamma_{r}^{(n)}(\cK) = \{ \bX \in \prod_{j \in \N} D_{r}^{M_n(\C)}: \tp_{\exists}^{M_n(\C)}(\bX) \in \cK\}$, and define
		\[
		\chi_{\exists,r}^{\cU}(\cK) = \inf_{\cO \supseteq \cK \text{ open}} \lim_{n \to \cU} \left( \frac{1}{n^2} \log \vol \Gamma_{r}^{(n)}(\cO) + 2d \log n \right).
		\]
		Then let
		\[
		\chi_{\exists}^{\cU}(\cK) = \sup_{r \in (0,\infty)} \chi_{\exists,r}^{\cU}(\cK).
		\]
	\end{definition}
	
	Because of the non-Hausdorff nature of $\mathbb{S}_{\exists}(\rT_{\tr})$, we will restrict our attention to an individual existential type rather than for a closed set of existential types.
	
	\begin{lemma} \label{lem:existentialversusfull}
		Let $\mu \in \mathbb{S}_{\exists,d}(\rT_{\tr})$.  Let $\pi_\exists: \mathbb{S}(\rT_{\tr}) \to \mathbb{S}_{\exists,d}(\rT_{\tr})$ be the canonical restriction map.  Then
		\[
		\chi_{\exists}^{\cU}(\mu) = \chi_{\full}^{\cU}\left(\pi^{-1}\left(\cK_\mu\right) \right) =  \max_{\nu \in \pi^{-1}(\cK_\mu)} \chi_{\full}^{\cU}(\nu),
		\]
		where $\cK_\mu$ is given by \eqref{eq:nbhd}.
	\end{lemma}
	
	\begin{proof}
		Fix $r \in (0,\infty)$.  If $\cO$ is a neighborhood of $\mu$ in $\mathbb{S}_{\exists,d}(\rT_{\tr})$, then it contains $\cK_\mu$, and hence $\pi_\exists^{-1}(\cO)$ is a neighborhood of $\pi^{-1}(\cK_\mu)$ in $\mathbb{S}(\rT_{\tr})$.  Moreover, $\Gamma_r^{(n)}(\cO) = \Gamma_r^{(n)}(\pi_\exists^{-1}(\cO))$, hence
		\[
		\chi_{\full,r}^{\cU}\left(\pi_\exists^{-1}\left( \cK_\mu \right) \right) \leq \chi_{\exists,r}^{\cU}(\mu).
		\]
		
		It remains to show the reverse inequality.  Since the space of definable predicates on $D_r^d$ relative to $\rT_{\tr}$ is separable with respect to the uniform metric, so is the space of existential definable predicates.  Let $(\phi_j)_{j \in \N}$ be a sequence of existential definable predicates that are dense in this space.  Let
		\[
		\cO_k = \left\{\nu \in \mathbb{S}_{\exists,d,r}(\rT_{\tr}): \nu(\phi_j) < \mu(\phi_j) + \frac{1}{k} \text{ for } j \leq k \right\}.
		\]
		Note that
		\[
		\bigcap_{k \in \N} \cO_k = \{\nu \in \mathbb{S}_{\exists,d,r}(\rT_{\tr}): \nu(\phi_k) \leq \mu(\phi_k) \text{ for } k \in \N\} = \cK_\mu.
		\]
		Moreover,
		\[
		\overline{\pi_\exists^{-1}(\cO_{k+1})} \subseteq \left\{\nu \in \mathbb{S}_d(\rT_{\tr}): \nu(\phi_j) \leq \mu(\phi_j) + \frac{1}{k+1} \text{ for } j \leq k + 1 \right\} \subseteq \pi_\exists^{-1}(\cO_k).
		\]
		Therefore, by Lemma \ref{lem:sequenceofneighborhoods} applied to $\pi_\exists^{-1}(\cO_k)$, we have
		\[
		\chi_{\exists,r}^{\cU}(\mu) \leq \inf_{k \in \N} \chi_{\exists,r}^{\cU}(\cO_k) = \inf_{k \in \N} \chi_{\full,r}^{\cU}(\pi_\exists^{-1}(\cO_k)) = \chi_r^{\cU}(\pi_\exists^{-1}(\cK)),
		\]
		where the last equality follows from the density of $\{\phi_k: k \in \N\}$.  Thus, $\chi_{\exists,r}^{\cU}(\mu) = \chi_{\full,r}^{\cU}(\pi_\exists^{-1}(\cK_\mu))$.  Taking the supremum over $r$ yields the first asserted equality $\chi_{\exists}^{\cU}(\mu) = \chi_{\full}^{\cU}(\pi_\exists^{-1}(\cK_\mu))$.  By applying the variational principle (Proposition \ref{prop:variational}) to the closed set $\pi_\exists^{-1}(\cK_\mu)$, this is equal to $\sup_{\nu \in \pi_\exists^{-1}(\cK_\mu)} \chi_{\full}^{\cU}(\nu)$.  Then the same argument as in Remark \ref{rem:qfmaximizer} shows that the supremum is in fact a maximum.
	\end{proof}
	
	\subsection{Existential entropy and entropy in the presence} \label{subsec:inthepresence}
	
	Let us finally explain why the existential entropy defined here agrees with (the ultrafilter version of) Voiculescu's $\chi^{\cU}(\mathbf{X},\cM)$.  The definition is given in terms of Voiculescu's microstate spaces for some $d$-tuple $\bX$ in the presence of some $m$-tuple $\bY$.
	
	\begin{definition}[{Voiculescu \cite[Definition 1.1]{VoiculescuFE3}}] \label{def:Voiculescupresence}
		Let $\cM$ be a tracial $\mathrm{W}^*$-algebra.  Let $d$, $m \in \N$ and let $\bX \in M^d$ and $\bY \in M^m$.  Let $r \in (0,\infty)$.  Let $\mathbb{S}_{\qf,d+m,r}(\rT_{\tr})$ be the set of quantifier-free types of tuples from $\prod_{i \in I} D_{r} \times \prod_{j \in J} D_r$ equipped with the weak-$*$ topology.  Let $p: M_n(\C)^{d+m} \to M_n(\C)^d$ be the canonical projection onto the first $d$ coordinates.  Then we define
		\[
		\chi_r^{\cU}(\bX:\bY) := \inf_{\cO \ni \tp_{\qf}^{\cM}(\bX,\bY)} \lim_{n \to \cU} \left( \frac{1}{n^2} \log \vol p[\Gamma_r^{(n)}(\cO)] + 2d \log n \right),
		\]
		where $\cO$ ranges over all neighborhoods of $\tp_{\qf}^{\cM}(\bX,\bY)$ in $\mathbb{S}_{\qf,d+m}(\rT_{\tr})$.  Then set
		\[
		\chi^{\cU}(\bX:\bY) := \sup_{r > 0} \chi_r^{\cU}(\bX:\bY).
		\]
		Finally, we define
		\[
		\chi^{\cM}(\bX: \cM) = \inf_{m \in \N} \inf_{\bY \in M^m} \chi^{\cU}(\bX: \bY).
		\]
	\end{definition}
	
	\begin{remark}
		It follows from \cite[Proposition 1.6]{VoiculescuFE3} that for finite tuples $\bX$, $\bY$, and $\bZ$,
		\[
		\chi^{\cU}(\bX: \bY,\bZ) \leq \chi^{\cU}(\bX: \bY)
		\]
		with equality when $\bZ$ comes from $\mathrm{W}^*(\bX,\bY)$.  As a consequence,
		\[
		\chi^{\cU}(\bX: \mathrm{W}^*(\bX,\bY)) = \chi^{\cU}(\bX:\bY).
		\]
		In fact, by a similar slightly more technical argument one can show that if $\cM$ is generated by $\bX$ and $Y_1$, $Y_2$, \dots, then
		\[
		\chi^{\cU}(\bX:\cM) = \lim_{k \to \infty} \chi^{\cU}(\bX:Y_1,\dots,Y_k).
		\]
	\end{remark}
	
	The idea behind why $\chi^{\cU}(\bX:\cM)$ corresponds to the entropy of $\tp_{\exists}^{\cM}(\bX)$ is that a matrix tuple $\bX'$ is in the projection $p[\Gamma_r^{(n)}(\cO)]$ if and only if \emph{there exists} some $\bY'$ such that $\tp_{\qf}^{M_n(\C)}(\bX',\bY') \in \cO$.  If $\bX', \bY'$ being in $\Gamma_r^{(n)}(\cO)$ can be detected by a quantifier-free formula being less than some $c$, then $\bX'$ being in $p[\Gamma_r^{(n)}(\cO)]$ can be detected by an existential formula.
	
	\begin{proposition} \label{prop:presenceequivalence}
		In the setup of Definition \ref{def:Voiculescupresence}, we have $\chi^{\cU}(\bX: \cM) = \chi_{\exists}^{\cU}(\tp_{\exists}^{\cM}(\bX))$.
	\end{proposition}
	
	\begin{proof}
		First, let us show that $\chi^{\cU}(\bX:\cM) \leq \chi_{\exists}^{\cU}(\tp_{\exists}^{\cM}(\bX))$.  Fix $r \in (0,\infty)$ with $\norm{X_j} < r$.  Let $\cO$ be a neighborhood of $\tp_{\exists}^{\cM}(\bX)$ in $\mathbb{S}_{\exists,d,r}(\rT_{\tr})$.  Then there exist existential definable predicates $\phi_1$, \dots, $\phi_k$ and $\epsilon_1$, \dots, $\epsilon_k > 0$ such that
		\[
		\{\nu \in \mathbb{S}_{\exists,d,r}(\rT_{\tr}): \nu(\phi_j) \leq \mu(\phi_j) + \epsilon_j \text{ for } j = 1, \dots, k\} \subseteq \cO.
		\]
		Since existential formulas are dense in the space of existential definable predicates, we can assume without loss of generality that the $\phi_j$'s are existential formulas, so there exist quantifier-free formulas $\psi_1$, \dots, $\psi_k$ such that
		\[
		\phi_j^{\cN}(\bX') = \inf_{\bY' \in \prod_{i=1}^{m_j} D_{t_{i,j}}^{\cN}} \psi_j^{\cN}(\bX',\bY') \text{ for all } \cN \models \rT_{\tr} \text{ and } \bX' \in N^{\N}.
		\]
		By rescaling the variables $\bY'$ in these formulas, we can assume without loss of generality that $t_{i,j} < r$.  Moreover, for our particular $\cM$ and $\bX$, there exists $\bY_j \in \prod_{i=1}^{m_j} D_{t_{i,j}}^{\cM}$ such that
		\[
		\psi_j^{\cM}(\bX,\bY_j) < \mu(\phi_j) + \epsilon_j.
		\]
		Then
		\[
		\cO' := \{\tp_{\qf}^{\cN}(\bX',\bY_1',\dots,\bY_k') \in \mathbb{S}_{\qf,d+m_1+\dots+m_k,r'}(\rT_{\tr}): \psi_j(\bX',\bY_j') < \mu(\phi_j) + \epsilon_j \text{ for } j = 1, \dots, k \}
		\]
		is a neighborhood of $\tp_{\qf}^{\cM}(\bX,\bY_1,\dots,\bY_k)$ in $\mathbb{S}_{\qf,d+m_1+\dots+m_k,r}(\rT_{\tr})$ such that
		\[
		p[\Gamma_r^{(n)}(\cO')] \subseteq \Gamma_r^{(n)}(\cO).
		\]
		Therefore,
		\[
		\vol p[\Gamma_r^{(n)}(\cO')] \leq \vol \Gamma_r^{(n)}(\cO),
		\]
		which implies
		\[
		\chi_r^{\cU}(\bX:\bY_1,\dots,\bY_k) \leq \lim_{n \to \cU} \left( \frac{1}{n^2} \log \vol p[\Gamma_r^{(n)}(\cO)] + 2d \log n \right) \leq \chi_{\exists,r}^{\cU}(\cO).
		\]
		By the same reasoning as \cite[Proposition 2.4]{VoiculescuFE2}, since $r > \norm{X_j}_\infty$ and $r > \norm{Y_{i,j}}_\infty$, we have
		\[
		\chi^{\cU}(\bX:\cM) \leq \chi^{\cU}(\bX:\bY_1,\dots,\bY_k) = \chi_r^{\cU}(\bX:\bY_1,\dots,\bY_k) \leq \chi_{\exists,r}^{\cU}(\cO).
		\]
		Since $\cO$ was an arbitrary neighborhood of $\tp_{\exists}^{\cM}(\bX)$, we have
		\[
		\chi^{\cU}(\bX:\cM) \leq \chi_{\exists,r}^{\cU}(\tp_{\exists}^{\cM}(\bX)) \leq \chi_{\exists}^{\cU}(\tp_{\exists}^{\cM}(\bX)).
		\]
		
		To prove the second inequality, fix $m$ and an $m$-tuple $\bY \in M^m$.  Let $\cO'$ be a neighborhood of $\tp_{\qf}^{\cM}(\bX,\bY)$.  By definition of the weak-$*$ topology on $\Sigma_{\qf,d+m,r}$, there exist quantifier-free formulas $\psi_1$, \dots, $\psi_k$ and $\epsilon_1$, \dots, $\epsilon_k > 0$ such that
		\[
		\{\sigma \in \mathbb{S}_{\qf,d+m,r}: |\sigma(\psi_j) - \psi_j^{\cM}(\bX,\bY)| < \epsilon_j \text{ for } j = 1, \dots, k\} \subseteq \cO'.
		\]
		Now consider the existential formula
		\[
		\phi^{\cN}(\bX') = \inf_{\bY' \in (D_r^{\cN})^m} \max_{j=1,\dots,k} \left( \frac{|\psi_j^{\cN}(\bX',\bY') - \psi_j^{\cM}(\bX,\bY)|}{\epsilon_j} \right)
		\]
		for every tracial $\mathrm{W}^*$-algebra $\cN$.  Note that $\phi^{\cN}(\bX') < 1$ implies that there exists some tuple $\bY'$ from $(D_r^{\cN})^m$ with $|\psi_j^{\cN}(\bX',\bY') - \psi_j^{\cM}(\bX,\bY)| < \epsilon_j$ for $j = 1, \dots, k$ and hence $\tp_{\qf}^{\cN}(\bX',\bY') \in \cO'$.  In particular, the set
		\[
		\cO := \{\mu \in \mathbb{S}_{\exists,d,r}(\rT_{\tr}): \mu(\phi) < 1\}
		\]
		is an open subset of $\mathbb{S}_{\exists,d,r}(\rT_{\tr})$ such that
		\[
		\Gamma_r^{(n)}(\cO) \subseteq p[\Gamma_r^{(n)}(\cO')].
		\]
		Because for every $\cO'$ there exists such an $\cO$, we obtain $\chi_{\exists,r}^{\cU}(\tp_{\exists}^{\cM}(\bX)) \leq \chi_r^{\cU}(\bX:\bY)$.  Then we take the supremum over $r$ and the infimum over $m$ and $\bY$ to complete the argument.
	\end{proof}
	
	\subsection{Proof of {Theorem \ref{thm:entropyembedding}}}
	
	Now we complete the proof of Theorem \ref{thm:entropyembedding}.  We state the theorem here with slight additional claims that we could not state earlier because we had not yet defined entropy for types.
	
	\begin{theorem}
		Let $\cM$ be a separable tracial $\mathrm{W}^*$-algebra and $\bX \in M^d$.  If $\chi^{\cU}(\bX: \cM) > -\infty$, then there exists an embedding of $\iota: \cM \to \cQ = \prod_{n \to \cU} M_n(\C)$ such that
		\[
		\chi^{\cU}(\iota(\bX): \cQ) = \chi_{\full}^{\cU}(\tp^{\cQ}(\iota(\bX))) = \chi^{\cU}(\bX: \cM).
		\]
	\end{theorem}
	
	\begin{proof}
		Let $\pi_{\exists}$ be the projection from full types to existential types, and let $\mu = \tp_{\exists}^{\cM}(\bX)$.  By Lemma \ref{lem:existentialversusfull}, there exists $\nu \in \pi_{\exists}^{-1}(\cK_\mu)$ such that $\chi^{\cU}(\nu) = \chi_\exists^{\cU}(\mu)$.  By Lemma \ref{lem:ultraproductrealization}, there exists $\bX' \in \cQ^d$ with $\tp^{\cQ}(\bX') = \nu$.
		
		Let $\bY \in \prod_{j \in \N} D_R^{\cM}$ be an $\N$-tuple that generates $\cM$, and write $(\bX,\bY) = (X_1,\dots,X_d,Y_1,Y_2,\dots)$.  Since the space of quantifier-free types $\mathbb{S}_{\qf,\N,R}(\rT_{\tr})$ for $\N$-tuples is compact and metrizable, there exists a quantifier-free definable predicate $\phi \geq 0$ such that for tracial $\mathrm{W}^*$-algebras $\cN$ and for $\bZ \in \prod_{j=1}^d D_R^{\cN}$ and $\bW \in \prod_{j \in \N} D_R^{\cN}$, we have  $\phi^{\cN}(\bZ,\bW) = 0$ if and only if $\tp_{\qf}^{\cN}(\bZ,\bW) = \tp_{\qf}^{\cM}(\bX,\bY)$.  Note that
		\[
		\psi^{\cN}(\bZ) = \inf_{\bW \in \prod_{j \in \N} D_R^{\cN}} \phi(\bZ,\bW) 
		\]
		is an existential definable predicate over $\rT_{\tr}$ in the variables $Z_1,\dots,Z_d$, and
		\[
		0 \leq \psi^{\cQ}(\bX') \leq \psi^{\cM}(\bX) \leq \phi^{\cM}(\bX,\bY) = 0,
		\]
		where we have applied the fact that $\tp_{\exists}(\bX') = \pi_{\exists}(\nu) \in \cK_\mu$.
		
		Now we may write $\bX' = [\bX^{(n)}]_{n \in \N}$ where $\bX^{(n)} \in \prod_{j \in \N} D_{r}^{M_n(\C)}$ (this is well-known, and it follows for instance from the construction of $\bX'$ through Lemma \ref{lem:ultraproductrealization}).  Then
		\[
		\lim_{n \to \cU} \psi^{M_n(\C)}(\bX^{(n)}) = \psi^{\cQ}(\bX') = 0,
		\]
		hence there exists $\bY^{(n)} \in \prod_{j \in \N} D_{r'}^{M_n(\C)}$ such that
		\[
		\lim_{n \to \cU} \phi^{M_n(\C)}(\bX^{(n)},\bY^{(n)}) = 0.
		\]
		Let $\bY' = [\bY^{(n)}]_{n \in \N} \in \prod_{j \in \N} D_{r_j'}^{\cQ}$.  Then $\phi^{\cQ}(\bX',\bY') = 0$, and therefore, $\tp_{\qf}^{\cQ}(\bX',\bY') = \tp_{\qf}^{\cM}(\bX',\bY')$.  Hence, there exists an embedding $\iota: \cM \to \cQ$ with $\iota(\bX,\bY) = (\bX',\bY')$.
		
		By Lemma \ref{lem:existentialversusfull} and Proposition \ref{prop:presenceequivalence}, we have
		\[
		\chi_{\full}^{\cU}(\nu) = \chi_{\full}^{\cU}(\tp^{\cQ}(\bX')) \leq \chi_\exists^{\cU}(\tp_{\exists}^{\cQ}(\bX')) = \chi^{\cU}(\bX': \cQ) = \chi_\exists^{\cU}(\pi_\exists(\nu)).
		\]
		Meanwhile, since $K_{\pi_\exists(\nu)} \subseteq \cK_\mu$,
		\[
		\chi_\exists^{\cU}(\pi_\exists(\nu)) \leq \chi_{\exists}^{\cU}(\mu) = \chi^{\cU}(\bX: \cM).
		\]
		By our choice of $\nu$, $\chi_{\full}^{\cU}(\nu) = \chi_{\exists}^{\cU}(\mu)$, so all these inequalities are equalities.
	\end{proof}
	
	\begin{remark}
		We mentioned in \ref{subsec:entropyintro} that one can prove the existence of an embedding into a matrix ultraproduct that has trivial relative commutant by showing that a randomly chosen embedding has this property.  This is an example of a probabilistic argument--showing that there exists some $x$ in a given set satisfying some conditions by showing that a random $x$ satisfies these conditions with high (or at least positive) probability.  Our proof of Theorem \ref{thm:entropyembedding} is in spirit also a probabilistic argument.  The variational principle shows that there must be some $\mu \in \pi_{\exists}^{-1}(\cK_\mu)$ with entropy equal to $\chi_{\exists}^{\cU}(\mu)$.  Unwinding the proof of Proposition \ref{prop:variational}, our intuition is that an existential microstate space for $\mu$ can be covered by microstate spaces for finitely many $\nu \in \pi_{\exists}^{-1}(\mu)$ and so there must be some $\nu$ which carries part of the mass of the microstate space for $\mu$ that is not exponentially small compared to the whole.
	\end{remark}
	
	\section{Toward a notion of independence} \label{sec:independence}
	
	In this section, we explore possible ways of adapting free independence to the model-theoretic framework for tracial $\mathrm{W}^*$-algebras.  We state many related open questions at the intersection of random matrix theory and model theory.
	
	\subsection{Non-commutative probability viewpoint}
	
	The notion of free independence mentioned in the introduction leads to the notion of free convolution analogous to the convolution of two probability measures on $\R^d$ in classical probability.  Suppose we have two $*$-laws $\mu$, $\nu \in \mathbb{S}_{\qf,d}(\rT_{\tr})$.  If $\bX$ and $\bY$ are $d$-tuples from $\cA$ and $\cB$ with $\tp_{\qf}^{\cA}(\bX) = \mu$ and $\tp_{\qf}^{\cB}(\bY) = \nu$, then it turns out that $\tp_{\qf}^{\cA * \cB}(\bX + \bY)$ is uniquely determined by $\mu$ and $\nu$.  This quantifier-free type $\tp_{\qf}^{\cA * \cB}(\bX + \bY)$ is called the \emph{(additive) free convolution} of $\mu$ and $\nu$ and it is denoted $\mu \boxplus \nu$.
	
	\begin{question}
		Is there a model-theoretic analog of free independence?  Is there a notion of free convolution for full types?
	\end{question}
	
	A na{\"\i}ve approach to constructing free convolution for types would be as follows.  Consider types $\mu$ and $\nu$.   Let $\bX$ and $\bY$ be $d$-tuples from $\cA$ and $\cB$ respectively with $\tp^{\cA}(\bX) = \mu$ and $\tp^{\cB}(\bY) = \nu$.  Then we may try to define $\tp^{\cA}(\bX) + \tp^{\cB}(\bY)$ as $\tp^{\cA * \cB}(\bX + \bY)$; however, it is not clear whether this is well-defined.
	
	\begin{question}
		Suppose that $\tp^{\cA_1}(\bX) = \tp^{\cA_2}(\bX')$ and $\tp^{\cB_1}(\bY) = \tp^{\cB_2}(\bY')$.  Then does $\tp^{\cA_1 * \cB_1}(\bX + \bY) = \tp^{\cA_2 * \cB_2}(\bX' + \bY')$?
	\end{question}
	
	This in turn provokes an even more basic question about the relationship between free products and model theory.  Observe that $\tp^{\cA_1}(\bX) = \tp^{\cA_2}(\bX')$ implies that $\cA_1 \equiv \cA_2$ because every sentence can be considered as a formula in the free variables $x_1, \dots, x_d$ (which happens to be independent of those variables).
	
	\begin{question}
		Do free products preserve elementary equivalence?  In other words, if $\cA_1 \equiv \cA_2$ and $\cB_1 \equiv \cB_2$, then does it necessarily follow that $\cA_1 * \cA_2 \equiv \cB_1 * \cB_2$?
	\end{question}
	
	The analogous question for free products of groups was highly nontrivial, and was answered in the affirmative by Sela \cite[Theorem 7.1]{Sela2010}.  (One can also ask the analogous question but with tensor products rather than free products.)
	
	\subsection{Random matrix viewpoint}
	
	However, it is unclear whether the free product is even the `right' construction of independence for types from a random matrix viewpoint.  In the quantifier-free setting, free independence can be characterized through random matrix theory as follows.   Recall that if $X_n$ is a random variable on a probability space $(\Omega_n,\cF_n,P_n)$ with values in a topological space $\mathcal{X}$ and if $x \in \mathcal{X}$, we say that $X_n$ \emph{converges to $x$ in probability} if for every neighborhood $\cO$ of $x$, we have $P_n(X_n \in \cO) = P_n(X_n^{-1}(\cO)) \to 1$.  Similarly, if $X_n$ is a real random variable and $x \in \R$, we say that $\limsup_{n \to \infty} X_n \leq x$ in probability if for every $\epsilon > 0$, we have $\lim_{n \to \infty} P_n(X_n < x + \epsilon) = 1$.
	
	\begin{theorem}[{Voiculescu \cite{Voiculescu1991,Voiculescu1998}}] \label{thm:asymptoticfreeness}
		Let $\bX$ and $\bY$ be $d$-tuples in a tracial $\mathrm{W}^*$-algebra $\cM$ that are freely independent.  Let $\bX^{(n)}$ and $\bY^{(n)}$ be random variables in $M_n(\C)^d$ such that
		\begin{enumerate}[(1)]
			\item $\tp_{\qf}^{M_n(\C)}(\bX^{(n)}) \to \tp_{\qf}^{\cM}(\bX)$ and $\tp_{\qf}^{M_n(\C)}(\bY^{(n)}) \to \tp_{\qf}^{\cM}(\bY)$ in probability (as random elements of the space $\mathbb{S}_{\qf,d}(\rT_{\tr})$).
			\item For some constant $C$, we have $\limsup_{n \to \infty} \norm{X_j^{(n)}} \leq C$ and $\limsup_{n \to \infty} \norm{Y_j^{(n)}} \leq C$ in probability.
			\item For each $n \in \N$, $\bX^{(n)}$ and $\bY^{(n)}$ are classically independent random variables.
			\item The probability distributions of $\bX^{(n)}$ and of $\bY^{(n)}$ are invariant under conjugation by each $n \times n$ unitary matrix $U$.
		\end{enumerate}
		Then $\tp_{\qf}^{M_n(\C)}(\bX^{(n)},\bY^{(n)}) \to \tp_{\qf}^{\cM}(\bX,\bY)$ in probability.  Thus, in particular, $\tp_{\qf}^{M_n(\C)}(\bX^{(n)} + \bY^{(n)}) \to \tp_{\qf}^{\cM}(\bX) \boxplus \tp_{\qf}^{\cM}(\bY)$.
	\end{theorem}
	
	\begin{question} \label{q:asymptoticfreeness}
		Is there some analog of Theorem \ref{thm:asymptoticfreeness} for full types rather than only quantifier-free types?
	\end{question}
	
	There are several issues in addressing this broad question.  The first is whether we want to use a limit as $n \to \infty$ or an ultralimit as $n \to \cU$.  In order for $\tp^{M_n(\C)}(\bX^{(n)})$ to converge as $n \to \infty$ to $\tp^{\cM}(\bX)$, it would be necessary for $\Th(M_n(\C)) \to \Th(\cM)$ as $n \to \infty$.  Indeed, every sentence is a formula with no free variables which can also be viewed as a formula in variables $x_1$, \dots, $x_d$.  Thus, the type of a $d$-tuple includes the information of the theory of the tracial $\mathrm{W}^*$-algebra that they came from.  This question was asked by Popa and also by Farah, Hart, and Sherman in \cite{FHS2014}.
	
	\begin{question} \label{q:matrixtheoryconvergence}
		Does $\lim_{n \to \infty} \Th(M_n(\C))$ exist, or in other words, for each formula $\phi$ with no free variables, does $\lim_{n \to \infty} \phi^{M_n(\C)}$ exist?  Equivalently, are $\prod_{n \to \cU} M_n(\C)$ and $\prod_{n \to \cV} M_n(\C)$ elementarily equivalent for all free ultrafilters $\cU$ and $\cV$ on $\N$?
	\end{question}
	
	This question is still wide open.  Many of the tools used in asymptotic random matrix theory to prove results like Theorem \ref{thm:asymptoticfreeness} work for every large value of $n$.  As such, they are not suited to distinguish between $\Th(M_n(\C))$ and $\Th(M_m(\C))$ for large values of $n$ and $m$.  However, perhaps group representations or quantum games could be used if one wants to show that matrix ultraproducts are not elementarily equivalent.
	
	To circumvent this question for the moment, let us consider limits with respect to a fixed ultrafilter $\cU$.  The second issue in Question \ref{q:asymptoticfreeness} is that hypothesis (4) of unitary invariance is not strong enough to make an analog of Theorem \ref{thm:asymptoticfreeness} for full types.  We illustrate this with Example \ref{ex:unitaryinvarianceinsufficient} below.  We will need the following lemma, which is an adaptation of the result that if all embeddings of $\cA$ into $\cQ$ are unitarily conjugate, then the $1$-bounded entropy $h(\cA)$ is less than or equal to zero, which implies that $\chi^{\cU}(\bX:\cA) = -\infty$ for every tuple $\bX$ from $\cA$ (see \cite{Hayes2018}); the argument is similar to \cite{VoiculescuFE3} and \cite{Hayes2018}.
	
	\begin{lemma}
		Let $\mu \in \mathbb{S}_d(\rT_{\tr})$ be a type realized in $\cQ$.  Let $\cQ = \prod_{n \to \cU} M_n(\C)$.  Suppose that for every $\bX$ and $\bY \in Q^d$ realizing $\mu$, there exists a unitary $U \in Q$ such that $U \bX U^* = \bY$, that is, $UX_jU^* = Y_j$ for each $j$.  Then $\chi^{\cU}(\mu) = -\infty$.
	\end{lemma}
	
	\begin{proof}
		Fix a large $r > 0$.  Let
		\[
		\cK = \{\tp^{\cQ}(\bX,\bY): \bX, \bY \in (D_r^{\cQ})^d, \tp^{\cQ}(\bX) = \tp^{\cQ}(\bY) = \mu \}.
		\]
		Since $\cQ$ is countably saturated, $\cK$ is closed.  Consider the definable predicate
		\begin{align*}
			\psi^{\cM}(\bX,\bY) &= \inf_{Z \in D_1} \left( \sum_{j=1}^d d^{\cM}(e^{\pi i (Z + Z^*)}X_j e^{-\pi i(Z + Z^*)}, Y_j)^2 \right)^{1/2} \\
			&= \inf_{\text{unitaries } U \in \cM} \norm{U \bX U^* - \bY}_2.
		\end{align*}
		One can check that $\psi^{\cM}$ is a definable predicate by expressing $e^{\pi i(Z + Z^*)}$ as a power series and hence approximating it by non-commutative $*$-polynomials uniformly on each operator norm ball.  Moreover, since every unitary $U$ in a tracial $\mathrm{W}^*$-algebra can be expressed as $e^{2\pi i Z}$ for some self-adjoint $Z$, we see that $e^{\pi i(Z + Z^*)}$ for $Z \in D_1$ will produce all the unitaries in $\cM$.
		
		Now by our assumption, $\psi^{\cQ}(\bX,\bY) = 0$ whenever $\tp^{\cQ}(\bX,\bY) \in \cK$.  Let $\cO_k$ be a sequence of neighborhoods of $\mu$ in $\mathbb{S}_{d,r}(\rT_{\tr})$ such that $\overline{\cO}_{k+1} \subseteq \cO_k$ and $\bigcap_{k \in \N} \cO_k = \{\mu\}$.  Then let
		\[
		\cO_k' = \{\tp^{\cM}(\bX,\bY): \cM \models \rT_{\tr}, \tp^{\cM}(\bX) \in \cO_k, \tp^{\cM}(\bY) \in \cO_k \}.
		\]
		Then $\overline{\cO_{k+1}'} \subseteq \cO_k'$ and $\bigcap_{k \in \N} \cO_k' = \cK$.  Fix $\epsilon > 0$.  Since $\cK \subseteq \{\gamma \in \mathbb{S}_{2d,r}: \gamma(\psi) = 0\}$, we see that $(\mathbb{S}_{2d,r}(\rT_{\tr}) \setminus \overline{\cO}_k')_{k \in \N}$ is an open cover of $\{\gamma \in \mathbb{S}_{2d,r}(\rT_{\tr}): \gamma(\psi) \geq \epsilon \}$.  Since the latter set is compact and the $\cO_k$'s are nested, there exists some $k$ such that
		\[
		\{\gamma \in \mathbb{S}_{2d,r}: \gamma(\psi) \geq \epsilon \} \subseteq \mathbb{S}_{2d,r}(\rT_{\tr}) \setminus \overline{\cO}_k',
		\]
		or in other words
		\[
		\cO_k' \subseteq \{\gamma \in \mathbb{S}_{2d,r}: \gamma(\psi) < \epsilon \}.
		\]
		This implies that for all $(\bX, \bY) \in \Gamma_r^{(n)}(\cO_k) \times \Gamma_r^{(n)}(\cO_k) = \Gamma_r^{(n)}(\cO_k')$, we have $\psi^{M_n(\C)}(\bX,\bY) < \epsilon$, or in other words, there exists a unitary $U \in M_n(\C)$ such that $\norm{U \bX U^* - \bY}_2 < \epsilon$.
		
		By \cite[Theorem 7]{Szarek1998}, the unitary group $U(n)$ can be covered by $(C / \epsilon)^{n^2}$ operator-norm balls of radius $\epsilon$, where $C$ is a universal constant (independent of $n$).   Let $(U_i)_{i=1}^m$ be unitaries such that the balls $B_{\norm{\cdot}}(U_i,\epsilon)$ cover the unitary group.  Now we claim that the balls $B_{\norm{\cdot}_2}(U_i \bX U_i^*, (2\sqrt{d} r + 1) \epsilon)$ cover $\Gamma_r^{(n)}(\cO_k)$.  To see this, let $\bY$ be in the microstate space $\Gamma_r^{(n)}(\cO_k)$.  Then there exists a unitary $U$ such that $\norm{U \bX U^* - \bY}_2 < \epsilon$.
		
		Next, there exists a unitary $U_j$ such that $\norm{U_j - U} < \epsilon$.  Then
		\begin{align*}
			\norm{U_i X_j U_i^* - U X_j U^*}_2 &\leq \norm{(U_i - U) X_j U_i^*}_2 + \norm{U X_j (U_i^* - U^*)}_2 \\
			&\leq \norm{U_i - U} \norm{X_j}_2 \norm{U_i^*} + \norm{U} \norm{X_j}_2 \norm{U_i^* - U^*} \\
			&\leq 2r \epsilon.
		\end{align*}
		Thus,
		\[
		\norm{U_i \bX U_i^* - U \bX U^*}_2 = \left( \sum_{j=1}^d \norm{U_i X_j U_i^* - U X_j U^*}_2^2 \right)^{1/2} \leq 2\sqrt{d} r \epsilon.
		\]
		So by the triangle inequality, $\norm{U_i \bX U_i^* - \bY}_2 < 2 \sqrt{d} r \epsilon + \epsilon$.  Hence, $\Gamma_r^{(n)}(\cO_k)$ is covered by the balls $B_{\norm{\cdot}_2}(U_i \bX U_i^*, (2\sqrt{d}r + 1)\epsilon)$.
		
		Now $M_n(\C)^d$ is a complex vector space of dimension $d n^2$, hence isometric to $\C^{dn^2}$.  A ball of radius $(2\sqrt{d}r + 1) \epsilon$ therefore has Lebesgue measure
		\[
		\frac{\pi^{n^2d}}{(n^2 d)!} [(2 \sqrt{d} r + 1) \epsilon]^{2n^2 d}.
		\]
		The number of balls needed to cover the microstate space is $(C / \epsilon)^{n^2}$, and hence
		\[
		\vol \Gamma_r^{(n)}(\cO_k) \leq \frac{C^{n^2}}{\epsilon^{n^2}} \frac{\pi^{dn^2}}{(n^2 d)!} (2\sqrt{d} r + 1)^{2dn^2} \epsilon^{2dn^2} = C^{n^2} \pi^{dn^2} (2\sqrt{d} r + 1)^{2dn^2} \epsilon^{(2d - 1)n^2}.
		\]
		Thus,
		\[
		\frac{1}{n^2} \log \vol \Gamma_r^{(n)}(\cO_k) - \frac{1}{n^2} \log (n^2 d)! \leq \log C + d \log \pi + 2d \log(2 \sqrt{d} r + 1) + (2d - 1) \log \epsilon.
		\]
		Now using Stirling's formula
		\[
		\log (n^2 d)! = n^2 d \log (n^2d) - n^2 d + O(\log(n^2d)),
		\]
		so that
		\[
		\frac{1}{n^2} \log (n^2 d)! = 2d \log n + d (\log d - 1) + O(\log(n^2d)/n^2),
		\]
		where the last error term goes to zero as $n \to \infty$.  Hence, replacing $(1/n^2) \log(n^2d)!$ by $2d \log n$ does not change the asymptotic behavior of the above equation.  After this replacement, sending $n \to \infty$ results in
		\[
		\chi_r^{\cU}(\mu) \leq \log C + d \log \pi - d(\log d - 1) + 2d \log(2 \sqrt{d} r + 1) + (2d - 1) \log \epsilon.
		\]
		Since $\epsilon$ was arbitrary, $\chi_r^{\cU}(\mu) = -\infty$, and since $r$ was arbitrary $\chi^{\cU}(\mu) = -\infty$.
	\end{proof}
	
	\begin{example} \label{ex:unitaryinvarianceinsufficient}
		This example shows that unitary invariance (4) is not a strong enough hypothesis to produce an analog of Theorem \ref{thm:asymptoticfreeness} for full types.  Indeed, choose some type $\mu$ with $\chi^{\cU}(\mu) > -\infty$.  It follows that there exist two tuples $\bX = [\bX^{(n)}]_{n \in \N}$ and $\bY = [\bY^{(n)}]_{n \in \N}$ in $\cQ$ with type $\mu$ that are not unitarily conjugate.  Letting $\psi$ be the formula in the previous lemma, we have
		\[
		\lim_{n \to \cU} \psi^{M_n(\C)}(\bX^{(n)},\bY^{(n)}) = \psi^{\cQ}(\bX,\bY) > 0.
		\]
		(If $\psi^{\cQ}(\bX,\bY)$ were equal to zero, then since $\cQ$ is countably saturated there would exist a unitary $U$ with $\norm{U \bX U^* - \bY}_2 = 0$.)  Now let $U^{(n)}$ and $V^{(n)}$ be independent random variables, each chosen according to the Haar probability measure on the $n \times n$ unitary group.  Consider the two pairs of matrix models:
		\begin{enumerate}[(A)]
			\item $U^{(n)} \bX^{(n)} (U^{(n)})^*$ and $V^{(n)} \bX^{(n)} (V^{(n)})^*$,
			\item $U^{(n)} \bX^{(n)} (U^{(n)})^*$ and $V^{(n)} \bY^{(n)} (V^{(n)})^*$.
		\end{enumerate}
		Then both pairs of random matrix models satisfy the hypotheses of Theorem \ref{thm:asymptoticfreeness}:
		\begin{enumerate}[(1)]
			\item The types of $U^{(n)} \bX^{(n)} (U^{(n)})^*$ and $V^{(n)} \bX^{(n)} (V^{(n)})^*$ and $V^{(n)} \bY^{(n)} (V^{(n)})^*$ in $M_n(\C)$ converge to $\mu$ almost surely since the type is invariant under unitary conjugation. 
			\item The operator norms of $\bX^{(n)}$ and $\bY^{(n)}$ are bounded.
			\item Since $\bX$ and $\bY$ are deterministic and $U^{(n)}$ and $V^{(n)}$ are independent, the tuples from (A) and (B) respectively are independent. 
			\item The probability distributions are unitarily invariant by construction becaue the Haar measure on the unitary group is left-invariant.
		\end{enumerate}
		However, we have
		\[
		\psi^{M_n(\C)}(U^{(n)} \bX^{(n)} (U^{(n)})^*, V^{(n)} \bX^{(n)} (V^{(n)})^*) = \psi^{M_n(\C)}(\bX^{(n)},\bX^{(n)}) \to 0
		\]
		and
		\[
		\psi^{M_n(\C)}(U^{(n)} \bX^{(n)} (U^{(n)})^*, V^{(n)} \bY^{(n)} (V^{(n)})^*) = \psi^{M_n(\C)}(\bX^{(n)},\bY^{(n)}) \to \psi^{\cQ}(\bX,\bY) > 0.
		\]
		Thus, $\tp^{M_n(\C)}(U^{(n)} \bX^{(n)} (U^{(n)})^*, V^{(n)} \bX^{(n)} (V^{(n)})^*)$ and $\tp^{M_n(\C)}(U^{(n)} \bX^{(n)} (U^{(n)})^*, V^{(n)} \bY^{(n)} (V^{(n)})^*)$ \emph{cannot} converge in probability to the same limit.
	\end{example}
	
	If we want some analog of Theorem \ref{thm:asymptoticfreeness} for full types, what condition could we make that is stronger than unitary invariance?  Intuitively, we want something like invariance under automorphisms of $\cQ$, since (under the continuum hypothesis) any two tuples with the same type are conjugate by an automorphism of $\cQ$.  This would translate into invariance under approximate automorphisms on the matrix level.  However, it is not clear to me how to formulate this notion precisely.
	
	Rather, we could think of an automorphism-invariant measure on the microstate space as one which gives equal weight to all the different microstates for the same type, or in other words, a uniform distribution on the microstate space.  This motivates the following definition, which analogous to a known property of the independent join for $*$-laws \cite[Lemma 3.5]{Voiculescu1998}.
	
	\begin{definition}[Independent join] \label{def:independentjoin}
		Let $\cQ = \prod_{n \to \cU} M_n(\C)$.  Let $\mu_1 \in \mathbb{S}_{d_1}(\Th(\cQ))$ and $\mu_2 \in \mathbb{S}_{d_2}(\Th(\cQ))$.  We say that a type $\mu \in \mathbb{S}_{d_1 + d_2}(\Th(\cQ))$ is an \emph{independent join} of $\mu_1$ and $\mu_2$ if for every sufficiently large $r$, we have
		\[
		\liminf_{\cO \searrow \mu} \liminf_{(\cO_1,\cO_2) \searrow (\mu_1,\mu_2)}
		\lim_{n \to \cU} \frac{\vol(\Gamma_r^{(n)}(\cO) \cap [\Gamma_r^{(n)}(\cO_1) \times \Gamma_r^{(n)}(\cO_2)]) }{\vol(\Gamma_r^{(n)}(\cO_1) \times \Gamma_r^{(n)}(\cO_2))} = 1,
		\]
		where the $\liminf$ as $\cO \searrow \gamma$ denotes the $\liminf$ over the directed system of neighborhoods of $\mu$ in $\mathbb{S}_{d_1+d_2,r}(\rT_{\tr})$ ordered by reverse inclusion, and similarly $(\cO_1,\cO_2) \searrow (\mu_1,\mu_2)$ denotes the directed system of pairs of neighborhoods of $\mu_1$ and $\mu_2$ in $\mathbb{S}_{d_1,r}(\rT_{\tr})$ and $\mathbb{S}_{d_2,r}(\rT_{\tr})$ respectively ordered by reverse inclusion of both elements of the pair.
	\end{definition}
	
	Intuitively, this definition says that for each neighborhood $\cO$ of $\gamma$, if the neighborhoods $\cO_1$ and $\cO_2$ are sufficiently small, then a tuple $\bX$ chosen uniformly at random from $\Gamma_r^{(n)}(\cO_1) \times \Gamma_r^{(n)}(\cO_2)$ has a high probability of being in $\Gamma_r^{(n)}(\cO)$.  We remark that the same definition makes sense for a triple $(\mu_1,\mu_2,\mu_3)$ or in general any finite tuple of types.
	
	The main issue with Definition \ref{def:independentjoin} is existence:
	\begin{question}
		Given $\mu_1 \in \mathbb{S}_{d_1,r}(\Th(\cQ))$ and $\mu_2 \in \mathbb{S}_{d_2,r}$, does an independent join exist?
	\end{question}
	Resolving this question will likely require a more detailed understanding of the behavior of formulas in independent random matrix tuples, just as Theorem \ref{thm:asymptoticfreeness} required analyzing the behavior of terms and quantifier-free formulas.  This endeavor may also be entangled with Question \ref{q:matrixtheoryconvergence}.
	
	However, Definition \ref{def:independentjoin} at least has several desirable properties.
	
	\begin{observation}[Uniqueness]
		Given $\mu_1 \in \mathbb{S}_{d_1,r}(\Th(\cQ))$ and $\mu_2 \in \mathbb{S}_{d_2,r}$, there is at most one independent join.
	\end{observation}
	
	\begin{proof}
		Suppose that $\mu$ and $\mu'$ are both independent joins of $\mu_1$ and $\mu_2$.  Fix a large $r > 0$.  Since $\mathbb{S}_{d_1+d_2,r}(\rT_{\tr})$ is Hausdorff, any sufficiently small neighborhoods of $\cO$ of $\mu$ and $\cO'$ of $\mu'$ will be disjoint.  Therefore, for all neighborhoods $\cO_1$ and $\cO_2$ of $\mu_1$ and $\mu_2$, we will have
		\[
		\frac{\vol(\Gamma_r^{(n)}(\cO) \cap [\Gamma_r^{(n)}(\cO_1) \times \Gamma_r^{(n)}(\cO_2)]) }{\vol(\Gamma_r^{(n)}(\cO_1) \times \Gamma_r^{(n)}(\cO_2))} + \frac{\vol(\Gamma_r^{(n)}(\cO') \cap [\Gamma_r^{(n)}(\cO_1) \times \Gamma_r^{(n)}(\cO_2)]) }{\vol(\Gamma_r^{(n)}(\cO_1) \times \Gamma_r^{(n)}(\cO_2))} \leq 1.
		\]
		This implies that the limits in Definition \ref{def:independentjoin} cannot be $1$ for both $\mu$ and $\mu'$.
	\end{proof}
	
	\begin{observation}[Symmetry]
		Suppose that $\tp^{\cQ}(\bX,\bY)$ is the independent join of $\tp^{\cQ}(\bX)$ and $\tp^{\cQ}(\bY)$.  Then $\tp^{\cQ}(\bY,\bX)$ is the independent join of $\tp^{\cQ}(\bY)$ and $\tp^{\cQ}(\bX)$.
	\end{observation}
	This follows from the symmetry of the definition and the symmetry of the notion of neighborhoods in the logic topology.  We leave the details to the reader.
	
	\begin{observation}[Amalgamation property]
		Let $\mu_1 \in \mathbb{S}_{d_1}(\Th(\cQ))$, $\mu_2 \in \mathbb{S}_{d_2}(\Th(\cQ))$, and $\mu_3 \in \mathbb{S}_{d_3}(\Th(\cQ))$.  Let $\mu$ be an independent join of $\mu_1$ and $\mu_2$, and let $\mu'$ be an independent join of $\mu$ and $\mu_3$.  Then $\mu'$ is an independent join of $\mu_1, \mu_2, \mu_3$.
	\end{observation}
	
	\begin{proof}
		Per the three-type analog of Definition \ref{def:independentjoin}, we need to show that
		\[
		\liminf_{\cO' \searrow \mu'} \liminf_{(\cO_1,\cO_2,\cO_3) \searrow (\mu_1,\mu_2,\mu_3)}
		\lim_{n \to \cU} \frac{\vol(\Gamma_r^{(n)}(\cO') \cap [\Gamma_r^{(n)}(\cO_1) \times \Gamma_r^{(n)}(\cO_2) \times \Gamma_r^{(n)}(\cO_3)]) }{\vol(\Gamma_r^{(n)}(\cO_1) \times \Gamma_r^{(n)}(\cO_2) \times \Gamma_r^{(n)}(\cO_3))} = 1.
		\]
		The $\liminf$ on the left will not change if we insert an auxiliary variable $\cO$ representing a neighborhood of $\mu$, so we can write our goal as
		\[
		\liminf_{\cO' \searrow \mu'} \liminf_{\cO \searrow \mu} \liminf_{(\cO_1,\cO_2,\cO_3) \searrow (\mu_1,\mu_2,\mu_3)}
		\lim_{n \to \cU} \frac{\vol(\Gamma_r^{(n)}(\cO') \cap [\Gamma_r^{(n)}(\cO_1) \times \Gamma_r^{(n)}(\cO_2) \times \Gamma_r^{(n)}(\cO_3)]) }{\vol(\Gamma_r^{(n)}(\cO_1) \times \Gamma_r^{(n)}(\cO_2) \times \Gamma_r^{(n)}(\cO_3))} = 1.
		\]
		Let
		\[
		\tilde{\cO} := \{\tp^{\cM}(\bX,\bY): \cM \models \rT_{\tr}, \tp^{\cM}(\bX,\bY) \in \cO, \tp^{\cM}(\bX) \in \cO_1, \tp^{\cM}(\bY) \in \cO_2 \}.
		\]
		Since $\Gamma_r^{(n)}(\tilde{\cO}) \subseteq \Gamma_r^{(n)}(\cO_1) \times \Gamma_r^{(n)}(\cO_2)$, the quantity that we want to estimate can be bounded from below by
		\begin{multline} \label{eq:volumeproduct}
			\frac{\vol(\Gamma_r^{(n)}(\cO') \cap [\Gamma_r^{(n)}(\tilde{\cO}) \times \Gamma_r^{(n)}(\cO_3)]) }{\vol(\Gamma_r^{(n)}(\cO_1) \times \Gamma_r^{(n)}(\cO_2) \times \Gamma_r^{(n)}(\cO_3))}
			\\ =
			\frac{\vol(\Gamma_r^{(n)}(\cO') \cap [\Gamma_r^{(n)}(\tilde{\cO}) \times \Gamma_r^{(n)}(\cO_3)]) }{\vol(\Gamma_r^{(n)}(\tilde{\cO}) \times \Gamma_r^{(n)}(\cO_3))}
			\frac{\vol(\Gamma_r^{(n)}(\tilde{\cO}) \times \Gamma_r^{(n)}(\cO_3)) }{\vol(\Gamma_r^{(n)}(\cO_1) \times \Gamma_r^{(n)}(\cO_2) \times \Gamma_r^{(n)}(\cO_3))}.
		\end{multline}
		To estimate the first term, we observe that for a fixed $\cO$, the net $(\tilde{\cO},\cO_3)$ indexed by $(\cO_1,\cO_2,\cO_3) \searrow (\mu_1,\mu_2,\mu_3)$ is a subnet of the net $(\cO_{1,2},\cO_3)$ of neighborhoods of $(\mu,\mu_3)$ (here $\cO_{1,2}$ is another variable distinct from $\cO$ but also representing neighborhoods of $\mu$).  We therefore have
		\begin{multline*}
			\liminf_{(\cO_1,\cO_2,\cO_3) \searrow (\mu_1,\mu_2,\mu_3)} \lim_{n \to \cU} \frac{\vol(\Gamma_r^{(n)}(\cO') \cap [\Gamma_r^{(n)}(\tilde{\cO}) \times \Gamma_r^{(n)}(\cO_3)]) }{\vol(\Gamma_r^{(n)}(\tilde{\cO}) \times \Gamma_r^{(n)}(\cO_3))} \\
			\geq \liminf_{(\cO_{1,2},\cO_3) \searrow (\mu,\mu_3)} \lim_{n \to \cU} \frac{\vol(\Gamma_r^{(n)}(\cO') \cap [\Gamma_r^{(n)}(\cO_{1,2}) \times \Gamma_r^{(n)}(\cO_3)])}{\vol(\Gamma_r^{(n)}(\cO_{1,2}) \times \Gamma_r^{(n)}(\cO_3))}.
		\end{multline*}
		On the right-hand side, applying $\liminf_{\cO \searrow \mu}$ does nothing, and then applying $\liminf_{\cO' \searrow \mu'}$ yields the limit in the definition of $\mu'$ being an independent join of $\mu$ and $\mu_3$, and thus
		\[
		\liminf_{\cO' \searrow \mu'} \liminf_{\cO \searrow \mu} \liminf_{(\cO_1,\cO_2,\cO_3) \searrow (\mu_1,\mu_2,\mu_3)} \lim_{n \to \cU} \frac{\vol(\Gamma_r^{(n)}(\cO') \cap [\Gamma_r^{(n)}(\tilde{\cO}) \times \Gamma_r^{(n)}(\cO_3)]) }{\vol(\Gamma_r^{(n)}(\tilde{\cO}) \times \Gamma_r^{(n)}(\cO_3))} \geq 1.
		\]
		Now the second term on the right-hand side of \eqref{eq:volumeproduct} is
		\begin{align*}
			\frac{\vol(\Gamma_r^{(n)}(\tilde{\cO}) \times \Gamma_r^{(n)}(\cO_3)) }{\vol(\Gamma_r^{(n)}(\cO_1) \times \Gamma_r^{(n)}(\cO_2) \times \Gamma_r^{(n)}(\cO_3))} &= \frac{\vol(\Gamma_r^{(n)}(\tilde{\cO})) }{\vol(\Gamma_r^{(n)}(\cO_1) \times \Gamma_r^{(n)}(\cO_2))} \\
			&= \frac{\vol(\Gamma_r^{(n)}(\cO) \cap [\Gamma_r^{(n)}(\cO_1) \times \Gamma_r^{(n)}(\cO_2)]) }{\vol(\Gamma_r^{(n)}(\cO_1) \times \Gamma_r^{(n)}(\cO_2))}.
		\end{align*}
		Taking the $\liminf$ as $(\cO_1,\cO_2,\cO_3) \searrow (\mu_1,\mu_2,\mu_3)$ is equivalent to taking the $\liminf$ as $(\cO_1,\cO_2) \searrow (\mu_1,\mu_2)$ since the expression is independent of $\cO_3$.  After that, we apply $\liminf_{\cO \searrow \mu}$, which results in the limit in the definition of $\mu$ being an independent join of $\mu_1$ and $\mu_2$.  Therefore,
		\[
		\liminf_{\cO' \searrow \mu'} \liminf_{\cO \searrow \mu} \liminf_{(\cO_1,\cO_2,\cO_3) \searrow (\mu_1,\mu_2,\mu_3)} \lim_{n \to \cU} \frac{\vol(\Gamma_r^{(n)}(\cO) \cap [\Gamma_r^{(n)}(\cO_1) \times \Gamma_r^{(n)}(\cO_2)]) }{\vol(\Gamma_r^{(n)}(\cO_1) \times \Gamma_r^{(n)}(\cO_2))} = \liminf_{\cO' \searrow \mu'} 1 = 1.
		\]
		Thus, when we multiply out the two terms in \eqref{eq:volumeproduct}, we obtain
		\[
		\liminf_{\cO' \searrow \mu'} \liminf_{\cO \searrow \mu} \liminf_{(\cO_1,\cO_2,\cO_3) \searrow (\mu_1,\mu_2,\mu_3)}
		\lim_{n \to \cU} \frac{\vol(\Gamma_r^{(n)}(\cO') \cap [\Gamma_r^{(n)}(\cO_1) \times \Gamma_r^{(n)}(\cO_2) \times \Gamma_r^{(n)}(\cO_3)]) }{\vol(\Gamma_r^{(n)}(\cO_1) \times \Gamma_r^{(n)}(\cO_2) \times \Gamma_r^{(n)}(\cO_3))} \geq 1.
		\]
		On the other hand, this limit is clearly less than or equal to $1$ by monotonicity of Lebesgue measure, so the proof is complete.
	\end{proof}
	
	We remark that if $\mu$ is an independent join of $\mu_1$ and $\mu_2$ according to Definition \ref{def:independentjoin}, then
	\[
	\chi^{\cU}(\mu) = \chi^{\cU}(\mu_1) + \chi^{\cU}(\mu_2).
	\]
	The argument is the same as \cite[Theorem 3.8]{Voiculescu1998}.
	Indeed, the inequality $\chi^{\cU}(\mu) \leq \chi^{\cU}(\mu_1) + \chi^{\cU}(\mu_2)$ follows because for all neighborhoods $\cO_1$ and $\cO_2$ of $\mu_1$ and $\mu_2$ respectively, $\Gamma_r^{(n)}(\cO_1) \times \Gamma_r^{(n)}(\cO_2)$ is a microstate space for $\mu$.  On the other hand, for any neighborhood $\cO$ of $\mu$, we have
	\begin{align*}
		\frac{1}{n^2} &\log \vol \Gamma_r^{(n)}(\cO) + (d_1 + d_2) \log n \\
		&\geq \frac{1}{n^2} \log \vol (\Gamma_r^{(n)}(\cO) \cap [\Gamma_r^{(n)}(\cO_1) \times \Gamma_r^{(n)}(\cO_2)]) + (d_1 + d_2) \log n \\
		&= \frac{1}{n^2} \log \frac{\vol(\Gamma_r^{(n)}(\cO) \cap [\Gamma_r^{(n)}(\cO_1) \times \Gamma_r^{(n)}(\cO_2)]) }{\vol(\Gamma_r^{(n)}(\cO_1) \times \Gamma_r^{(n)}(\cO_2))} + \frac{1}{n^2} \log \vol \Gamma_r^{(n)}(\cO_1) + d_1 \log n + \frac{1}{n^2} \log \vol \Gamma_r^{(n)}(\cO_2) + d_2 \log n \\
		&\geq \log \frac{\vol(\Gamma_r^{(n)}(\cO) \cap [\Gamma_r^{(n)}(\cO_1) \times \Gamma_r^{(n)}(\cO_2)]) }{\vol(\Gamma_r^{(n)}(\cO_1) \times \Gamma_r^{(n)}(\cO_2))} + \frac{1}{n^2} \log \vol \Gamma_r^{(n)}(\cO_1) + d_1 \log n + \frac{1}{n^2} \log \vol \Gamma_r^{(n)}(\cO_2) + d_2 \log n
	\end{align*}
	Taking the limit as $n \to \cU$, we get
	\[
	\chi_r^{\cU}(\cO) \geq \lim_{n \to \cU} \log \frac{\vol(\Gamma_r^{(n)}(\cO) \cap [\Gamma_r^{(n)}(\cO_1) \times \Gamma_r^{(n)}(\cO_2)]) }{\vol(\Gamma_r^{(n)}(\cO_1) \times \Gamma_r^{(n)}(\cO_2))} + \chi_r^{\cU}(\mu_1) + \chi_r^{\cU}(\mu_2).
	\]
	Then by taking the $\liminf$ as $(\cO_1,\cO_2) \searrow (\mu_1,\mu_2)$ and then the $\liminf$ as $\cO \searrow \mu$, we obtain $\chi_r^{\cU}(\mu) \geq \chi_r^{\cU}(\mu_1) + \chi_r^{\cU}(\mu_2)$.
	
	On the other hand, we do not know if $\chi^{\cU}(\mu) = \chi^{\cU}(\mu_1) + \chi^{\cU}(\mu_2)$ implies that $\mu$ is an independent join of $\mu_1$ and $\mu_2$ according to Definition \ref{def:independentjoin}.  Rather it only implies the weaker condition that
	\[
	\liminf_{\cO \searrow \mu} \liminf_{(\cO_1,\cO_2) \searrow (\mu_1,\mu_2)}
	\lim_{n \to \cU} \frac{1}{n^2} \log \frac{\vol(\Gamma_r^{(n)}(\cO) \cap [\Gamma_r^{(n)}(\cO_1) \times \Gamma_r^{(n)}(\cO_2)]) }{\vol(\Gamma_r^{(n)}(\cO_1) \times \Gamma_r^{(n)}(\cO_2))} = 0.
	\]
	For Voiculescu's free entropy for non-commutative $*$-laws (with limits respect to an ultrafilter $\cU$), the freely independent join is the unique $*$-law $\mu$ satisfying $\chi_{\qf}^{\cU}(\mu) = \chi_{\qf}^{\cU}(\mu_1) + \chi_{\qf}^{\cU}(\mu_2)$; the case of $d_1 = d_2 = 1$ is handled in \cite[Proposition 4.3]{VoiculescuFE4}.  However, the natural proof for $d_1$, $d_2 > 1$ based on the ideas of \cite{Voiculescu1998} would rely on unitary invariance to obtain freeness, which as explained above does not work in the setting of full types.  %to do: Look for other citations
	
	\begin{question}
		Let $\mu_1$ and $\mu_2$ be types of $d_1$-tuple and $d_2$-tuple respectively.  Under what conditions is there a unique join $\mu$ such that $\chi^{\cU}(\mu) = \chi^{\cU}(\mu_1) + \chi^{\cU}(\mu_2)$?
	\end{question}
	
	\subsection{Independence axioms from model theory}
	
	Independence has been studied in model theory in the context of stable theories.  It turns out that the theory of any $\mathrm{II}_1$ factor is not stable \cite{FHS2013}.  However, there are still situations where unstable theories admit a ``well-behaved'' independence relation.  For the present discussion, we will focus on the axioms that we want an independence relation in $\cL_{\tr}$ to satisfy.
	
	\begin{definition}[{See \cite{Adler2009}, \cite{EG2012}}]
		Let $\cQ$ be a ``large'' tracial $\mathrm{W}^*$-algebra.  For ``small'' sets $A, B, C \subseteq \cQ$, we say that $A \equiv_C B$ if there is an automorphism of $\cQ$ mapping $A$ to $B$ and fixing $C$ pointwise (and $A \equiv B$ if there is any automorphism mapping $A$ to $B$).  An \emph{independence relation} is a ternary relation $A \forkindep[C] B$ between ``small'' sets in $\cQ$ (read as ``$A$ is independent from $B$ over $C$'') satisfying the following properties:
		\begin{enumerate}[(1)]
			\item \emph{Invariance:} If $A, B, C \equiv A', B', C'$, then $A \forkindep[C] B$ if and only if $A' \forkindep[C'] B'$.
			\item \emph{Monotonicity:} If $A' \subseteq A$ and $B' \subseteq B$ and $A \forkindep[C] B$, then $A' \forkindep[C] B'$.
			\item \emph{Base monotonicity:} Suppose $D \subseteq C \subseteq B$.  If $A \forkindep[C] B$, then $A \forkindep[D] B$.
			\item \emph{Transitivity:}  Suppose $D \subseteq C \subseteq B$.  If $B \forkindep[C] A$ and $C \forkindep[D] A$, then $B \forkindep[D] A$.
			\item \emph{Normality:} $A \forkindep[C] B$ implies $A \cup C \forkindep[C] B$.
			\item \emph{Extension:} If $A \forkindep[C] B$ and $B \subseteq \tilde{B}$, then there exists $A' \equiv_{B \cup C} A$ such that $A' \forkindep[C] \tilde{B}$.
			\item \emph{Finite character:} If $A_0 \forkindep[C] B$ for all finite $A_0 \subseteq A$, then $A \forkindep[C] B$.
			\item \emph{Local character:} For every $A$, there is a cardinal $\kappa(A)$ such that for every $B$ there exists $C \subseteq B$ with $|C| < \kappa(A)$ such that $A \forkindep[C] B$.
		\end{enumerate}
		We also define the following properties that an independence relation may have
		\begin{itemize}
			\item[(9)] \emph{Symmetry:} If $A \forkindep[C] B$, then $B \forkindep[C] A$.
			\item[(10)] \emph{Full existence:} For every $A$, $B$, and $C$, there exists $A' \equiv_C A$ such that $A' \forkindep[C] B$.
		\end{itemize}
	\end{definition}
	
	Goldbring, Hart, and Sinclair \cite{GHS2013} asked whether there was some independence relation $A \forkindep[C] B$ related to $\mathrm{W}^*(A,C)$ being freely independent from $\mathrm{W}^*(B,C)$ with amalgamation over $\mathrm{W}^*(C)$.  Free independence with amalgamation is a moment condition similar to free independence except using the conditional expectation onto $\mathrm{W}^*(C)$ rather than the trace. The properties (1) - (5), (7), (9) above hold for free independence with amalgamation.  However, free independence with amalgamation only depends on the quantifier-free type of $(A,B,C)$, and ideally we would want an independence relation that depends on the full type.
	
	\begin{question}
		Let $\cQ=\prod_{n \to \cU} M_n(\C)$.  Does there exist an independence relation on countable subsets of $\cQ$ (that is, a relation satisfying (1) - (8))?  Furthermore, does there exist such an independence relation such that $A \forkindep[C] B$ implies that $\mathrm{W}^*(A,C)$ and $\mathrm{W}^*(B,C)$ are freely independent with amalgamation over $\mathrm{W}^*(C)$?
	\end{question}
	
	One might na\"ively hope that there could be an independence relation satisfying the above properties such that the full type of $(A,B,C)$ is determined by the types of $(A,C)$ and $(B,C)$ and the specification of independence $A \forkindep[C] B$.  However, this is too much to ask because having a unique way of creating an independent join would imply that theory $\rT$ is stable (see \cite[\S 8]{TZ2012} for the discrete model theory case), but \cite{FHS2013} showed that the theory of a $\mathrm{II}_1$ factor is never stable.
	
	The main difficulty in studying free independence with amalgamation in the framework of the independence axioms above concerns the extension property (8) and the related full existence property (10).  By \cite[Remark 1.2]{Adler2009}, if (1), (2), (3), (4), (5), (7), and (9) hold, then the extension property (8) is equivalent to the full existence property (10).  To show (10), we would want to know that independent products always exist in $\cQ$ in the following sense:  If $\cM \subseteq \cM_1 \subseteq \cQ$ and $\cM \subseteq \cM_2 \subseteq \cQ$, then there exists a copy $\tilde{\cM}_2$ of $\cM_2$ and an isomorphism $\Phi: \cM_2 \to \tilde{\cM}_2$ that restricts to the identity on $\cM$, such that $\cM_1$ and $\tilde{\cM}_2$ are ``independent'' over $\cM$ with respect to whatever independence relation we are studying, and such that the $\cM$-type of $\phi(\bX)$ is the same as the $\cM$-type of $\bX$ for all tuples $\bX$ in $\cM_2$.
	
	We can state the same problem in terms of automorphisms of $\cQ$ as follows.  Assuming the continuum hypothesis, if $\alpha$ and $\beta$ are embeddings of $\cM$ into $\cQ$ such that $\alpha(\bX)$ and $\beta(\bX)$ have the same type for every tuple $\bX$ from $\cM$, then $\alpha$ and $\beta$ are conjugate by an automorphism of $\cQ$.  Thus, given $\cM \subseteq \cM_1 \subseteq \cQ$ and $\cM \subseteq \cM_2 \subseteq \cQ$, we would need there to exist an automorphism $\phi$ of $\cQ$ that fixes $\cM$ pointwise and such that $\cM_1$ and $\phi(\cM_2)$ are independent according to whatever independence relation we are studying.  Thus, the question is whether there are ``enough'' automorphisms fixing $\cM$ that we can move the $\cM_2$ into an ``independent'' position from $\cM_1$.
	
	But at this point, it is not even known whether we can move $\cM_2$ so as to become freely independent from $\cM_1$ with amalgamation over $\cM$.  Indeed, it is not known in general whether there even exists an embedding of the amalgamated free product $\cM_1 *_{\cM} \cM_2$ into $\cQ$ (Brown, Dykema, and Jung \cite{BDJ2008} proved this when $\cM$ is amenable).  %This issue is also discussed in \S 6.1 of Golbring and Hart's article in this volume.
	
	As in \ref{def:independentjoin}, we could try choosing the embeddings of $\cM_1$ and $\cM_2$ ``uniformly at random'' out of all embeddings that restrict to a given embedding of $\cM$.  In other words, fixing generators $\bX$ and $\bY$ for $\cM_1$ and $\cM_2$ and generators $\bZ$ for $\cM$, we would study the relative microstate spaces for $\bX$ conditioned on $\bZ$ and $\bY$ conditioned on $\bZ$.  Let us assume for the sake of the heuristic discussion that $\bX$, $\bY$, and $\bZ$ are finite tuples.  Fixing a sequence of matrix approximations $\bZ^{(n)}$ such that $\bZ = [\bZ^{(n)}]_{n \in \cU}$ in $\cQ$, we define for neighborhoods $\cO_1$ of $\tp^{\cQ}(\bX,\bZ)$ the microstate space
	\[
	\Gamma_r^{(n)}(\cO_1 | \bZ^{(n)} \rightsquigarrow \bZ) := \{\bX' \in M_n(\C)^{d_1}: \tp^{M_n(\C)}(\bX',\bZ^{(n)}) \in \cO_1)
	\]
	and the analogous microstate space for each neighborhood $\cO_2$ of $\tp^{\cQ}(\bY,\bZ)$.  Then just as in Definition \ref{def:independentjoin}, we can ask:
	
	\begin{question}
		In the setup above, does there exist some $\bY' \in \cQ^d$ with $\tp^{\cQ}(\bY',\bZ) = \tp^{\cQ}(\bY,\bZ)$ and
		\begin{multline*}
			\liminf_{\cO \searrow \tp^{\cQ}(\bX,\bY',\bZ)} \liminf_{(\cO_1,\cO_2) \searrow (\tp^{\cQ}(\bX,\bZ), \tp^{\cQ}(\bY,\bZ))} \lim_{n \to \cU} \\ \frac{\vol(\Gamma_r^{(n)}(\cO \mid \bZ^{(n)} \rightsquigarrow \bZ) \cap [\Gamma_r^{(n)}(\cO_1 \mid \bZ^{(n)} \rightsquigarrow \bZ) \times \Gamma_r^{(n)}(\cO_2 \mid \bZ^{(n)} \rightsquigarrow \bZ)]) }{\vol(\Gamma_r^{(n)}(\cO_1 \mid \bZ^{(n)} \rightsquigarrow \bZ) \times \Gamma_r^{(n)}(\cO_2 \mid \bZ^{(n)} \rightsquigarrow \bZ))} = 1?
		\end{multline*}
	\end{question}
	
	We also point out that even if such a $\bY'$ exists, it is not immediate whether $\tp^{\cQ}(\bX,\bY',\bZ)$ only depends on $\tp^{\cQ}(\bX,\bZ)$ and $\tp^{\cQ}(\bY,\bZ)$, independent of the particular choice of $\bZ$ and $\bZ^{(n)}$ realizing those types.
	
	\section{Model theory and non-commutative optimization problems} \label{sec:optimization}
	
	\subsection{Wasserstein distance}
	
	Biane and Voiculescu \cite{BV2001} studied a non-commutative analog of the Wasserstein distance.  For two quantifier-free types $\mu$, $\nu \in \mathbb{S}_{\qf,d}(\rT_{\tr})$, the Wasserstein distance is
	\begin{equation} \label{eq:Wasserstein}
		d_W(\mu,\nu) = \inf \{ \norm{\bX - \bY}_{L^2(\cM)}: \cM \models \rT_{\tr}, \tp_{\qf}^{\cM}(\bX) = \mu, \tp_{\qf}^{\cM}(\bY) = \nu  \}.
	\end{equation}
	This is the non-commutative tracial version of the classical Wasserstein distance defined for probability measures on $\R^d$ with finite second moment, which plays a major role in optimal transport theory.  The Wasserstein distance is defined by an optimization problem, namely, to minimize the $L^2$-distance between some $\bX$ and $\bY$ with quantifier-free types $\mu$ and $\nu$ respectively.  We remark that Song studied the $L^1$ classical Wasserstein distance between types in atomless probability spaces from a model-theoretic viewpoint in \cite[\S 5.3-5.4]{SongThesis}.
	
	The analogous distance for full types $\mu$ and $\nu$ in $\mathbb{S}_d(\rT_{\tr})$ is
	\begin{equation} \label{eq:typeWasserstein}
		d_W(\mu,\nu) = \inf \{ \norm{\bX - \bY}_{L^2(\cM)}: \cM \models \rT_{\tr}, \tp^{\cM}(\bX) = \mu, \tp^{\cM}(\bY) = \nu  \}.
	\end{equation}
	Of course, since the type determines the theory of $\cM$, we have $d_W(\mu,\nu) = \infty$ unless the types $\mu$ and $\nu$ can be realized in the same tracial $\mathrm{W}^*$-algebra.  For this reason, it is natural to fix a complete theory $\rT$ (for instance the theory of the hyperfinite $\mathrm{II}_1$ factor $\cR$ or the theory of some matrix ultraproduct $\cQ$) and then study the Wasserstein distance on $\mathbb{S}_d(\rT)$.  This Wasserstein distance for types is exactly the $d$-metric between types for a given complete theory described in \cite[\S 8, p.\ 44]{BYBHU2008}.
	
	For each $r$, the Wasserstein distance $d_W$ provides a complete metric on $\mathbb{S}_{d,r}(\rT)$ as observed in \cite[Proposition 8.8]{BYBHU2008} and the topology of $d_W$ refines the logic topology \cite[Proposition 8.7]{BYBHU2008}.  The analogous properties were also shown for the non-commutative Wasserstein distance on quantifier-free types.  The topology refines the weak-$*$ topology by \cite[Proposition 1.4(b)(c)]{BV2001}.  One can also argue similarly to Proposition \cite[Proposition 8.8]{BYBHU2008} that $d_W$ gives a complete metric on the space of quantifier-free types.
	
	In classical probability theory, one can always find $\bX$ and $\bY$  in $L^\infty[0,1]$ with probability distributions $\mu$ and $\nu$ that achieve the infimum in \eqref{eq:Wasserstein} (such a pair achieving the infimum is called an \emph{optimal coupling}).  By contrast, in the non-commutative setting of \cite{BV2001}, one has very little control over which tracial von Neumann algebra is needed to achieve an optimal coupling, as explained in \cite[\S 5]{GJNS2021}.  For instance, given the negative solution to the Connes-embedding problem proposed in \cite{JNVWY2020}, for some $d$ and $r$, there exist $\mu$ and $\nu \in \mathbb{S}_{\qf,d,r}(\rT_{\tr})$ such that each of $\mu$ and $\nu$ is the $*$-law of a some $d$-tuple in a matrix algebra $M_n(\C)$, and yet the infimum \eqref{eq:Wasserstein} can only be achieved using a non-Connes embeddable algebra $\cM$ \cite[Corollary 5.14]{GJNS2021} (for a survey on the Connes embedding problem, see for instance \cite{Capraro2010} %Goldbring's article in this volume
	).  This is related to the issue discussed in \S \ref{subsec:classicalQE} that $L^\infty[0,1]$ is $\aleph_0$-categorical and admits quantifier elimination, but very few tracial $\mathrm{W}^*$-algebras have these properties.  This difficulty in the tracial $\mathrm{W}^*$-setting seems to be ameliorated by using Wasserstein distance for full types rather than quantifier-free types.  Indeed, if we fix a countably saturated model $\cM$ for a given theory $\rT$, then the Wasserstein distance between two types $\mu, \nu \in \mathbb{S}_d(\rT)$ can always be achieved by some coupling in $\cM$.  Moreover, in any model of $\rT$, there are approximately optimal couplings.
	
	Thus, consistent with the discussion in \ref{subsec:modelintro}, it is worth considering whether the space of full types relative to a complete theory $\rT \supseteq \rT_{\tr}$ may be a better analog for the space of probability distributions on $\R^d$ than the space of quantifier-free types.  However, even in this setting, the weak-$*$ and Wasserstein topologies will usually not agree; by the continuous Ryll-Nardzewski theroem \cite[Proposition 12.10]{BYBHU2008} for a complete theory $\rT$, the logic topology and $d$-metric topology agree if and only if $\rT$ is $\aleph_0$-categorical.  More precisely, using similar reasoning to \cite[\S 5.4]{GJNS2021}, we can characterize the agreement of the weak-$*$ and Wasserstein topologies at a particular point $\mu \in \mathbb{S}_{d,r}(\rT)$ through a stability or lifting property (Corollary \ref{cor:twotopologies} below), which follows from the ultraproduct characterizations of weak-$*$ convergence and of Wasserstein convergence.  The characterization of weak-$*$ convergence is as follows.
	
	\begin{lemma} \label{lem:ultraproductweakstar}
		Fix $r \in (0,\infty)$ and a free ultrafilter $\cU$ on $\N$.  For $n \in \N$, let $\cM^{(n)} \models \rT_{\tr}$ and $\bX^{(n)} \in (D_r^{\cM^{(n)}})^d$.  Let $\cM$ be a separable model of $\rT$ and $\bX \in (D_r^{\cM})^d$.  Then the following are equivalent:
		\begin{enumerate}[(1)]
			\item $\lim_{n \to \cU} \tp^{\cM^{(n)}}(\bX^{(n)}) = \tp^{\cM}(\bX)$ in the weak-$*$ topology on $\mathbb{S}_{d,r}(\rT)$.
			\item There exists an elementary embedding $\iota: \cM \to \prod_{n \to \cU} \cM^{(n)}$ such that $\iota(\bX) = [\bX^{(n)}]_{n \in \N}$.
		\end{enumerate}
	\end{lemma}
	
	\begin{proof}
		(2) $\implies$ (1) because by the fundamental theorem of ultraproducts $\tp^{\prod_{n \to \cU} \cM^{(n)}}([\bX^{(n)}]_{n\in \N} = \lim_{n \to \cU} \tp^{\cM^{(n)}}(\bX^{(n)})$.
		
		(1) $\implies$ (2) follows by similar reasoning as \cite[Lemma 4.12]{FHS2014}.
	\end{proof}
	
	To characterize Wasserstein convergence, we use the model-theoretic analog of the factorizable maps of \cite{AD2006}.
	
	\begin{definition}
		Let $\cM$ and $\cN$ be models of a complete $\rT \supseteq \rT_{\tr}$.  An \emph{elementarily factorizable map} $\Phi: \cM \to \cN$ is a map of the form $\Phi = \beta^* \alpha$ where $\alpha: \cM \to \cP$ and $\beta: \cN \to \cP$ are elementary embeddings into another $\cP \models \rT$, and $\beta^*: \cP \to \cN$ is the trace-preserving conditional expectation.
	\end{definition}
	
	\begin{lemma} \label{lem:ultraproductWasserstein}
		Fix $r \in (0,\infty)$ and a free ultrafilter $\cU$ on $\N$.  For $n \in \N$, let $\cM^{(n)} \models \rT_{\tr}$ and $\bX^{(n)} \in (D_r^{\cM^{(n)}})^d$.  Let $\cM$ be a separable model of $\rT$ and $\bX \in (D_r^{\cM})^d$.  Then the following are equivalent:
		\begin{enumerate}[(1)]
			\item $\lim_{n \to \cU} \tp^{\cM^{(n)}}(\bX^{(n)}) = \tp^{\cM}(\bX)$ with respect to Wasserstein distance.
			\item There exists an elementary embedding $\iota: \cM \to \prod_{n \to \cU} \cM^{(n)}$ such that $\iota(\bX) = [\bX^{(n)}]_{n \in \mathbb{N}}$, and there exist elementarily factorizable maps $\Phi^{(n)}: \cM \to \cM^{(n)}$ such that
			\[
			\iota(Z) = [\Phi^{(n)}(Z)]_{n \in \cN} \text{ for all } Z \in \mathrm{W}^*(\bX).
			\]
		\end{enumerate}
		In fact, if (1) holds, then every elementary $\iota: \cM \to \cM^{(n)}$ with $\iota(\bX) = [\bX^{(n)}]_{n \in \N}$ admits such a family $\Phi^{(n)}$ of elementary factorizable maps as in (2).
	\end{lemma}
	
	\begin{proof}
		(2) $\implies$ (1).  Suppose that (2) holds.
		Let $\Phi^{(n)} = (\beta^{(n)})^* \alpha^{(n)}$ where $\alpha^{(n)}: \cM \to \cP^{(n)}$ and $\beta^{(n)}: \cM^{(n)} \to \cP^{(n)}$ are elementary embeddings.  Then
		\begin{align*}
			d_W(\tp^{\cM^{(n)}}(\bX^{(n)}), \tp^{\cM}(\bX))^2 &\leq \norm{\alpha^{(n)}(\bX) - \beta^{(n)}(\bX^{(n)})}_{L^2(\cP^{(n)})^d}^2 \\
			&= \norm{\alpha^{(n)}(\bX)}_{L^2(\cP^{(n)})^d}^2 - 2 \re \ip{\alpha^{(n)}(\bX), \beta^{(n)}(\bX^{(n)})}_{L^2(\cP^{(n)})^d} + \norm{\beta^{(n)}(\bX^{(n)})}_{L^2(\cP^{(n)})^d}^2 \\
			&= \norm{\bX}_{L^2(\cM)^d}^2 - 2 \re \ip{\Phi^{(n)}(\bX), \bX^{(n)}}_{L^2(\cM^{(n)})^d} + \norm{\bX^{(n)}}_{L^2(\cM^{(n)})^d}^2 \\
			&= \norm{\bX}_{L^2(\cM)^d}^2 - \norm{\Phi^{(n)}(\bX)}_{L^2(\cM^{(n)})^d}^2 + \norm{\Phi^{(n)}(\bX) - \bX^{(n)}}_{L^2(\cM^{(n)})^d}^2.
		\end{align*}
		The latter goes to zero as $n \to \cU$ since $\iota(\bX) = [\Phi^{(n)}(\bX)]_{n \in \N} = [\bX^{(n)}]_{n \in \N}$ in $\prod_{n \to \cU} \cM^{(n)}$.
		
		Now let us prove (1) $\implies$ (2) as well as the final claim of the lemma.  Assume that (1) holds.  Let $\iota: \cM \to \prod_{n \to \cU} \cM^{(n)}$ be any elementary embedding with $\iota(\bX) = [\bX^{(n)}]_{n \in \N}$ (which exists by Lemma \ref{lem:ultraproductweakstar} since Wasserstein convergence implies weak-$*$ convergence).  Since $\tp^{\cM^{(n)}}(\bX^{(n)})$ converges to $\tp^{\cM}(\bX)$ in the Wasserstein distance, there exists models $\cP^{(n)} \models \rT$ and $\bY^{(n)}, \bZ^{(n)} \in (D_r^{\cP^{(n)}})^d$ such that
		\[
		\tp^{\cP^{(n)}}(\bY^{(n)}) = \tp^{\cM}(\bX), \qquad \tp^{\cP^{(n)}}(\bZ^{(n)}) = \tp^{\cM^{(n)}}(\bX^{(n)}), \qquad \lim_{n \to \cU} \norm{\bY^{(n)} - \bZ^{(n)}}_{L^2(\cP^{(n)})} = 0.
		\]
		By enlarging the model $\cP^{(n)}$ if necessary, we can arrange that there are elementary embeddings $\alpha^{(n)}: \cM \to \cP^{(n)}$ and $\beta^{(n)}: \cM^{(n)} \to \cP^{(n)}$ such that $\alpha^{(n)}(\bX) = \bY^{(n)}$ and $\beta^{(n)}(\bX^{(n)}) = \bZ^{(n)}$.  Now consider the elementarily factorizable map $\Phi^{(n)} = (\beta^{(n)})^* \alpha^{(n)}$.
		
		We claim that $\iota(Z) = [\Phi^{(n)}(Z)]_{n \in \N}$ for all $Z \in \mathrm{W}^*(\bX)$.  Fix $Z$.  By Proposition \ref{prop:deffuncrealization}, there exist a quantifier-free definable function $f$ relative to $\rT_{\tr}$ such that $Z = f^{\cM}(\bX)$.  Since the conditional expectation is contractive in $L^2$,
		\begin{align*}
			\norm*{\Phi^{(n)}(Z) - f^{\cM^{(n)}}(\bX^{(n)})}_{L^2(\cM^{(n)})^d} &= \norm*{(\beta^{(n)})^* \alpha^{(n)}(f^{\cM}(\bX)) - (\beta^{(n)})^* \beta^{(n)}(f^{\cM^{(n)}}(\bX^{(n)}))}_{L^2(\cM^{(n)})^d} \\
			&\leq \norm*{\alpha^{(n)}(f^{\cM}(\bX)) - \beta^{(n)}(f^{\cM^{(n)}}(\bX^{(n)}))}_{L^2(\cP^{(n)})^d} \\
			&= \norm*{f^{\cP^{(n)}}(\bY^{(n)})  - f^{\cP^{(n)}}(\bZ^{(n)})}_{L^2(\cP^{(n)})^d}.
		\end{align*}
		The right-hand side goes to zero as $n \to \cU$ because of the fact that $\norm{\bY^{(n)} - \bZ^{(n)}}_{L^2(\cP^{(n)})^d} \to 0$ and the $L^2$-uniform continuity property of definable functions (see \cite[Proposition 9.23]{BYBHU2008} and \cite[Lemma 3.19]{JekelCoveringEntropy}).  This implies that
		\[
		[\Phi^{(n)}(Z)]_{n \in \N} = [f^{\cM^{(n)}}(\bX^{(n)})]_{n \in \N} = f^{\prod_{n \to \cU} \cM^{(n)}} ([\bX^{(n)}]_{n \in \N}) = \iota(f^{\cM}(\bX)) = \iota(Z),
		\]
		where we have used the fact that quantifier-free definable functions commute with embeddings.
	\end{proof}
	
	A type $\mu \in \mathbb{S}_{d,r}(\rT)$ is called \emph{principal} if the weak-$*$ and Wasserstein topologies agree at $\mu$, meaning that every weak-$*$ neighborhood of $\mu$ in $\mathbb{S}_{d,r}(\rT)$ contains a Wasserstein neighborhood and vice versa (see \cite[Definition 12.2 and Proposition 12.4]{BYBHU2008}).  Since both the weak-$*$ and the Wasserstein topologies are metrizable, this is equivalent to saying that for a given free ultrafilter $\cU$ on $\N$, for every sequence $\mu^{(n)}$ in $\mathbb{S}_{d,r}(\rT)$, we have $\lim_{n \to \cU} \mu^{(n)} = \mu$ in the weak-$*$ topology if and only if $\lim_{n \to \cU} \mu^{(n)} = \mu$ in the Wasserstein distance.  Combining Lemmas \ref{lem:ultraproductweakstar} and \ref{lem:ultraproductWasserstein}, we obtain the following result.
	
	\begin{corollary} \label{cor:twotopologies}
		Let $\rT \supseteq \rT_{\tr}$ be a complete theory, let $\cM$ be a separable model of $\rT$, and let $\bX \in (D_r^{\cM})^d$.  Let $\cU$ be a free ultrafilter on $\N$.  Then the following are equivalent:
		\begin{enumerate}[(1)]
			\item $\tp^{\cM}(\bX)$ is a principal type.
			\item Given any models $\cM^{(n)} \models \rT$ and any elementary embedding $\iota: \cM \to \prod_{n \to \cU} \cM^{(n)}$, where $\cM^{(n)} \models \rT$, there exist elementarily factorizable maps $\Phi^{(n)}: \cM \to \cM^{(n)}$ such that
			\[
			\iota(Z) = [\Phi^{(n)}(Z)]_{n \in \mathbb{N}} \text{ for all } Z \in \mathrm{W}^*(\bX).
			\]
		\end{enumerate}
	\end{corollary}
	
	In the analogous situation for non-commutative $*$-laws, the stability or lifting property that characterized agreement of the weak-$*$ and Wasserstein topologies turned out to be closely related to amenability (more precisely, it was equivalent to amenability under the assumption of Connes-embeddability \cite[Proposition 5.26]{GJNS2021}).  One direction of this argument adapts readily to the model-theoretic setting.
	
	\begin{observation}
		Suppose that $\rT$ is a complete theory of a $\mathrm{II}_1$ factor.  Let $\cM \models \rT$ and $\bX \in M^d$.  If $\mathrm{W}^*(\bX)$ is amenable, then $\tp^{\cM}(\bX)$ is principal.
	\end{observation}
	
	\begin{proof}
		Let $\iota: \cM \to \prod_{n \to \cU} \cM^{(n)}$ be an elementary embedding, where $\cM^{(n)} \models \rT$.  Since $\cM$ and $\cM^{(n)}$ are elementarily equivalent, there exists some $\cN^{(n)} \models \rT$ that contains $\cM$ and $\cN^{(n)}$ as elementary submodels; let $\alpha_n: \cM \to \cN^{(n)}$ and $\beta_n: \cM^{(n)} \to \cN^{(n)}$ be the elementary embeddings.  Consider the induced elementary embeddings $\alpha: \cM \to \prod_{n \to \cU} \cN^{(n)}$ and $\beta: \prod_{n \to \cU} \cM^{(n)} \to \prod_{n \to \cU} \cN^{(n)}$.  Then $\alpha$ and $\beta \circ \iota$ are both elementary embeddings of $\cM$ into $\prod_{n \to \cU} \cN^{(n)}$.  Since $\cA = \mathrm{W}^*(\bX)$ is amenable and $\prod_{n \to \cU} \cN^{(n)}$ is an $\aleph_0$-saturated $\mathrm{II}_1$-factor, the embeddings $\alpha|_{\cA}$ and $\beta \circ \iota|_{\cA}$ are unitarily conjugate, so there exist unitaries $U_n \in \cN^{(n)}$ such that
		\[
		\beta \circ \iota(Z) = [U_n \alpha_n(Z)U_n^*]_{n \in \N} \text{ for } Z \in \cA.
		\]
		Then $\Phi_n = \beta_n^* \operatorname{ad}_{U_n} \circ \alpha_n$ is the desired elementarily factorizable map $\cM \to \cM^{(n)}$.  (For further detail, see the proof of \cite[Proposition 5.26]{GJNS2021}.)
	\end{proof}
	
	This motivates the following question concerning the lifting property in Corollary \ref{cor:twotopologies} for future research.
	
	\begin{question}
		Let $\rT \supseteq \rT_{\tr}$ be a complete theory.  Let $\cM \models \rT$ and let $\cA$ be a tracial $\mathrm{W}^*$-subalgebra of $\cM$.  Under what conditions on $\rT$, $\cM$, and/or $\cN$ does it hold that, for every elementary embedding $\iota$ of $\cM$ into an ultraproduct $\prod_{n \to \cU} \cM^{(n)}$ of models of $\rT$, there exist elementarily factorizable maps $\Phi^{(n)}: \cM \to \cM^{(n)}$ such that $\iota(Z) = [\Phi^{(n)}(Z)]_{n \in \N}$ for all $Z \in \cA$.
	\end{question}
	
	\subsection{Free Gibbs $*$-laws and entropy}
	
	Free Gibbs laws are an important construction in the statistical mechanics / large deviations approach to random matrix theory studied in \cite{BdMPS}.  The multivariable analog of free Gibbs laws has been studied in \cite{BS2001,GMS2006,GS2009,GS2014,DGS2016,Dabrowski2017,JekelEntropy,JekelExpectation,JLS2022}.  Here we shall study free Gibbs $*$-laws for tuples of non-self-adjoint operators.
	
	Fix a $d$-variable definable predicate $V$ relative to $\rT_{\tr}$, and assume that
	\[
	a + b \norm{\bX}_{L^2(\cM)^d}^2 \leq V^{\cM}(\bX) \leq A + B \norm{\bX}_{L^2(\cM)^d}^2
	\]
	for some $a$, $A \in \R$ and $b, B \in (0,\infty)$.  Let $V^{(n)} = V^{M_n(\C)}$ and let $\mu^{(n)}$ be the probability measure on $M_n(\C)^d$ given by
	\[
	d\mu^{(n)}(\bX) = \frac{1}{\int_{M_n(\C)^d} e^{-n^2 V^{(n)}}} e^{-n^2 V^{(n)}(\bX)}\,d\bX,
	\]
	where $d\bX$ denotes Lebesgue measure (our assumed upper and lower bounds for $V$ guarantee that the integral of $e^{-n^2 V^{(n)}}$ converges).  Let $\bX^{(n)}$ be a random element of $M_n(\C)^d$ chosen according to the measure $\mu^{(n)}$.
	
	\begin{question} \label{q:Gibbslawconvergence}
		Under what conditions does $\tp^{M_n(\C)}(\bX^{(n)})$ converge in probability as $n \to \infty$ (or as $n \to \cU$ for some given free ultrafilter $\cU$)?
	\end{question}
	
	As in \cite[Proposition 7.11 and Corollary 7.12]{JLS2022}, one can show that if there is a unique type $\mu$ maximizing
	\[
	\chi_{\full}^{\cU}(\mu) - \mu[V],
	\]
	where $\mu[V]$ denotes the evaluation of $\mu$ viewed as an element of $C(\mathbb{S})_d(\rT_{\tr})$ on $V$, then Question \ref{q:Gibbslawconvergence} has a positive answer, and in fact, for every neighborhood $\cO$ of $\mu$ and every $r > 0$, we have
	\[
	\lim_{n \to \cU} \frac{1}{n^2} \log(1 - \mu^{(n)}(\Gamma_r^{(n)}(\cO))) < 0,
	\]
	meaning that a matrix chosen randomly according to $\mu^{(n)}$ has an exponentially small probability of its type landing outside a given neighborhood of $\mu$.  When $V$ is quantifier-free definable, a $*$-law which maximizes $\chi_{\qf}(\mu) - \mu[V]$ is called a \emph{free Gibbs $*$-law} for $V$.  Similarly, we can call $\mu$ a \emph{free Gibbs type} for a given definable predicate $V$ if it maximizes $\chi_{\full}(\mu) - \mu[V]$.  This prompts the following related question.
	
	\begin{question}
		Under what conditions does there exist a unique type $\mu$ that maximizes $\chi^{\cU}(\mu) - \mu[V]$?  Does this $\mu$ depend on the choice of free ultrafilter $\cU$?
	\end{question}
	
	In the case where $V$ is quantifier-free definable and satisfies certain smoothness assumptions and estimates on the derivatives, then one can obtain a unique solution by the stochastic diffusion techniques of \cite{BS2001} or the Dyson-Schwinger equation techniques of \cite{GMS2006} or a combination of the two; see for instance \cite{GS2009} or \cite[Proposition 8.1]{JLS2022}.  However, it is not clear to me at the time of writing how these techniques should be adapted to the setting of full types.  For the diffusion techniques, we would need to understand how the logical operations of $\sup$ and $\inf$ over operator norm balls combine with the theory of free stochastic differential equations (see below for further discussion).  Similarly, to adapt the Dyson-Schwinger equation approach to the setting of full types, we would some notion of derivatives and smoothness for definable predicates and definable functions.
	
	An open problem in Voiculescu's free entropy theory is whether, for a Connes-embeddable non-commutative $*$-law $\mu$, the microstate version $\chi(\mu)$ or $\chi^{\cU}(\mu)$ from \cite{VoiculescuFE2,VoiculescuFE3} agrees with another version of free entropy called $\chi^*(\mu)$, which is defined by studying the free Fisher information $\Phi^*(\mu)$ \cite{VoiculescuFE6}. The inequality $\chi(\mu) \leq \chi^*(\mu)$ is known thanks to \cite{BCG2003}.  Moreover, \cite{Dabrowski2017} and \cite{JekelEntropy} established equality when $\mu$ is a free Gibbs law for a quantifier-free definable predicate $V$ satisfying some convexity / semi-concavity assumptions.
	
	A key part of this analysis is to study the smoothness of the potential $V^{(n)}(\bX,t)$ obtained from convolving $\mu^{(n)}$ with a Gaussian distribution.  For the sake of Proposition \ref{prop:stochasticcontrol} below, we state the problem in terms of stochastic analysis.  We recall that a \emph{filtration} on a probability space $\Omega$ is a collection of $\sigma$-algebras such $(\cF_t)_{t \in [0,\infty)}$ with $\cF_s \subseteq \cF_t$ for $s \leq t$.  A stochastic process $(X_t)_{t \in [0,\infty)}$ (such as Brownian motion) is \emph{adapted} to $(\cF_t)_{t \in [0,\infty)}$ if $X_t$ is $\cF_t$-measurable for each $t$; moreover, $(X_t)_{t \in [0,\infty)}$ is \emph{progressively measurable} if for each $u \in [0,\infty)$, the function $(s,\omega) \mapsto X_s(\omega)$ on $[0,u] \times \Omega$ is measurable with respect to the product of the Borel $\sigma$-algebra on $[0,u]$ with $\cF_u$.
	
	Fix a probability space $\Omega$ with a filtration $(\cF_t)_{t \in [0,\infty)}$, and fix an $M_n(\C)^d$-valued Brownian motion $\bZ_t^{(n)}$ adapted to $\mathcal{F}_t$, normalized so that $\E \norm{\bZ_t^{(n)}}_{M_n(\C)^d}^2 = 2td$.  Let $\mu_t^{(n)}$ be the probability distribution of $\bX^{(n)} + \bZ_t^{(n)}$, and define $V^{(n)}(\cdot,t): M_n(\C)^d \to \R$ by
	\[
	d\mu_t^{(n)}(\bX) = \frac{1}{\int_{M_n(\C)^d} e^{-n^2 V^{(n)}}} e^{-n^2 V^{(n)}(\bX,t)}\,d\bX.
	\]
	The following expression for $V^{(n)}(\bX,t)$ can be obtained from stochastic optimal control theory (see for instance \cite[\S VI]{FR1975}).
	
	\begin{proposition} \label{prop:stochasticcontrol}
		\begin{equation} \label{eq:stochasticcontrol}
			V^{(n)}(\bX,t) = \inf_{\bA} \left[ \E[V^{(n)}(\bX_t) + \frac{1}{2t^2}  \int_0^t \norm{\bA(u)}_{L^2(M_n(\C))^d}^2 \,du \right],
		\end{equation}
		where $\bA$ ranges over all progressively measurable functions $\Omega \times [0,t] \to M_n(\C)^d$, and where
		\begin{align*}
			\bX_0 &= \bX \\
			d\bX_t &= d\bZ_t^{(n)} + A(t)\,dt
		\end{align*}
	\end{proposition}
	
	Roughly speaking, \cite[\S 6]{JekelEntropy} showed that when $V^{(n)}$ is quantifier-free definable and satisfies some convexity and semi-concavity hypotheses, then $V^{(n)}(\bX,t)$ asymptotically approaches some quantifier-free definable predicate as $n \to \infty$.  In other words, the formula above which has an $\inf$ in it can be re-expressed in another way that does not have any quantifiers!  Heuristically speaking, this enabled the study of $V_t^{(n)}$ as $n \to \infty$ to proceed along the same lines as in classical probability theory where we have quantifier elimination, even though a priori the large-$n$ behavior would seem to fall in the realm of tracial $\mathrm{W}^*$-algebras which do not have quantifier elimination.  However, the regularity coming from the convexity assumptions on $V^{(n)}$ was crucial in the argument, so I expect that studying more general $V$ will require the analysis of quantifiers.
	
	Motivated by the strategy of \cite[\S 6]{JekelEntropy} (closely related to the earlier strategies of \cite{BCG2003,Dabrowski2017}), we want to show that this $V_t^{(n)}$ will asymptotically be described by some definable predicate.
	
	\begin{question} \label{q:asymptoticapproximation}
		Let $V$ be a definable predicate, and let $V^{(n)}(\bX,t)$ as above.  Does there exist a definable predicate $V_t$ such that for every $r > 0$,
		\[
		\lim_{n \to \cU} \sup_{\bX \in (D_r^{M_n(\C)})^d} \norm{V^{(n)}(\bX,t) - V_t^{M_n(\C)}(\bX)}_{L^2(M_n(\C))^d} = 0?
		\]
	\end{question}
	
	The following is a heuristic attempt to address this question, which shows that although we may not be able to obtain a quantifier-free definable $V_t$ for each quantifier-free definable $V$, there is some hope of obtaining a definable $V_t$.  This is part of our motivation for exploring free probability outside the quantifier-free setting.
	
	To address Question \ref{q:asymptoticapproximation}, we want to make a more precise connection between Proposition \ref{prop:stochasticcontrol} and the model theory of tracial $\mathrm{W}^*$-algebras.  Rather than getting into the technicalities of stochastic differential equations, we will look at discrete-time approximation of \eqref{eq:stochasticcontrol} as in \cite[\S 6]{JekelEntropy}.  Let
	\[
	\Phi_t^{(n)} V^{(n)}(\bX) = \inf_{\bA \in M_n(\C)^d} \left[ V^{(n)}(\bX + \bA + \bZ_t^{(n)}) + \frac{1}{2t} \norm{\bA}_{L^2(M_n(\C))^d}^2 \right].
	\]
	We conjecture that
	\[
	V^{(n)}(\bX,t) = \lim_{k \to \infty} (\Phi_{t/k}^{(n)})^k V^{(n)}(\bX);
	\]
	the results of \cite[\S 6]{JekelEntropy} essentially prove this assuming that $V^{(n)}$ is uniformly convex and semiconcave.
	
	Once again, let $\cQ = \prod_{n \to \cU} M_n(\C)$.  Heuristically, the limit of $\Phi_t^{(n)} V^{(n)}$ as $n \to \cU$ should be described by some function $\Phi_t V^{\cQ}: Q^d \to \R$ given by
	\[
	\Phi_t V^{\cQ}(\bX) = \inf \{ V^{\cQ}(\bX + \bA + t^{1/2} \bZ): \bA \in Q^d, \bZ \in Q^d \text{ free circular family free from } \bX \text{ and } \bA\}.
	\]
	This $\Phi_t V^{\cQ}$ is not a priori a definable predicate relative to $\Th(\cQ)$, but it can be described as a (pointwise) limit of definable predicates.  Indeed, one can verify that
	\[
	\{\tp^{\cQ}(\bX,\bA,\bZ): \bX, \bA, \bZ \in (D_r^{\cQ})^d, \bZ \text{ free circular family freely independent of } \bX, \bA \}.
	\]
	is a closed subset of $\mathbb{S}_{3d,r}(\bT_{\tr})$.  Hence, by Urysohn's lemma, there exists $\psi_r \in C(\mathbb{S}_{3d,r}(\rT_{\tr})$ with $\psi_r \geq 0$ and $\phi_r = 0$ precisely on this set.  This $\psi_r$ may be viewed as a definable predicate on $D_r^{3d}$, and using Proposition \ref{prop:extension}, $\psi_r$ may be extended globally to a definable predicate on $\mathbb{S}_{3d}(\rT_{\tr})$ satisfying $\psi_r \geq 0$ and $\psi_r^{\cM}(\bX,\bA,\bZ) = 0$ if and only if $\bZ$ is a standard free circular family freely independent of $\bX$ and $\bA$.  Let $\phi_{t,\epsilon}$ be the definable predicate given by
	\[
	\phi_{t,\epsilon}^{\cM}(\bX) = \inf_{\bA, \bZ \in D_r^{\cM}} V^{\cM}(\bX + \bA + \bZ) + \frac{1}{2t} \norm{\bA}_{L^2(\cM)^d}^2 +  \frac{1}{\epsilon} \psi_r^{\cM}(\bX,\bA,\bZ)
	\]
	for $\cM \models \rT_{\tr}$.  Then $\phi_{t,\epsilon}$ increases as $\epsilon$ decreases.  Furthermore, one can check that for $r > 2$, for $\bX \in D_r^{\cQ}$,
	\[
	\lim_{\epsilon \to 0^+} \phi_{t,\epsilon}^{\cQ}(\bX) = \inf \{V^{\cQ}(\bX+\bA+t^{1/2}\bZ) + \frac{1}{2t} \norm{\bA}_{L^2(\cQ)^d}: \bA, \bZ \in (D_r^{\cQ})^d, \psi_r(\bX,\bA,\bZ) = 0 \}.
	\]
	This is the same as $\Phi_t V^{\cQ}$ except that now the $\bA$ in the infimum is restricted to an operator norm ball.  After taking the limit as $r \to \infty$, we obtain $\Phi_t V^{\cQ}$.
	
	\begin{question}
		Can one show, under certain hypotheses on $V$, that the limits as $\epsilon \to 0^+$ and $r \to 0^+$ described above occur uniformly on each operator norm ball?
		
		Then can one show that $V_t^{\cQ} = \lim_{k \to \infty} (\Phi_{t/k})^k V^{\cQ}$ exists uniformly on each operator norm ball?
		
		Finally, if so, does this $V_t$ provide a positive answer to Question \ref{q:asymptoticapproximation}?
	\end{question}
	
	One possible flaw in this approach is that our construction only required the free circular family $\bZ$ to be freely independent of $\bA$ and $\bX$, which only describes something about the quantifier-free type of $(\bA,\bX,\bZ)$.  If we can discover a good notion of independence that pertains to the full type rather than only the quantifier-free type as proposed in \S \ref{sec:independence}, then it may be necessary or natural to require independence of $\bZ$ from $(\bA,\bX)$ with respect to their full types, not only the quantifier-free types.  This would be potentially more restrictive on the possible choices of $\bZ$ and thus lead to a larger infimum.  We hope to choose the right setup for the infimum that will accurately describe the large-$n$ limit of $V^{(n)}(\cdot,t)$ as in Question \ref{q:asymptoticapproximation}.
	
	If this can be done, then another goal would be to adapt the work of \cite{Dabrowski2017,JekelEntropy} to show some result that the free microstate entropy $\chi^{\cU}$ of a type agrees with some version of Voiculescu's free non-microstate entropy $\chi^*$.  Of course, a prerequisite for this endeavor would be to sort out what non-microstate entropy for types even is.
	
	\begin{question}
		Is there an analog of free Fisher information $\Phi^*(\mu)$ and free entropy $\chi^*(\mu)$ for a full type rather than a quantifier-free type?
	\end{question}
	
	I do not have much to say about this question now, but I hope that future researchers will explore the many interesting questions that arise from the combination of model theory and free probability.
	
	\bibliographystyle{plain}
	\bibliography{modelmicrostate}

\begin{thebibliography}{10}

\bibitem{Adler2009}
Hans Adler.
\newblock A geometric introduction to forking and thorn-forking.
\newblock {\em Journal of Mathematical Logic}, 09(01):1--20, 2009.

\bibitem{AD2006}
Claire Anantharaman-Delaroche.
\newblock On ergodic theorems for free group actions on noncommutative spaces.
\newblock {\em Probab. Theory Rel. Fields}, 135:520--546, 2006.

\bibitem{AGZ2009}
Greg~W. Anderson, Alice Guionnet, and Ofer Zeitouni.
\newblock {\em An Introduction to Random Matrices}.
\newblock Cambridge Studies in Advanced Mathematics. Cambridge University
  Press, 2009.

\bibitem{AP2017}
Claire Antharaman and Sorin Popa.
\newblock An introduction to {II1} factors.
\newblock Preprint at http://www.math.ucla.edu/~popa/Books/IIun-v10.pdf, 2017.

\bibitem{BYBHU2008}
Ita{\"i} {Ben Yaacov}, Alexander Berenstein, C.~Ward Henson, and Alexander
  Usvyatsov.
\newblock Model theory for metric structures.
\newblock In Z.~Chatzidakis et~al., editor, {\em Model Theory with Applications
  to Algebra and Analysis, Vol. II}, volume 350 of {\em London Mathematical
  Society Lecture Notes Series}, pages 315--427. Cambridge University Press,
  2008.

\bibitem{BY2012}
Itai{\"i} {Ben Yaacov}.
\newblock On theories of random variables.
\newblock {\em Israel J. Math.}, 194(2):957--1012, 2013.

\bibitem{BYU2010}
Itai{\"i} {Ben Yaacov} and Alexander Usvyatsov.
\newblock Continuous first order logic and local stability.
\newblock {\em Transactions of the American Mathematical Society},
  362(10):5213--5259, 10 2010.

\bibitem{BCG2003}
P.~Biane, M.~Capitaine, and A.~Guionnet.
\newblock Large deviation bounds for matrix {B}rownian motion.
\newblock {\em Inventiones Mathematicae}, 152:433--459, 2003.

\bibitem{BS2001}
Philippe Biane and Roland Speicher.
\newblock Free diffusions, free entropy and free fisher information.
\newblock {\em Annales de l'Institut Henri Poincare (B) Probability and
  Statistics}, 37(5):581 -- 606, 2001.

\bibitem{BV2001}
Philippe Biane and Dan-Virgil Voiculescu.
\newblock A free probability analogue of the {W}asserstein metric on the
  trace-state space.
\newblock {\em Geometric and Functional Analysis}, 11:1125--1138, 2001.

\bibitem{BdMPS}
A.~Boutet~de Monvel, L.~Pastur, and M.~Shcherbina.
\newblock On the statistical mechanics approach in the random matrix theory:
  Integrated density of states.
\newblock {\em Journal of Statistical Physics}, 79:585--611, 05 1995.

\bibitem{BDJ2008}
Nathanial~P. Brown, Kenneth~J. Dykema, and Kenley Jung.
\newblock Free entropy dimension in amalgamated free products.
\newblock {\em Proceedings of the London Mathematical Society}, 97(2):339--367,
  2008.

\bibitem{Capraro2010}
Valerio Capraro.
\newblock A survey on {C}onnes' embedding conjecture.
\newblock arXiv:1003.2076, 2010.

\bibitem{Dabrowski2017}
Yoann Dabrowksi.
\newblock A {L}aplace principle for {H}ermitian {B}rownian motion and free
  entropy {I}: the convex functional case.
\newblock arXiv:1604.06420, 2017.

\bibitem{DGS2016}
Yoann Dabrowski, Alice Guionnet, and Dimitri Shlyakhtenko.
\newblock Free transport for convex potentials.
\newblock arXiv:1701.00132, 2016.

\bibitem{EFKV2017}
Christopher~James Eagle, Ilijas Farah, Eberhard Kirchberg, and Alessandro
  Vignati.
\newblock Quantifier elimination in $mathrm{C}^*$-algebras.
\newblock {\em International Mathematics Research Notices},
  2017(24):7580--7606, 11 2016.

\bibitem{EG2012}
Clifton Ealy and Isaac Goldbring.
\newblock Thorn-forking in continuous logic.
\newblock {\em Journal of Symbolic Logic}, 77:63--93, 2012.

\bibitem{FHS2013}
Ilias Farah, Bradd Hart, and David Sherman.
\newblock Model theory of operator algebras {I}: stability.
\newblock {\em Bulletin of the London Mathematical Society}, 45(4):825--838,
  2013.

\bibitem{FHS2014}
Ilias Farah, Bradd Hart, and David Sherman.
\newblock Model theory of operator algebras {II}: model theory.
\newblock {\em Israel Journal of Mathematics}, 201(1):477--505, 2014.

\bibitem{FHS2014b}
Ilias Farah, Bradd Hart, and David Sherman.
\newblock Model theory of operator algebras {III}: elementary equivalence and
  $\mathrm{II}_1$ factors.
\newblock {\em Bulletin of the London Mathematical Society}, 46(3):609--628,
  2014.

\bibitem{FR1975}
Wendell~H. Fleming and Raymond~W. Rishel.
\newblock {\em Deterministic and Stochastic Optimal Control}, volume~1 of {\em
  Applications of Mathematics}.
\newblock Springer, Berlin, Heidelberg, New York, 1975.

\bibitem{GJNS2021}
Wilfrid Gangbo, David Jekel, Kyeongsik Nam, and Dimitri Shlyakhtenko.
\newblock Duality for optimal couplings in free probability.
\newblock {\em Communications in Mathematical Physics}.
\newblock To appear. Preprint, arXiv:2105.12351.

\bibitem{Goldbring2021enforceable}
Isaac Goldbring.
\newblock Enforceable operator algebras.
\newblock {\em Journal of the Institute of Mathematics of Jussieu}, 20:31--63,
  2021.

\bibitem{GHS2013}
Isaac Goldbring, Bradd Hart, and Thomas Sinclair.
\newblock The theory of tracial von neumann algebras does not have a model
  companion.
\newblock {\em Journal of Symbolic Logic}, 78(3):1000--1004, 2013.

\bibitem{GMS2006}
Alice Guionnet and Edouard Maurel-Segala.
\newblock Combinatorial aspects of random matrix models.
\newblock {\em Latin American Journal of Probability and Statistics (ALEA)},
  1:241--279, 2006.

\bibitem{GS2009}
Alice Guionnet and Dimitri Shlyakhtenko.
\newblock Free diffusions and matrix models with strictly convex interaction.
\newblock {\em Geometric and Functional Analysis}, 18(6):1875--1916, 03 2009.

\bibitem{GS2014}
Alice Guionnet and Dimitri Shlyakhtenko.
\newblock Free monotone transport.
\newblock {\em Inventiones Mathematicae}, 197(3):613--661, 09 2014.

\bibitem{Hayes2018}
Ben Hayes.
\newblock 1-bounded entropy and regularity problems in von {N}eumann algebras.
\newblock {\em International Mathematics Research Notices}, 1(3):57--137, 2018.

\bibitem{JekelEntropy}
David Jekel.
\newblock An elementary approach to free entropy theory for convex potentials.
\newblock arXiv:1805.08814. To appear in Analysis \& PDE Journal, 2018.

\bibitem{JekelExpectation}
David Jekel.
\newblock Conditional expectation, entropy, and transport for convex {G}ibbs
  laws in free probability.
\newblock arXiv:1906.10051, 2019.

\bibitem{JekelCoveringEntropy}
David Jekel.
\newblock Covering entropy for types in tracial $\mathrm{W}^*$-algebras.
\newblock preprint, arXiv:2204.02582, 2022.

\bibitem{JLS2022}
David Jekel, Wuchen Li, and Dimitri Shlyakhtenko.
\newblock Tracial non-commutative smooth functions and the free {W}asserstein
  manifold.
\newblock {\em Dissertationes Mathematicae}, 580:1--150, 2022.

\bibitem{JNVWY2020}
Zhengfeng Ji, Anand Natarajan, Thomas Vidick, John Wright, and Henry Yuen.
\newblock {MIP*=RE}.
\newblock arXiv:2001.04383, 2020.

\bibitem{Jing2015}
Naihuan Jing.
\newblock Unitary and orthogonal equivalence of sets of matrices.
\newblock {\em Linear Algebra and its Applications}, 481:235--242, 2015.

\bibitem{Jung2007S1B}
Kenley Jung.
\newblock Strongly $1$-bounded von {N}eumann algebras.
\newblock {\em Geom. Funct. Anal.}, 17(4):1180--1200, 2007.

\bibitem{MS2017}
James~A. Mingo and Roland Speicher.
\newblock {\em Free probability and random matrices}, volume~35 of {\em Fields
  Institute Monographs}.
\newblock Springer-Verlag, New York, 2017.

\bibitem{Sakai1971}
Sh{\^o}ichir{\^o} Sakai.
\newblock {\em $\mathrm{C}^*$-algebras and $\mathrm{W}^*$-algebras}, volume~60
  of {\em Ergebnisse der {M}athematik und ihrer {G}renzgebiete}.
\newblock Springer-Verlag, Berlin Heidelberg, 1971.

\bibitem{Sela2010}
Zlil Sela.
\newblock Diophantine geometry over groups {X}: The elementary theory of free
  products of groups.
\newblock Preprint arXiv:1012.0044, 2010.

\bibitem{Shapiro1991}
Helene Shapiro.
\newblock A survey of canonical forms and invariants for unitary similarity.
\newblock {\em Lin. Algebra Appl.}, 147:101--167, 1991.

\bibitem{ST2022}
Dimitri Shlyakhtenko and Terence Tao.
\newblock Fractional free convolution powers.
\newblock {\em Indiana University Mathematics Journal}, To appear.
\newblock With an appendix by David Jekel.

\bibitem{SongThesis}
Shichang Song.
\newblock {\em Model theory and probability}.
\newblock PhD thesis, University of Illinois at Urbana-Champaign, 2011.

\bibitem{Specht1940}
W.~Specht.
\newblock Zur theorie der matrizen, ii.
\newblock {\em Jahresber. Deutsch. Math.-Verein.}, 50:19--23, 1940.

\bibitem{Szarek1998}
Stanislaw Szarek.
\newblock Metric entropy of homogeneous spaces.
\newblock In {\em Quantum Probability}, volume~43 of {\em Banach Center
  Publications}, pages 395--410. Polish Academy of Science, Warsaw, 1998.

\bibitem{Voiculescu1986}
Dan-Virgil Voiculescu.
\newblock Addition of certain non-commuting random variables.
\newblock {\em Journal of Functional Analysis}, 66(3):323--346, 1986.

\bibitem{Voiculescu1991}
Dan-Virgil Voiculescu.
\newblock Limit laws for random matrices and free products.
\newblock {\em Inventiones mathematicae}, 104(1):201--220, Dec 1991.

\bibitem{VoiculescuFE2}
Dan-Virgil Voiculescu.
\newblock The analogues of entropy and of {F}isher's information in free
  probability, {II}.
\newblock {\em Inventiones Mathematicae}, 118:411--440, 1994.

\bibitem{VoiculescuFE3}
Dan-Virgil Voiculescu.
\newblock The analogues of entropy and of {F}isher's information measure in
  free probability, {III}: Absence of {C}artan subalgebras.
\newblock {\em Geometric and Functional Analysis}, 6:172--199, 1996.

\bibitem{VoiculescuFE4}
Dan-Virgil Voiculescu.
\newblock The analogues of entropy and of fisher's information measure in free
  probability theory, iv: Maximum entropy and freeness.
\newblock In Dan-Virgil Voiculescu, editor, {\em Free Probability Theory},
  volume~12 of {\em Fields Inst. Commun.}, pages 293--302. Amer. Math. Soc.,
  Providence, 1997.

\bibitem{Voiculescu1998}
Dan-Virgil Voiculescu.
\newblock A strengthened asymptotic freeness result for random matrices with
  applications to free entropy.
\newblock {\em International Mathematics Research Notices}, 1998(1):41--63,
  1998.

\bibitem{VoiculescuFE6}
Dan-Virgil Voiculescu.
\newblock The analogues of entropy and of fisher's information measure in free
  probability theory, {VI}: Liberation and mutual free information.
\newblock {\em Advances in Mathematics}, 146:101--166, 1999.

\bibitem{Voiculescu2002}
Dan-Virgil Voiculescu.
\newblock Free entropy.
\newblock {\em Bulletin of the London Mathematical Society}, 34:257--278, 2002.

\bibitem{VDN1992}
Dan-Virgil Voiculescu, Kenneth~J. Dykema, and Alexandru Nica.
\newblock {\em Free Random Variables}, volume~1 of {\em CRM Monograph Series}.
\newblock American Mathematical Society, Providence, RI, 1992.

\bibitem{vN1942}
John {von Neumann}.
\newblock Approximative properties of matrices of high finite order.
\newblock {\em Portugal. Math.}, 3:1--62, 1942.

\bibitem{Wiegmann1961}
N.~Wiegmann.
\newblock Necessary and sufficient conditions for unitary similarity.
\newblock {\em J. Austral. Math. Soc.}, 2:122--126, 1961.

\end{thebibliography}
	
\end{document}